\documentclass[11pt]{article}
\usepackage{geometry}
\usepackage{graphicx}
\usepackage{amsmath}
\usepackage{amssymb}
\usepackage{amsthm}
\usepackage{amsfonts}
\usepackage{enumerate}
\usepackage{bbm}
\usepackage{epsfig}
\usepackage{color}
\usepackage{setspace}
\usepackage{rotating}
\usepackage{subfig}
\usepackage{sidecap}
\usepackage{array}
\usepackage{calc}
\newlength{\mylength}

\numberwithin{equation}{section}

\newtheorem{theorem}{Theorem}[section]
\newtheorem{proposition}[theorem]{Proposition}

\newtheorem{definition}[theorem]{Definition}
\newtheorem{remark}{Remark}
\newtheorem{numobservation}{Numerical observation}

\renewcommand{\vec}[1]{\mbox{\boldmath$#1$}}

\definecolor{orange}{rgb}{1,0.5,0}
\definecolor{dgreen}{rgb}{0,0.5,0}

\newcommand{\Srm}{\textrm{\S}}

\newcommand{\du}{\, \mathrm{d}}

\newcommand{\mat}[1]{\mbox{\boldmath$#1$}}

\newcommand{\param}{\delta}
\newcommand{\delt}{k}
\newcommand{\negNum}{\mathbb{R}_{< 0}}

\newcommand{\imex}{ImEx }
\newcommand{\timen}{n}
\newcommand{\sizeN}{N}

\newcommand{\vOmega}{\varOmega}
\newcommand{\vGamma}{\varGamma}
\newcommand{\rzeta}{\xi}

\def\XXint#1#2#3{{\setbox0=\hbox{$#1{#2#3}{\int}$}
     \vcenter{\hbox{$#2#3$}}\kern-.5\wd0}}

\begin{document}

\title{Unconditional Stability for Multistep ImEx Schemes: Theory} 
\author{Rodolfo R. Rosales
\thanks{Department of Mathematics, MIT, Cambridge, MA 02139, rrr@mit.edu.},  
Benjamin Seibold
\thanks{Department of Mathematics, Temple University, Philadelphia, PA 19122, 
seibold@temple.edu.}, 
David Shirokoff
\thanks{Corresponding author. Department of Mathematical Sciences, NJIT, Newark, NJ 07102, 
\newline \indent \hspace{2mm} david.g.shirokoff@njit.edu.}, 
Dong Zhou
\thanks{Department of Mathematics, Temple University, Philadelphia, PA 19122, 
dzhou@temple.edu.},
}


\maketitle

%
%
\begin{abstract}
This paper presents a new class of high order linear ImEx multistep schemes 
with large regions of unconditional stability. Unconditional stability is a 
desirable property of a time stepping scheme, as it allows the choice of time 
step solely based on accuracy considerations. Of particular interest are 
problems for which both the implicit and explicit parts of the ImEx splitting 
are stiff.  Such splittings can arise, for example, in variable-coefficient 
problems, or the incompressible Navier-Stokes equations. 
To characterize the new ImEx schemes, an unconditional stability 
region is introduced, which plays a role analogous to that of the stability 
region in conventional multistep methods. Moreover, computable quantities 
(such as a numerical range) are provided that guarantee an unconditionally 
stable scheme for a proposed implicit-explicit matrix splitting. The 
new approach is illustrated with several examples. 
Coefficients of the new schemes up to fifth order are provided.
\end{abstract}    



%
\section{Introduction} \label{sec:intro}
When a stiff differential equation is solved via an explicit time
stepping scheme, stability requires time steps that are much smaller than
imposed by accuracy. Implicit schemes can overcome this limitation.
Unfortunately, for many practical problems, a fully implicit treatment may
be structurally difficult or computationally costly.
Implicit-Explicit (ImEx) methods are based on splitting the problem into
two parts, one to be treated implicitly, and the other explicitly. In many
problems, the stiff modes can be conveniently
treated implicitly, while the explicitly treated modes are non-stiff.
Moreover, for many \imex schemes a time step restriction is incurred from
the explicit part, which is generally acceptable if it is non-stiff.

The study presented here is motivated by a different situation, namely the
case where an \imex splitting is conducted for which both parts are stiff
(see \Srm\ref{subsec:motivating_applications} for examples in which this
structure arises naturally). In that case, a time step restriction based on
the explicit part is not acceptable. We therefore aim for more, namely that
the \imex time stepping scheme, for the particular splitting, be
unconditionally stable, i.e., arbitrarily large time steps can be chosen
without losing stability.

At first glance it may sound impossible to achieve unconditional
stability if some parts of the problem are treated explicitly. The reason why
it \emph{is} possible is that the \imex scheme is applied to problems and
splitting choices that possess specific properties, so that the implicit part
can stabilize any growing modes produced by the explicit part.
This concept goes further than one may think: a properly chosen \imex
scheme can stabilize a large explicit part via a
smaller implicit part (see \Srm\ref{Subsec_A_smaller_B}).

While the task outlined above is of interest for any time stepping scheme,
this paper focuses on \imex linear multistep methods (LMMs)
\cite{AscherRuuthWetton1995, Crouzeix1980, Varah1980}. These achieve a high
order of accuracy by using information from previous time steps.  Thus, 
in each time step, they need a single evaluation of the explicit 
part, and a single solve with the implicit part 
(chapter II.3, pg.~171 \cite{HundsdorferVerwer2003}).
Because high order multistep methods tend to possess
less favorable stability properties than Runge-Kutta methods, the
task of achieving unconditional stability is of particular importance.
%
\subsection{Outline of the problem and contributions of this paper}
The problem of interest is a linear system of ordinary differential equations
\begin{equation} \label{Eqn_ODE}
 \vec{u}_t = \mat{L} \vec{u} + \vec{f}(t)
 \quad\text{with}\quad \vec{u}(0) = \vec{u}_0\;,
\end{equation}
where $\vec{u}(t),\vec{u}_0,\vec{f}(t)\in\mathbb{R}^{\sizeN}$ and
$\mat{L}\in\mathbb{R}^{\sizeN\times\sizeN}$ is a matrix. We assume that
$\mat{L}$ is stable, i.e., the homogeneous equation
$\vec{u}_t = \mat{L} \vec{u}$ has solutions that remain bounded for all time
(stability is independent of the forcing $\vec{f}$).
The term $\mat{L} \vec{u}$ in problem \eqref{Eqn_ODE} is now split into
an implicit part ($\mat{A} \vec{u}$) and an explicit part
($\mat{B} \vec{u}$), transforming \eqref{Eqn_ODE} into
\begin{equation} \label{Split}
 \vec{u}_t = \mat{A} \vec{u} + \mat{B} \vec{u} + \vec{f}(t)\;,
\end{equation}
where $\mat{B} \vec{u} = \mat{L} \vec{u} - \mat{A} \vec{u}$.

Of course, the choice of splitting $\mat{L} = \mat{A}+\mat{B}$ is not unique.
One approach is to choose $\mat A$ as the stiff terms in $\mat L$ (i.e., the
terms that would give rise to unnecessarily small time step restrictions if
treated explicitly) and $\mat B$ as the non-stiff terms in $\mat L$. In such
a case, one can guarantee stability for an \imex LMM 
\cite{FrankHundsdorferVerwer1997} by requiring a time step restriction
roughly dictated by an explicit treatment of $\mat B$. However, as
outlined above, here we are concerned with the situation where such a
splitting strategy is not feasible/practical. Hence, we seek for \imex time
stepping schemes that are unconditionally stable when applied to \eqref{Split},
where $\mat{B}$ can involve stiff terms.

Whether a time stepping scheme (of whatever kind) for \eqref{Eqn_ODE} or
\eqref{Split} is stable, depends on both the scheme and the problem's
right-hand side $\mat{L}$. A classical approach (for non-\imex schemes) in
stability analysis (chapter 7, \cite{LeVeque2007}) is to separate stability
into a property of the scheme and another property of the problem's
right-hand side, as follows. For a linear scheme, the region of absolute
stability $S\subset \mathbb{C}$ is the set of all $z = \delt\lambda$,
where $\delt$ is the time step, for which the
numerical solution remains bounded when applied to the test equation
$u_t=\lambda u$.  Similarly, one can define a region of unconditional
stability
$S_\text{u} = \{z\in\mathbb{C}\,:\, \mu z\in S\,\; \forall \mu \ge 0\}$
as the largest cone contained within $S$. If the eigenvalues of
$\mat{L}$ lie in $S_\text{u}$, then the scheme is unconditionally stable.
This concept decouples the scheme stability analysis from the detailed
properties of $\mat{L}$, relying on its spectrum $\sigma(\mat{L})$ only.
Moreover, it allows one to make stability statements
about whole classes of problems. For instance, if $S_\text{u}$ is the cone
$|\theta-\pi| < \alpha$,
where $0 < \alpha < \pi/2$
and $\theta$ is the polar angle (i.e., the scheme is $A(\alpha)$ stable),
then the scheme is unconditionally stable for all problems where $\mat{L}$
is negative definite. Conversely, we know that the same scheme is not
unconditionally stable if $\mat{L}$ is skew-symmetric.

In this paper, an analogous concept is developed for the \imex framework.
This extension is not straightforward, because one now has two right-hand
side operators $\mat{A}$ and $\mat{B}$ that, in general, do not commute and
thus do not share a set of common eigenvectors (see
\Srm\ref{subsec:existing_results} for references to the commutative case).

While the fundamental idea of \emph{stability criteria} for \imex schemes has
been presented before (see \Srm\ref{subsec:existing_results}), here we present
sufficient criteria for unconditional stability that are less restrictive
than prior work. The stability set $\mathcal{D}$ that we introduce depends
only on the coefficients of the \imex schemes, and not the matrices
$\mat{A}$ and $\mat{B}$ in the splitting \eqref{Split}. Moreover, we
devise new high order \imex schemes with \emph{very large} stability regions
that can stabilize splittings of the form \eqref{Split} which are unstable
with current schemes (see \Srm\ref{Sec_Examples}).
%
\subsection{Motivating applications}
\label{subsec:motivating_applications}
While the ideas developed here apply to an abstract ODE system \eqref{Split},
particular interest lies in systems that arise from a method of lines
(chapter 9.2, \cite{LeVeque2007}) discretization (e.g., via finite
differences, finite elements, or spectral) of linear PDE problems. Two
important applications are (let $\nabla_h$ denote the spatial discretization
of $\nabla$ in an appropriate basis with smallest length scale $h$):
\begin{enumerate}[(i)]
 \item Variable coefficient diffusion with
 \[
  \mat L \vec u = \nabla_h \cdot \big( d( x) \; \nabla_h u \big)\;
  \quad\mbox{where}\;\;d(x) > 0\;.
 \]
 Here $\mat{L}$ can be split into a constant coefficient diffusion $\mat{A}$
 and a variable coefficient diffusion $\mat{B}$. Then, fast solvers
 \cite{GreengardRokhlin1987, Trefethen2000} can treat $\mat{A}$ efficiently.
 However, $\mat{B}$ remains stiff, because it scales the same as $\mat{A}$
 (i.e., like $1/h^2$). See \Srm\ref{Subsec_var_coeff} for more details.
 \item Non-local operators, such as the Stokes operator in the linearized
 Navier-Stokes equations, whose discretization either yields a dense matrix
 or requires the addition of extra variables through the introduction of
 Lagrange multipliers,
 \[
	\mat L \vec u = \nu \nabla_h^2 u - \nabla_h p
	\quad\text{and constraint}\quad
	\nabla_h \cdot u = 0\;.
 \]
 A splitting where $\nu \nabla_h^2 u$ is implicit can create a stiff explicit
 $\nabla_h p$, \cite{JohnstonLiu2004,LiuLiuPego2007,ShirokoffRosales2010}.
\end{enumerate}

\smallskip\noindent
The theory in this paper does not directly apply to cases where
$\vec L(\vec u, t)$ is nonlinear or time-dependent, as arising for instance
with discretizations of the Cahn-Hilliard equation \cite{CahnHilliard1958}.
However, the ideas presented below for linear splitting may nevertheless be
useful in stabilizing more general splittings as well.
%
\subsection{Existing results and the new contributions in context}
\label{subsec:existing_results}
The simplest \imex scheme that can achieve unconditional stability
is a first order in time combination of forward and backward Euler steps.
The application to (\ref{Split}) yields
\begin{align}\label{EulerImex}
 \frac{1}{\delt} \big( \vec{u}_{n+1} - \vec{u}_n \big) =
 \mat{A}\vec{u}_{n+1} + \mat{B}\vec{u}_n + \vec{f}(n\delt)\;.
\end{align}
Here $\delt > 0$ is the time step, and $\vec{u}_n$ is the numerical solution
at time $t = n\delt$.

First order in time schemes that achieve unconditional stability originated
with Douglas and Dupont \cite{DouglasDupont1970b}. Other first order
approaches are:
   (i)   iterative schemes for steady state elliptic
         problems \cite{ConcusGolub1973};
   (ii)  variable coefficient diffusion with spectral methods (chapter 9, 
         \cite{GottliebOrszag1977});
   (iii) non-linear convex--concave splittings for the Cahn-Hilliard
         equation \cite{Eyre1998};
   (iv)  non-local explicit terms \cite{AnitescuLaytonPahlevani2004};
   (v)   Hele-Shaw flows \cite{Glasner2003};
   (vi)  phase-field models
\cite{BertozziJuLu2011,ElseyWirth2013,ShengWangDuWangLiuChen2010,Smereka2003};
   (vii) viscosity-pressure splittings in incompressible Navier-Stokes
         \cite{JohnstonLiu2004, LiuLiuPego2007}.   

A disadvantage of first order approaches is that, in addition to the low order,
large error constants have been reported for stable splitting choices in
dissipative equations \cite{ChristliebJonesPromislowWettonWilloughby2014}, as
well as dispersive equations~\cite{Ceniceros2002}.

Better accuracy requires higher order \imex time stepping methods. Two of the
most commonly used approaches, which can be applied to (\ref{Split}), are:
\begin{itemize}\itemsep0pt
 \item\textbf{CN-AB:} Implicit Crank-Nicolson for $\mat{A} \vec{u}$, and
  explicit Adams-Bashforth extrapolation for $\mat{B} \vec{u}$.
 \item\textbf{SBDF (Semi-implicit Backward Differentiation Formula):} 
  Implicit BDF for $\mat{A} \vec{u}$,
  and explicit Adams-Bashforth extrapolation for $\mat{B} \vec{u}$.
\end{itemize}

\smallskip\noindent
  For second order schemes, unconditional stability,
  or at least the absence of a stiff time step restriction, have
  been reported in practice for the semi-implicit
  treatment of the incompressible Navier-Stokes equations
  \cite{KarniadakisIsraeliOrszag1991, KimMoin1985}
  and the Cahn-Hilliard equation
  \cite{BadalassiCenicerosBanerjee2003}.
  Rigorous proofs that guarantee unconditional stability for second order
  \imex schemes such as SBDF or CN-AB have been given for
  convex--concave splittings of gradient flow systems
  \cite{GlasnerOrizaga2016, GuanLowengrubWangWise2014, YanChenWangWise2015},
  a coupled Stokes-Darcy system \cite{LaytonTrenchea2012} and a system
  with an explicit treatment of non-local terms \cite{Trenchea2016}.  See also
  \cite{ElseyWirth2013, XuLiWu2016} for an interpretation of some convex--concave
  splittings as fully implicit schemes with a rescaled time step.

Higher order semi-implicit schemes that guarantee unconditional stability are
not as well studied as their first and second
  order counterparts.  Some third order schemes for the Navier-Stokes equations
  have been found that do not require a diffusion-restricted time step
  \cite{KarniadakisIsraeliOrszag1991, LiuLiuPego2010}.
  General sufficient conditions on $\mat{A}$ and $\mat{B}$ guaranteeing
  unconditional stability for any order of SBDF have been outlined in
  \cite{AkrivisCrouzeixMakridakis1999} and related works
  \cite{Akrivis2013, AkrivisCrouzeixMakridakis1998}.
  Specifically
\cite{Akrivis2013,AkrivisCrouzeixMakridakis1998,AkrivisCrouzeixMakridakis1999}
  assume that $\mat{A}$ is negative definite and also allow for $\mat{B}$ to be
  nonlinear.
  The results in \cite{AkrivisCrouzeixMakridakis1999} applied
  to the case where $\mat{B}$ is a matrix, guarantee unconditional
  stability for an SBDF scheme of
  order $1 \leq r \leq 6$,\footnote{See equations (1.4)--(1.5),
  Theorem 2.1 and also Remark 2.3 in
  \cite{AkrivisCrouzeixMakridakis1999}.}
  if
\begin{align}\label{PreviousWork}
 \| (-\mat{A})^{-1/2} \mat{B}  (-\mat{A})^{-1/2} \|_2 < (2^r-1)^{-1}.
\end{align}
In related work, a set of new second order \imex coefficients was introduced 
in \cite{AkrivisKarakatsani2003}, allowing for a weaker upper bound in 
(\ref{PreviousWork}) --- it can be made arbitrarily close to $1$. The
unconditional stability criteria devised here are more general than previous
bounds such as (\ref{PreviousWork}). Instead of prescribing norm bounds, we
introduce the concept of unconditional stability diagrams for \imex schemes.
The new diagrams generalize the previous work on \imex stability regions
\cite{FrankHundsdorferVerwer1997} (see also \cite{Koto2009}) to (i) the case
of unconditional stability, and (ii) the case where $\mat{A}$ and $\mat{B}$
do not commute.  We then prescribe a set of new \imex coefficients and show
that they can achieve unconditional stability for some problems which 
violate (\ref{PreviousWork}) by orders of magnitude.  
See also chapter IV of \cite{HundsdorferVerwer2003}
for an overview of different splitting methods for ODE integration. 
Other techniques for specific problems are:
   (i) explicit RK schemes with very large stability
   regions for parabolic problems \cite{AbdulleMedovikov2001}, 
   (ii) semi-implicit deferred correction methods \cite{Minion2003}, and
   (iii) semi-implicit schemes when an integration factor (matrix exponential)
   is easily evaluated  \cite{JuZhangZhuDu2014, MilewskiTabak1999}.

This paper is organized as follows. In \Srm\ref{sec:prelim}--\ref{sec:lmm} 
we introduce \imex LMMs, the new criteria for unconditional stability, and
the definition of the unconditional stability region. In
\Srm \ref{sec:new_coefficients} we define new \imex coefficients, characterize
their unconditional stability region, and examine their effect on the
approximation error. Finally, \Srm \ref{Sec_Examples} demonstrates how a
small implicit term may stabilize a large explicit term. It also provides an
example showing how the new coefficients may be used to stabilize splittings
(\ref{Split}) that arise from a variable coefficient diffusion problem.  
We conclude with tables of the new \imex coefficients in 
\Srm \ref{Supp_coeff_tables} so that they may be used by practitioners.
%
\section{Mathematical foundations} \label{sec:prelim}
The purpose of this paper is to examine \imex LMMs (linear multistep methods) 
for splittings of the form (\ref{Split}), where $\mat A \vec u$ is treated 
implicitly, and $\mat B\vec u$ explicitly. Moreover, we are particularly 
interested in the case where both $\mat A$ and $\mat B$ are stiff, 
i.e., each term alone would 
result in severely limited time steps (due to stability) when treated 
explicitly. The goal is to first devise simple sufficient conditions that 
guarantee unconditionally stability of a time stepping scheme when applied 
to (\ref{Split}).  We will then devise new \imex schemes that allow one to 
satisfy the simple unconditional stability conditions, thereby guaranteeing 
an unconditionally stable scheme.

Here we restrict $\mat{A}\/$ to be real, self-adjoint, and negative definite.
Thus: $\mat{A}^T = \mat{A}$, and $\langle \vec x, \mat A \vec x\rangle < 0$
for all $\vec{x} \neq \vec{0}$.  We use the notation 
\[
 \langle \vec{x}\/,\,\vec{y} \rangle = \overline{\vec{x}}^T \vec{y}\/,
 \quad\mbox{and}\quad
 \|\vec{x}\|^2 = \langle \vec{x}\/,\,\vec{x} \rangle \;\; \mbox{for}\;\;
 \vec{x}\/,\,\vec{y} \in \mathbb{C}^{\sizeN}\/.
\]
Note that the restriction above, which is needed for the theoretical
presentation in this paper, is not as limiting as it might seem. A
self-adjoint, negative definite, matrix $\mat{A}$ yields desirable properties
for the efficient solution of linear systems (chapter IV, lecture 38,
\cite{TrefethenBau1997}) with coefficient matrices of the form
$(\mat I - \gamma \mat A)$,  with $\gamma > 0$. The need to solve such linear
systems arises in the time stepping of LMMs, as well as for implicit
Runge-Kutta schemes.
Hence, even if the matrix $\mat{L}$ is not symmetric (e.g.: the discretization
of a dispersive wave problem), it may still be advantageous to take $\mat {A}$
to be symmetric and negative definite, with $\mat{B} := \mat{L}-\mat{A}$.

Let $\vec u_\timen$ be the numerical solution of (\ref{Split}) at time
$t = \timen \delt$, where $\delt$ is the time step, and let
$\vec f_\timen = \vec f(\timen \delt)$. Then a LMM with $s \geq 1\/$ steps
takes the form
\begin{eqnarray} \label{fulltimestepping} 
 \frac{1}{\delt}\sum_{j = 0}^{s} a_j \; \vec u_{\timen+j} =
    \sum_{j = 0}^{s} \Big( c_j \; \mat A \vec u_{\timen+j} +
    b_{j} \; \mat B \vec u_{\timen+j} + b_j  \vec f_{\timen+j} \Big)\/,
\end{eqnarray}
where $(a_j\/,\,b_j\/,\,c_j)\/$, with $0 \leq j \leq s\/$, are the time
stepping coefficients. Here we will assume that $b_s = 0\/$ and
$a_s, c_s \neq 0\/$, so that the method is implicit in $\mat{A}$ and explicit
in $\mat{B}$ --- i.e., it is an \imex time stepping scheme. To accompany
equation (\ref{fulltimestepping}), one must also supply $s\/$ initial vectors
$\vec{u}_0\/,\,\vec{u}_1\/,\,\ldots\,\vec{u}_{s-1}\/$.

We wish to avoid any unnecessarily small time step restriction, and therefore
demand that the scheme (\ref{fulltimestepping}) be \emph{unconditionally
stable}. That is: the solutions to (\ref{fulltimestepping}), with
$\vec{f} = 0\/$, remain bounded for arbitrarily large time steps $\delt > 0$.
This leads to:
%
\begin{definition} (Unconditional stability)
 A scheme (\ref{fulltimestepping}) is unconditionally stable if: when
 $\vec f = 0\/$, there exists a constant $C$ such that
 \[
      \|\vec{u}_{\timen}\| \leq C\,\max_{0 \leq j \leq s-1} \|\vec{u}_j\|\/,
      \quad \textrm{for all}\;\;\timen \geq s, \; \delt > 0
      \quad \textrm{and}\;\; \vec{u}_j \in \mathbb{R}^{\sizeN}\/,
      \;\;\textrm{where}\;\;  0 \leq j \leq s-1\/.
 \]
 Note that $C\/$ may depend on the matrices $\mat{A}$, $\mat{B}\/$, and the
 coefficients $(a_j, b_j, c_j)$, but is independent of the time step $\delt\/$,
 the time index $\timen\/$, and the initial vectors
 $\vec{u}_j$, $0 \leq j \leq s-1$.
\end{definition}
   Unconditional stability is a strong requirement for \imex LMMs, 
   and requires the following caveat: unconditional stability is a
coupled property of \emph{both} the set of \imex coefficients
  $(a_j\/,\,b_j\/,\,c_j)\/$ and the matrices $(\mat{A}\/,\,\mat{B})$.
   Hence:
\begin{itemize}
 \item A given set of coefficients, $(a_j\/,\,b_j\/,\,c_j)\/$, may yield
    unconditional stability for some splittings $(\mat{A}\/,\,\mat{B})\/$,
    and not others.
 \item If the splitting $(\mat{A}\/,\,\mat{B})\/$ arises from the spatial
 discretization of a PDE, then a given set of coefficients
 $(a_j\/,\,b_j\/,\,c_j)\/$ may not yield 
 unconditional stability for all model parameters.   
\end{itemize}
   If the matrices $\mat{A}$ and $\mat{B}$ commute and are diagonalizeable,
   then the stability of (\ref{fulltimestepping}) can be examined by
   using the spectra, $\sigma(\mat{A})$ and $\sigma(\mat{B})$.
   In this paper we \emph{do not} assume that $\mat{A}\/$ and
   $\mat{B}\/$ commute. Hence we cannot rely on the existence of common
   eigenvectors, and must develop a different approach to study the
   stability of (\ref{fulltimestepping}), as follows:
\begin{itemize}
 \item We introduce an unconditional stability region/diagram
    $\mathcal{D}$, which is computable in terms of the scheme coefficients
    $(a_j\/,\,b_j\/,\,c_j)\/$ \emph{only}.
 \item We introduce a region in the complex plane that generalizes the
    notion of spectrum, and depends on the matrix splitting
    $(\mat{A}\/,\,\mat{B})\/$ \emph{only}.
\end{itemize}
   This approach gives a pathway to the design of splittings that
  are guaranteed to be stable for a fixed set of \imex coefficients;
  or to the choosing of \imex coefficients for which a given
  splitting $(\mat{A}\/,\,\mat{B})\/$ yields a stable scheme.
  In fact, in this paper we introduce a new class of \imex
  coefficients that may be chosen to stabilize a given splitting.
  For these new schemes the coefficients yield diagrams that
  permit (arbitrarily) large regions of unconditional stability.

%
%
\section{Stability for linear multistep methods} \label{sec:lmm}
In this section we review the stability criteria
for \imex linear multistep methods (LMMs) defined by equation 
(\ref{fulltimestepping}).  
Following a standard procedure 
(chapter III.4, \cite{HairerNorsettWanner1987}), one may recast the linear recursion relation 
(\ref{fulltimestepping}) with matrix coefficients, as a single vector recursion 
on an $s \times \sizeN$ vector:
\begin{align}\label{MatrixRecurrence}
	\vec V^\timen = \mat W \vec V^{\timen-1}, \quad 
	\textrm{where  }\; \vec V^\timen := \begin{pmatrix}
	\vec u_{\timen+s}, \; \vec u_{\timen+s - 1}, \; \ldots , \; \vec u_{\timen+1}
	\end{pmatrix}^T  \in \mathbb{R}^{s\sizeN }.
\end{align}
Here $\mat W$ is a matrix with block structure:
\begin{align}\label{LargeMatrixIteration}
	\mat W =  \begin{pmatrix}
		a_s - \delt c_s \mat A & 0 & 0 & \ldots & 0 \\
		 0 & \mat I &   0 & \ldots & 0 \\
 		 0 & 0 & \mat I &   \ldots & 0 \\
  		 \vdots &  &  & \ddots & 0 \\
   		 0 &  0 & 0 & \ldots  & \mat I 
	\end{pmatrix}^{-1}
	\begin{pmatrix}
		 \mat C_{s-1} & \mat C_{s-2} & \ldots & \mat C_1 & \mat C_0 \\
		 \mat I &  0 & \ldots & 0 & 0  \\
 		 0 &  \mat I &   \ldots & 0 & 0 \\
 		 \vdots &  & \ddots & 0 & 0 \\
   		 0 &  0 &   \ldots  & \mat I & 0 
	\end{pmatrix},	
\end{align}
where $\mat I$ is the $\sizeN \times \sizeN$ identity matrix, and
\[ \mat C_j = \delt c_j \mat A + \delt b_j \mat B - a_j \mat I, 
\quad 0 \leq j \leq s - 1. \]
%
%
%
Recall (chapter III.4, \cite{HairerNorsettWanner1987}, 
chapter V.1, \cite{WannerHairer1991}) 
that equation 
(\ref{MatrixRecurrence}), and hence the scheme (\ref{fulltimestepping}), 
is stable for a given $\delt$ if every semisimple\footnote{An eigenvalue 
$\zeta$ is semisimple if its algebraic multiplicity equals its geometric 
multiplicity.} eigenvalue of $\mat W$ satisfies $|\zeta| \leq 1$, and every 
non-semisimple eigenvalue satisfies $|\zeta| < 1$.  
In the case when $\mat A$ and $\mat B$ do not commute,
    the eigenvalues of $\mat W$ depend on both:
    (i) the matrices $\mat A$ and $\mat B$, and
    (ii) the \imex time stepping coefficients $(a_j, b_j, c_j)$.
    Hence the eigenvalues of $\mat W$ do not provide a way to
    characterize unconditional stability in a way analogous
    to that for non-\imex schemes: Some set depending on $\mat L$
    only (e.g., its spectrum) must be included within some set that is
    defined by the scheme coefficients only (the unconditional
    stability set). In what follows we devise a strategy to get
    around this problem, so that conditions that guarantee
    unconditional stability of \imex schemes can be formulated
    in a language similar to the one for non-\imex schemes, or
    for \imex schemes with commutative splits (though the set
    depending on $\mat L = \mat A + \mat B$ is no longer a
    spectrum).

Let $\vec V^* \neq \vec 0$ be an eigenvector of $\mat{W}$
with eigenvalue $\zeta$. Then, due to the structure of the bottom $(s-1)$ 
matrix blocks in $\mat W$, $\vec V^* \in \mathbb{C}^{ s \sizeN }$
has the form 
\begin{align} \label{Big_Eigenvector}
	\vec V^* = \begin{pmatrix}
	\zeta^{s-1} \vec v, \; \zeta^{s-2} \vec v, \; \ldots, 
	\; \zeta \vec v, \; \vec v
	\end{pmatrix}^T, 
	\quad \textrm{where } \vec v \neq \vec 0, \; \vec v \in \mathbb{C}^\sizeN.
\end{align}
The characteristic equation for $\mat W$ can be rewritten in the form
\[ 
	\det (\mat W - \zeta \mat I) = 0 \quad \Longleftrightarrow \quad 
	\det\Big( \frac{1}{\delt} a(\zeta) \; \mat I -  c(\zeta) \; \mat A 
	- b(\zeta) \; \mat B \Big) = 0. 
\]
%
%
%
where
\[a(z) = \sum_{j = 0}^s a_j z^j, \quad 
b(z) = \sum_{j = 0}^{s-1} b_j z^j, \quad
c(z) = \sum_{j = 0}^{s} c_j z^j\]
are polynomials determined by the time stepping coefficients
$(a_j, b_j, c_j)$, $0 \leq j \leq s$.

Hence if $\zeta$ is an eigenvalue of $\mat W$ (with possible algebraic 
multiplicity greater than one), then there always exists at least 
one $\vec V^*$ from
(\ref{Big_Eigenvector}) with $\vec v$ satisfying:
\begin{align} \label{eqn_NLeigenvalue0} 
	\mat T(\zeta) \vec v = 0, \quad \textrm{where }\; 
	\mat T(z) := \Big( \frac{1}{\delt} a(z) \; \mat{I} -  c(z) \; \mat A 
	-  b(z) \; \mat B\Big).
\end{align}
Note that one may also arrive at equation (\ref{eqn_NLeigenvalue0}) by 
substituting the normal mode ansatz $\vec u_\timen = \zeta^\timen \vec v$ 
into the general linear \imex time-stepping scheme (\ref{fulltimestepping}).  

Clearly if $\mat T(z)$ is singular for $|z| < 1$, then any eigenvector 
$\vec V^*$ of $\mat W$ has every eigenvalue (regardless of algebraic 
multiplicity) $|\zeta| < 1$. Conditions on $\mat T(z)$ for 
the stability of (\ref{fulltimestepping}) can then be stated as follows:

\begin{proposition}\label{Stability2}
	If, for a fixed $\delt > 0$, the matrix $\mat T(z)$ is non-singular  
for all $|z| \geq 1$, i.e., $\det \mat T(z) \neq 0$ for $|z| \geq 1$, then the 
scheme (\ref{fulltimestepping}) is stable.
\end{proposition}

\begin{remark}
	Proposition~\ref{Stability2} is not sharp as we have omitted
	the possibility for $\det \mat{T}(\zeta) = 0$ with $|\zeta| = 1$.  
\end{remark}
%
\subsection{The stability region $\mathcal{D}$}\label{Subsec_stabilityregion}
The $\sizeN \times \sizeN$ matrix equation (\ref{eqn_NLeigenvalue0}) still
couples together both the matrices $(\mat A\/,\,\mat B)$ to the scheme
coefficients $(a_j, b_j, c_j)$.  To decouple the time stepping stability
analysis (i.e., the time stepping coefficients) from the details of the ODE
being solved (i.e., the matrices $\mat{A}\/$ and  $\mat{B}\/$), we multiply
(\ref{eqn_NLeigenvalue0}) by the positive definite matrix $(-\mat A)^{p-1}$,
where $p \in \mathbb{R}$ --- $p$ real is all that is needed for the analysis
below to hold. 
In the examples in \S\ref{Sec_Examples},  we will eventually focus on
$p = 1\/$, as it is observed that this choice provides sufficient estimates
for the test problems we consider.  The stability theory obtained with other
values of $p\neq 1$ may still however be of use in the numerical treatment
of other PDEs, distinct from those in \S\ref{Sec_Examples}.
Thus:   
\[
 \frac{1}{\delt} a(\zeta) (-\mat A)^{p-1} \vec v =
 -c(\zeta) (-\mat A)^p \vec v + b(\zeta) (-\mat A)^{p-1} \mat B \vec v\/.
\]
Dotting through with $\vec v$ and setting
\begin{align} \label{Range} 
 y = -\delt \frac{\langle \vec v, (-\mat A)^{p} \vec v \rangle}{\langle
      \vec v, (- \mat A)^{p-1} \vec v \rangle }\/,  \quad \quad
 \mu = \frac{\langle \vec v, (-\mat A)^{p-1} \mat B \vec v\rangle }{\langle
      \vec v, (-\mat A)^{p} \vec v\rangle}\/, 
\end{align}
we obtain the equation\footnote{The polynomial (\ref{Model_Eq}) with
   $\mu = 0$ was used in convergence proofs in
\cite{Akrivis2013,AkrivisCrouzeixMakridakis1998,%
AkrivisCrouzeixMakridakis1999, Crouzeix1980}.
   A similar equation $a(\zeta) = \lambda c(\zeta) + \mu b(\zeta)$ was
   obtained in \cite{AscherRuuthWetton1995} for commuting matrices
   $\mat{A}\/$ and $\mat{B}\/$, and studied as a model equation
   for stability in \cite{FrankHundsdorferVerwer1997} to estimate explicit
   time step $\delt$ restrictions.
   However, note that here we \emph{do not} assume that $\mat{A}\/$
      and $\mat{B}\/$ commute.}
\begin{align} \label{Model_Eq} 
 a(\zeta) = y\; c(\zeta)  - y \mu \; b(\zeta).
\end{align}
   Since $(-\mat A)$ is positive definite, $y$ may take
   any value $y < 0$ as $\delt$ varies over the allowable values
   $\delt > 0$, with any $\vec{v} \neq 0\/$ fixed.  The following 
   definition is then justified by the result in Proposition~\ref{Stability2}.
\begin{definition} \label{Def_Stability}
 (Stability) The polynomial equation (\ref{Model_Eq}) is stable,
 for a given $y < 0$ and $\mu \in \mathbb{C}$, if every solution
 satisfies $|\zeta| < 1$.
\end{definition}
\begin{definition} \label{Def_regionUstab}
(Unconditional stability region)
We define the region of unconditional stability $\mathcal{D}$, as the values
of $\mu$ so that  (\ref{Model_Eq}) is stable for all
$y \in \negNum \cup \{-\infty\}$.  Formally, define the following sets
\begin{align}\nonumber
 \mathcal{D}_y &:= \{ \mu \in \mathbb{C} : (\ref{Model_Eq})
    \textrm{ is stable for a fixed } y \in \negNum \}\/, \\ \nonumber
 \mathcal{D}_{-\infty} &:= \{ \mu \in \mathbb{C} : c(\zeta)-\mu b(\zeta)
    \textrm{ has stable roots}\}\/,
\end{align}
   \[ \mathcal{D} = \bigcap_{y \in \negNum\cup\{-\infty\}} \mathcal{D}_y\/.\]
\end{definition}  
Note that $\mathcal{D}$ depends only on the \imex time-stepping coefficients
and not on the matrices $\mat A, \mat B$. Moreover,
$\mathcal{D}$ may be empty for some schemes.  
%
\subsection{Numerical range and sufficient condition for unconditional
   stability}\label{SubSection_StabilityRegions} 
The exact realizable values of $\mu$ defined by the expression in
(\ref{Range}), for a given splitting $(\mat{A}, \mat{B})$ and time stepping 
coefficients, are determined through
the normal modes $\vec v$.  To find these values of $\mu$, 
which form a discrete, finite set in the complex plane, one must solve the 
fully coupled eigenvalue problem given by (\ref{eqn_NLeigenvalue0}).
A better and simpler approach is to overestimate the region in the complex plane 
where the values of $\mu$ reside.  Specifically, the values of $\mu$ belong 
to the complex set obtained by allowing $\vec v$ to vary over all possible 
vectors. That is:
\[ \mu \in W_p\/, \quad \textrm{where }
   W_p := \Big\{ \langle \vec v, (-\mat A)^{p-1} \mat B \vec v\rangle :
          \langle \vec v, (-\mat A)^{p} \vec v\rangle = 1 \Big\}\/. \]
Using a straightforward change of variables
$\vec v = (-\mat A)^{\frac{p}{2}} \vec x $, and the fact that $\mat A$ is
symmetric, the set $W_p$ can be identified as:
\[ W_p = W\Big( (-\mat A)^{\frac{p}{2} - 1} \mat \; \mat B \;
    (-\mat A)^{-\frac{p}{2}} \Big)\/. \]
Here $W(\mat X)$ denotes the \emph{numerical range} (also known as 
the \emph{field of values}) of a matrix
$\mat X \in \mathbb{C}^{\sizeN \times \sizeN}$ and is defined by
\begin{align} \label{Def_NumericalRange} 
 W(\mat X) := \{ \langle \vec x, \mat X \vec x \rangle :
            \| \vec x \| = 1, \vec x \in \mathbb{C}^{\sizeN}\}\/. 
\end{align}
See \Srm\ref{Supp_properties_W} for a list of standard properties for $W(\mat{X})$. 
One then arrives at a sufficient condition for unconditional stability for 
equation (\ref{fulltimestepping}):
%
\begin{theorem} [Sufficient condition for unconditional stability]
 Suppose that a matrix splitting $(\mat{A}\/,\,\mat{B})$ has sets $W_p$
 for $p \in \mathbb{R}$
 and that the LMM time stepping coefficients $(a_j, b_j, c_j)$ have an
 unconditional stability region $\mathcal{D}$. Then, if there exists a
 $p\in \mathbb{R}$ such that $W_p \subseteq \mathcal{D}$, the scheme
 (\ref{eqn_NLeigenvalue0}) is unconditionally stable.
\end{theorem}
%
\begin{remark}
 Different values of $p$ may modify the size of $W_p$ in the complex plane.
 The sufficient condition for unconditional stability only requires one
 value of $p$ to satisfy $W_p \subseteq \mathcal{D}$, (even if other
 values of $p$ violate $W_p \subseteq \mathcal{D}$).
\end{remark}
%

%
\section{New \imex coefficients}\label{sec:new_coefficients}
%
%
%
\subsection{Definition of the new \imex coefficients}\label{SubSection_NewImEx}
The property of unconditional stability is not limited to LMMs, however
here we focus on LMMs only.
Any \imex LMM where the number of steps equals the order of the scheme 
$s = r$, is completely defined by specifying the polynomial $c(z)$.  
For instance given $s = r$ and a fixed $c(z)$, the order conditions define the 
polynomials $a(z),  b(z)$ and subsequently all time stepping coefficients.  
Therefore, the roots\footnote{Since rescaling the \imex coefficients 
$(a_j, b_j, c_j)$ by an overall constant does not modify a scheme, one can 
take without loss of generality the leading coefficient of $c(z)$ to be $1$.} 
of the polynomial $c(z)$ can also be used to uniquely define any \imex scheme 
when $r = s$.  
The new \imex coefficients proposed in this paper will be prescribed by the
location 
of the roots of $c(z)$.  In particular, regions of unconditional stability 
$\mathcal{D}$ depend strongly on the location of the roots of $c(z)$, and 
become large when the roots of $c(z)$ become close to $1$ (see also
\Srm\ref{Supp_roots}). 
Although there are many options for parameterizing how the roots of $c(z)$
approach $1$, we choose the simplest approach and \emph{lock} all the roots
together.
%
\begin{definition} \label{NewImExCoeff} (New ImEx Coefficients) 
For orders $1 \leq r \leq 5$, and $0 < \param \leq 1$, the new ImEx 
coefficients $(a_j, b_j, c_j)$, for $0 \leq j \leq r$, are defined as the 
following polynomial coefficients:
\begin{align} \label{NewImex_c}
	 &\textrm{(Implicit coeff.)} \quad 
	 c(z) = (z - 1 + \param)^r, \\ \label{NewImex_b}
	 &\textrm{(Explicit coeff.)} \quad 
	 b(z) = (z - 1 + \param)^r - (z- 1)^r, 
\end{align}
The time stepping polynomial $a(z)$ is concisely written as the $r$-th order 
Taylor polynomial centered at $z = 1$ of the generating function $f(z)$,
\begin{align}\label{NewImex_a}
	\textrm{(Derivative coeff.)} \hspace{2mm} 
	a(z) = \sum_{j = 1}^r \frac{f^{(j)}(1)}{j!} (z- 1)^j, \hspace{2mm} 
	f(z) = (\ln z) (z - 1 + \param)^r.
\end{align}
\end{definition}
%
Note that once $c(z)$ is chosen, $a(z)$ and $b(z)$ are uniquely determined. 
For more on this, see Proposition~\ref{Order_conditions} below.
In \Srm \ref{Supp_coeff_tables} we report the \imex coefficients
$(a_j, b_j, c_j)$ as polynomial functions of $\param$. In the case when
$\param = 1$, the new coefficients recover the combined SBDF -- backward
differentiation formula (for the implicit $c(z)$) and  Adams-Bashforth (for
the explicit $b(z)$). For $\param < 1$ the roots of $c(z)$ shift towards
$z = 1$.
The new coefficients bear some similarity to the one-parameter, high order,
multistep schemes with large absolute stability regions studied in
\cite{JeltschNevanlinna1981, JeltschNevanlinna1982}.  We stress, however,
that our use of the \imex coefficients in Definition~\ref{NewImExCoeff} is 
of a fundamentally different nature than the non-ImEx investigation found 
in \cite{JeltschNevanlinna1981, JeltschNevanlinna1982}.
Specifically, we select a $\param$ value that is strictly bounded away from
$0$, based on the \imex splitting $(\mat{A}\/,\,\mat{B})$ of $\mat{L}$, which
yields an unconditionally stable method.
Moreover, a subsequent error investigation indicates that $\param$ should be
selected as large as possible, while still maintaining unconditional
stability.
\begin{remark}
 We limit Definition~\ref{NewImExCoeff} to orders $r \leq 5$. SBDF schemes
 ($\param = 1$) with orders $r \geq 7$ are not zero stable. Furthermore, the
 characterization of $\mathcal{D}$ for $r = 6$ is not contained within the
 theory presented in the following subsection. Specifically, the Numerical
 Observation~\ref{NumObservation} (see \Srm\ref{Subsec_stab_region}) fails
 for $r = 6$ and $\param = 1$.
\end{remark}
%
\begin{proposition}\label{Order_conditions}
 For all $0 < \param \leq 1$ and orders $1 \leq r \leq 5$, the ImEx
 coefficients in Definition~\ref{NewImExCoeff} are zero-stable and satisfy
 the $r$-th order conditions. 
\end{proposition}
%
See \Srm\ref{Supp_Prop_Order_conditions} for the verification of 
Proposition~\ref{Order_conditions}.
%
\subsection{Stability regions for the new \imex coefficients}
\label{Subsec_stab_region}
\hspace*{0em}
The region $\mathcal{D}$ was introduced in the context of the sufficient
conditions for unconditional stability. As we will see later (in
\Srm\ref{SubSection_NecessaryConditions}) it also plays a role in the
necessary conditions for unconditional stability. In this section we
characterize the geometry of $\mathcal{D}$ for the \imex coefficients in
Definition~\ref{NewImExCoeff}. This geometry (i.e., the size and shape of
$\mathcal{D}$ in the complex plane) fixes classes of splittings
$(\mat A, \mat B)$ that are, or are not, unconditionally stable.  Roughly
speaking, for small $\param$ values, $\mathcal{D}$ approaches the union of
(i)  a large circle with radius $\sim (r \delta)^{-1}$ and center
     $\sim -(r\delta)^{-1}\/$, and
(ii) a triangular region, symmetric relative to the real axis, with its tip
     on the positive real axis. See Figure~\ref{NewImExStabilityRegions}.

We first focus on describing the set $\mathcal{D}_{-\infty}$, since by
definition the unconditional stability region $\mathcal{D}$ is a subset of
$\mathcal{D}_{-\infty}$, i.e. $\mathcal{D} \subseteq \mathcal{D}_{-\infty}$. 
However, we show later that this subset inclusion is in fact an equality, so
that $\mathcal{D} = \mathcal{D}_{-\infty}$.  Thus one should keep in mind that
statements characterizing $\mathcal{D}_{-\infty}$ are statements about
$\mathcal{D}$.  The main result regarding $\mathcal{D}_{-\infty}$ is
summarized by the following theorem.
\begin{theorem}\label{CharacterizationUStab} (The set $\mathcal{D}_{-\infty}$)
The set $\mathcal{D}_{-\infty}$ is simply connected, contains the origin
$\mu = 0$, and has a boundary parameterized by the curve
\begin{align}\label{ExactBoundary}
 &\partial \mathcal{D}_{-\infty} =
   \Big\{ \frac{(z-1 + \param)^r}{(z-1 + \param)^r-(z-1)^r} : |z| = 1, \;
   \mathrm{arg}\; z_0 \leq \mathrm{arg} \; z \leq \; 
   2\pi - \mathrm{arg}\; z_0 \Big\},\\  \nonumber
 &\textrm{where: } z_0 = 1, \hspace*{13.4em} \textrm{for order } r = 1,
  \textrm{and}\\ \label{ExactStart}
 &\phantom{\textrm{where: }} 
   z_0 = \frac{2-\param - 2(1-\param)\cos(\pi/r) 
   e^{\imath \pi/r}}{2-\param - 2\cos(\pi/r) e^{\imath \pi/r}}, 
   \quad \textrm{for orders } 2 \leq r \leq 5.
\end{align}
Moreover, let $m_r$ (resp. $m_l$) be the right-most (resp. left-most) point
of $\partial\mathcal{D}_{-\infty}$. Then $m_r$ (resp. $m_l$) is obtained at the
parameter value $z = z_0$ (resp. $z = -1$). Thus
\[ \begin{array}{lll}
 \mbox{for } r = 1, & m_l = \frac{-(2-\param)}{\param} &
    \mbox{and}\quad m_r = 1, \\
 \mbox{for } 2 \leq r \leq 5, & 
    m_l = \frac{-(2-\param)^r}{2^r-(2-\param)^r} &
    \mbox{and}\quad 
    m_r = \frac{(2-\param)^r}{(2-\param)^r + 2^r \cos^r(\pi/r)}.
\end{array} \]
Note that both $m_l$ and $m_r$ are on the real axis.
\end{theorem}
%
\begin{proof}
 For $r = 1$ the proof is straightforward as $\partial \mathcal{D}_{-\infty}$
 is a circle for all $0 < \param \leq 1$. The idea for the proof when
 $2 \leq r \leq 5$ is to show that 
 $\mathcal{D}_{-\infty} = \varphi^{-1}(\mathcal{T})$ is the preimage of a set
 $\mathcal{T}$ (which is a triangle for $r \geq 3$ and a strip for $r = 2$)	
 under the mapping of a complex function $\varphi(z)$.  
 The results in the theorem then follow from basic calculus arguments, and
 the conformal properties of complex mappings. 
	
 The set $\mathcal{D}_{-\infty}$ consists of the values $\mu \in \mathbb{C}$ 
 that ensure that the solutions $z\in \mathbb{C}$ to the following polynomial
 equation are stable (see Definition~\ref{Def_Stability}):
 \begin{align}\label{Simplified_Polynomial}
  c(z) - \mu b(z) = 0 \quad \Longleftrightarrow \quad 
  (z - 1 + \param)^r - \mu \Big( (z-1+\param)^r - (z - 1)^r \Big) = 0.
 \end{align}
 Note that $0 \in \mathcal{D}_\infty$, since $c(z)$ has a single root:
 $z=1-\delta$ (with multiplicity $r$). 
 As a direct result of the simple structure of the polynomials $c(z)$ and
 $b(z)$, the equation \eqref{Simplified_Polynomial} can be solved explicitly to write the
 solutions $z_j(\mu)$ (for $0 \leq j \leq r-1$) in terms of $\mu$ as:
 \begin{align}\label{z_roots}
  z_j(\mu) = 1 + \frac{\param}{ \rzeta_j \varphi(\mu) - 1}, 
  \quad \textrm{where}
  \quad \rzeta_j = e^{\frac{\imath 2\pi j}{r}}, \quad 0 \leq j \leq r-1.
 \end{align}
 Here $\varphi(\mu)$ is the complex-valued function defined using a branch cut
 taken along the negative real axis:
 \begin{align}
  \varphi(\mu) := \Big( \frac{\mu}{\mu - 1} \Big)^{1/r}, \quad \textrm{where } 
  \big( R e^{\imath \theta} \big)^{1/r} := R^{1/r} e^{\frac{\imath \theta}{r}}, 
  \quad (-\pi < \theta \leq \pi, \; R \geq 0).
 \end{align}
 Observe that $\varphi(\mu)$ is the composition of a M\"{o}bius transformation
 (which has the property that it is a one-to-one mapping of the compactified
 complex plane to itself, with the identification that the point
 $1 \rightarrow \infty$ and $\infty \rightarrow 1$), with the $r$-th root
 function. Hence, $\varphi(\mu) : \mathbb{C} \rightarrow \mathcal{W}$ where
 \[
  \mathcal{W} = \Big\{ z \in \mathbb{C} : z = 0, \textrm{ or }
   -\frac{\pi}{r} < \mathrm{arg}z \leq \frac{\pi}{r} \Big\}. 
 \]
 Next, we note that the modulus constraints $|z_j| < 1$ restrict the range
 of $\varphi(\mu)$ to the intersection of $r$ half-planes given by the
 following inequalities:
 \begin{align} \label{Inequality}
  \Big|1 + \frac{\param}{\rzeta_j \varphi(\mu) - 1}\Big| < 1
    \quad \Longleftrightarrow  \quad
  \mathrm{Re}(\rzeta_j \varphi(\mu) ) < 1 - \frac{\param}{2}.
 \end{align}
 Clearly, the inequality \eqref{Inequality} must be satisfied by all roots
 $0 \leq j \leq r- 1$.  Satisfying the inequality \eqref{Inequality} for
 $j = 0$, however, will automatically guarantee the satisfaction of the
 remaining $1 \leq j \leq r-1$ inequalities. To make this correspondence
 precise, we introduce the set $\mathcal{T}$ (which is a triangle for
 $r \geq 3$, a strip for $r = 2$ and half-plane	for $r = 1$), obtained by
 taking the intersection of $\mathcal{W}$ with the $j = 0$ inequality in
 \eqref{Inequality},
 \begin{align}
  \mathcal{T} = \Big\{ z \in \mathcal{W} : 
  \mathrm{Re}(z) < 1 - \frac{\param}{2} \Big\}
 \end{align}
 Figure \ref{calT_Gpos} (left) shows the triangle $\mathcal{T}$, as well as
 the rotated triangles $\rzeta_j \mathcal{T}$, for $r = 3$. A simple use of
 inequalities,\footnote{Specifically: if $w = R e^{\imath \theta}$ with
    $R < (1 - \param/2) \sec(\theta)$ so that
    $\textrm{Re}(w) < 1 - \param/2$, then 
    $\mathrm{Re}(\rzeta_j w) = R \cos(\theta + 2\pi j /r) < (1 -\param/2)$,
    since $\cos(\theta + 2\pi j /r ) \leq \cos(\theta)$ for
    $|\theta| \geq \pi/r$.}
 whose geometric interpretation is highlighted in Figure~\ref{calT_Gpos}
 (left), shows that if $w \in \mathcal{T}$, then
 $\textrm{Re}(\rzeta_j w) < 1 - \frac{\param}{2}$. Hence, if
 $\varphi(\mu) \in \mathcal{T}$, then $\mu \in \mathcal{D}_{-\infty}$. That is:
 $\mathcal{D}_{-\infty} = \varphi^{-1}(\mathcal{T})$ is the preimage of
 $\mathcal{T}$ under the mapping $\varphi(z)$.
 The sets $\mathcal{D}_{-\infty}$, for the parameter value $\param = 1$
      and orders $1\leq r \leq 3$, are shown in
      Figure~\ref{BoundaryLocus}.

 The properties of $\mathcal{D}_{-\infty}$ now follow by observing that the set
 $\varphi^{-1}(\mathcal{T})=M(\mathcal{T}^r)$ is the image under the M\"{o}bius
 transformation $M(z) = z/(z-1)$ of the set $\mathcal{T}^r$, where
 $\mathcal{T}^r = \{ z^r : z \in \mathcal{T}\}$ is the $r$-th power of
 $\mathcal{T}$. 
 %
 %
 Below, we will use the following simple properties (chapter 3,
      \cite{Ahlfors1979}) of the M\"{o}bius transformation $M(z)$ in the
      Riemann sphere, with the understanding that
      $M(1) = \infty$ and $M(\infty) = 1$.
 \begin{enumerate}
  \item[M1.] The real axis is invariant under $M(z)$.
  \item[M2.] If $D$ is a closed disk centered on the real axis, with
    $\textrm{Re}(D) < 1$, then $M(D)$ is also a disk centered on the real axis
    with $\mathrm{Re}(M(D) ) < 1$.
  \item[M3.] The half-plane $\mathrm{Re}(z) \leq 1$ is invariant under $M(z)$.
    Any half-plane $\mathrm{Re}(z) \leq\alpha < 1$ ($\alpha\in\mathbb{R}$) is
    mapped to a disk $D$ with center on the real axis and $\mathrm{Re}(D)< 1$.
  \item[M4.] $M$ is a continuous map on the Riemann sphere, and
    $M=M^{-1}$.
 \end{enumerate}
 Note that $\mathcal{D}_{-\infty} =\varphi^{-1}(\mathcal{T}) =M(\mathcal{T}^r)$
 is simply connected, since $M$ is continuous and $\mathcal{T}^r$ is simply
 connected.
 To obtain the formula for the boundary $\partial\mathcal{D}_{-\infty}$, we
 observe that the line segments $\theta = \pm \pi/r$ on $\partial \mathcal{T}$ are
 mapped (under the $r$-th power, $\mathcal{T} \rightarrow \mathcal{T}^r$) to
 to identical line segments
 along the negative real axis. Further, these segments are contained
 in the interior of $\mathcal{T}^r$.
 Hence the boundary of $\mathcal{T}^r$, and subsequently the boundary
 $\partial\mathcal{D}_{-\infty}=\varphi^{-1}( \ell_r)$,  is the preimage of the
 line or line segment which is the right side of $\mathcal{T}$. Here $\ell_r$ is defined
 as:
 \begin{align*}
   &\textrm{For } r = 2: \quad \ell_2 =
      \Big\{ \mathrm{Re}(z) = 1 - \param/2 \Big\}, \\
   &\textrm{For } r \geq 3: \quad \ell_r =
      \Big\{(1-\tau)\bar{z}_{e} + \tau z_{e} : \; 0 < \tau \leq 1, \;
      z_{e} = (1-\param/2) \sec(\pi/r) e^{\imath \pi/r} \Big\}.
 \end{align*}
 Substituting $\varphi(\ell_r)$ into \eqref{z_roots} for $j = 0$, yields the
 root locus parameterization of the boundary $\partial\mathcal{D}_{-\infty}$
 stated in the theorem. The value $z_0$ in the theorem statement corresponds
 to substituting the endpoint $\bar{z}_e$ of $\ell_r$ for 
 $\mu =\varphi^{-1}(z_e)$ into
 the formula for $z_0(\mu)$ in \eqref{z_roots}
 \[  
  z_0 = \frac{2-\param - 2(1-\param)\cos(\pi/r) e^{\imath \pi/r}}{
              2-\param - 2\cos(\pi/r) e^{\imath \pi/r}}, 
        \quad \textrm{for } 2 \leq r \leq 5. 
 \]
 In the above expression, and for our subsequent calculations below, it is
 understood that for $r = 2$, $z_e$ is taken as
 $z_e = (1 - \param/2) + \imath \infty$. 	

	Lastly, to verify the result for the right and left-most 
	endpoints of $\partial \mathcal{D}_{-\infty}$, our goal is to 
	show that $\mathcal{T}^r$ is contained in a suitably chosen
	disk ($r\geq 3$) or half-plane ($r = 2$) and to use 
	properties (M1--M3).
	First denote the midpoint 
	of $\ell_r$ as $z_m = (1-\param/2)$.	
	Then the only values of $\partial\mathcal{T}^r$
	along the real axis are $z_m^r$ and $z_e^r$. Hence by
	property (M1),	$m_l := \varphi^{-1}(z_m)$, and 
	$m_r := \varphi^{-1}(z_e)$ are the only values of 
	$\partial \mathcal{D}_{-\infty}$ along the real axis. 
	To show	that $m_l$ and $m_r$ are the left-most and right-most
	points of $\mathcal{D}_{-\infty}$ for $r = 2$, note that
	$\mathcal{T}^r$ is contained within the half-plane 
	$\mathrm{Re}(z) \leq z_m^2$, and contains the point along the 
	negative real axis $-\infty \in \mathcal{T}^r$. Hence, by 
	property (M3), $m_r = 1$ is the rightmost point, and by combining
	property (M1) and (M3), $m_l$ is the left-most point 
	of $\partial \mathcal{D}_{-\infty}$.
 For $r \geq 3$, it is sufficient to show that $\mathcal{T}^r$ is contained
 in the disk $D = \{z \in \mathbb{C} : |z-z_d| \leq R_d\}$ centered at 
 $z_d =\frac{1}{2}(z_e^r+z_m^r)$ with a radius $R_d =\frac{1}{2}(z_m^r-z_e^r)$,
 and right and left endpoints $z_m^r$ and $z_e^r$, respectively. This is
 because properties (M1) and (M2) imply that $m_r=M(z_m^r)$ and $m_l=M(z_e^r)$
 will be preserved as the right and left-most points of
 $\partial \mathcal{D}_{-\infty}$ under the transformation $M(z)$. To show
 $\mathcal{T}^r \subseteq D$, write the boundaries $\partial \mathcal{T}^r$
 and $\partial D$ in polar coordinates $r\,e^{\imath\theta}$, with
 $r = f(\theta)$ and  $r = g(\theta)$ respectively.
 Then, with $\beta_r = \sec^r(\pi/r)$,
 \begin{align*}
  f(\theta) &= (1-\param/2)^r \sec^r(\theta/r) , \quad \mbox{and}\\
  g(\theta) &= (1-\param/2)^r \Big(\frac{1}{2}(1-\beta_r) \cos(\theta) +
               \sqrt{\beta_r + \Big(\frac{1}{2}(1-\beta_r) \cos(\theta)
               \Big)^2} \Big),
 \end{align*}
 By symmetry across the real axis, it is sufficient to show that
 $f(\theta) \leq g(\theta)$ for $0 \leq \theta \leq \pi$.  This is true
 (i.e. after manipulating $f(\theta) \leq g(\theta)$), provided that the
 following inequality holds for $0 \leq \theta \leq \pi$,
 \begin{align*}
  h_r(\theta) := \beta_r - \sec^{2r}(\theta/r) 
  + \sec^r(\theta/r) \cos(\theta) \big(1 - \beta_r \big) \geq 0.
 \end{align*}
 Expanding $\cos(\theta)$ in powers of $\cos(\theta/r)$ via the binomial
 series, a direct computation of $h_r(\theta)$ (on $0 \leq \theta \leq \pi$)
 yields
 \begin{align*}
  h_3(\theta) &=
     \big(\sec^2(\theta/3) - 1\big) \big(4-\sec^2(\theta/3)\big)
     \big(5+\sec^2(\theta/3)\big) \geq 0\/,\\
  h_4(\theta) &=
     \big(\sec^2(\theta/4) - 1\big) \big(2 - \sec^2(\theta/4)\big)
     \big(10 + 3\sec^2(\theta/4)+ \sec^4(\theta/4) \big) \geq 0\/.
 \end{align*}
 For $h_5(\theta)$ we write:
 \begin{align*}
  h_5(\theta) &=
    \big(\sec^2(\theta/5) - 1\big)\,\tilde{h}_5\big(\sec^2(\theta/5)\big),\\
  \textrm{where} \quad \tilde{h}_5(x) &=
    -x^4 - x^3 - x^2 - (5\beta_5 - 4)x - 16 + 15\beta_5\/.
 \end{align*}
 We claim now that $\tilde{h}_5(x) \geq 0$ for $1 \leq x \leq \sec^2(\pi/5)$.
 For this, note that $\beta_5 > \sec^4(\pi/6) = 16/9$ which shows that
 $\tilde{h}_5(1) > 10(16/9 - 3/2) > 0$. By construction, we also know that
 the boundary $\partial \mathcal{T}^5$ and $D$ touch at $\theta = \pi$, which
 implies $f(\pi) = g(\pi)$. This can then be used to show that
 $\tilde{h}_5(\sec^2(\pi/5)) = 0$. Finally, applying Descartes' rule of signs
 to the derivative $\tilde{h}_5'(x)$ shows that $\tilde{h}_5'(x)$ has no roots
 for $x > 0$. Hence, $\tilde{h}_5(x)$ is decreasing, and thus
 $\tilde{h}_5(x) \geq 0$ on $1 \leq x \leq \sec^2(\pi/5)$.
\end{proof}
%
%
%
%
\begin{figure}
\begin{minipage}{0.65\textwidth}
 \includegraphics[width=0.53\textwidth]{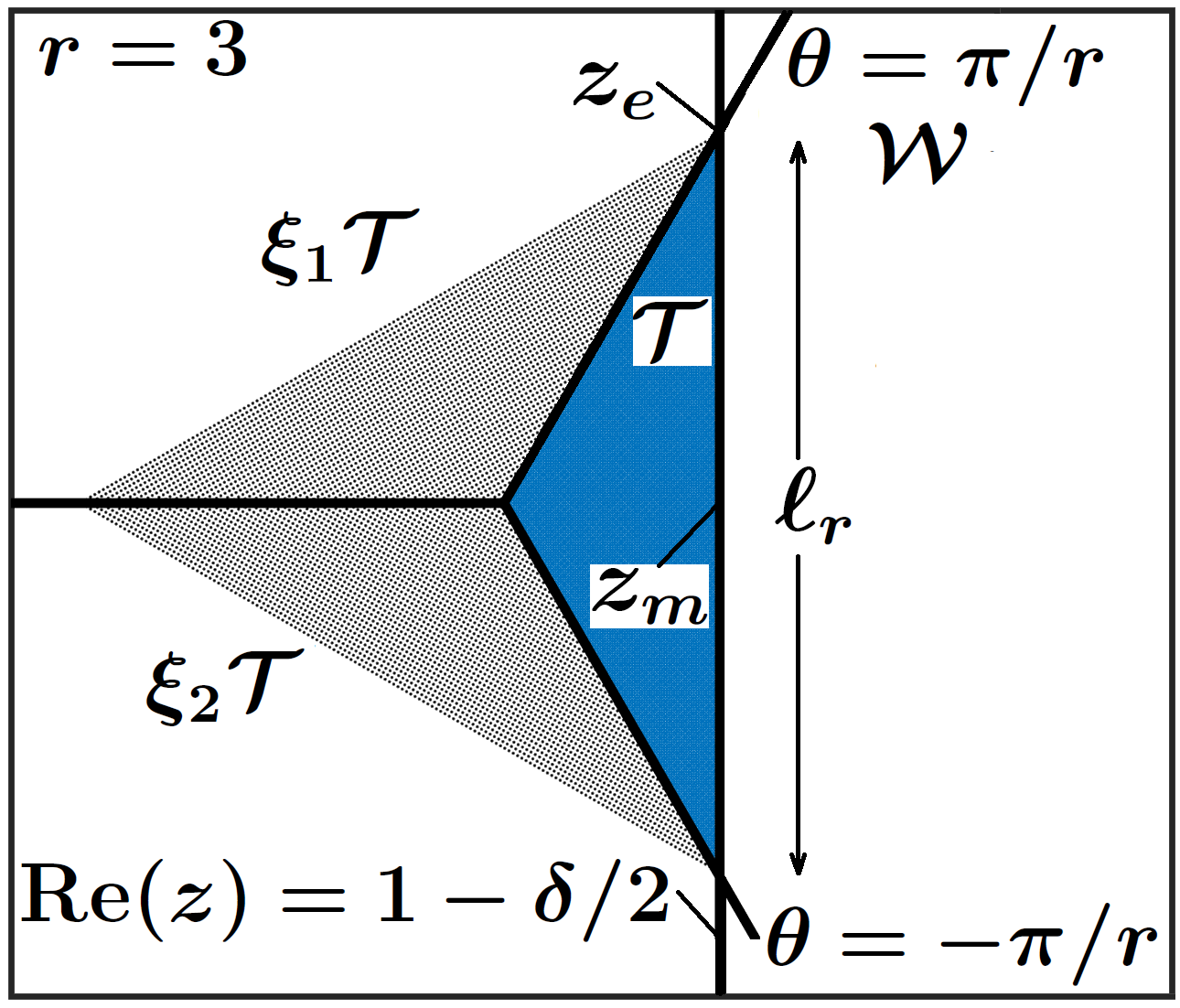}
 \hfill
 \parbox[b]{0.44\textwidth}{\small The set $\mathcal{T}$
  (darker shaded region) in relation to $\mathcal{W}$, for $r = 3$. The
  rotated sets $\rzeta_j \mathcal{T}$ (lighter shaded regions) satisfy the
  constraint inequality in equation \eqref{Inequality},
  $\mathrm{Re}( \rzeta_j \mathcal{T} ) < 1 - \param/2$. The set
  $\mathcal{D}_{-\infty}$ is given by 
  $\mathcal{D}_{-\infty} = \varphi^{-1}(\mathcal{T})$.\\\rule{0mm}{0mm}}
\end{minipage}
\hfill
\begin{minipage}{0.31\textwidth}
 \includegraphics[width = 0.98\textwidth]{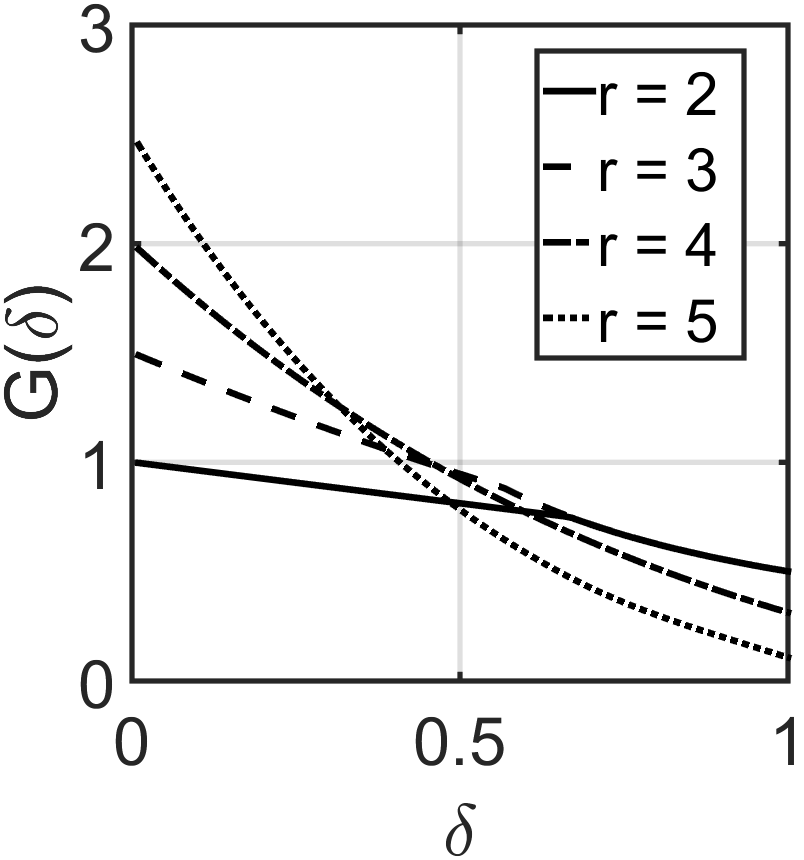}
\end{minipage}
\caption{Left: the set $\mathcal{T}$.\; Right: plot of $G(\param)$, as
 defined by equation \eqref{GFunction}.}\vspace*{-1.25em} \label{calT_Gpos}
\end{figure}
%
\begin{figure}[htb!]
\centering
\includegraphics[width = 0.31\textwidth]{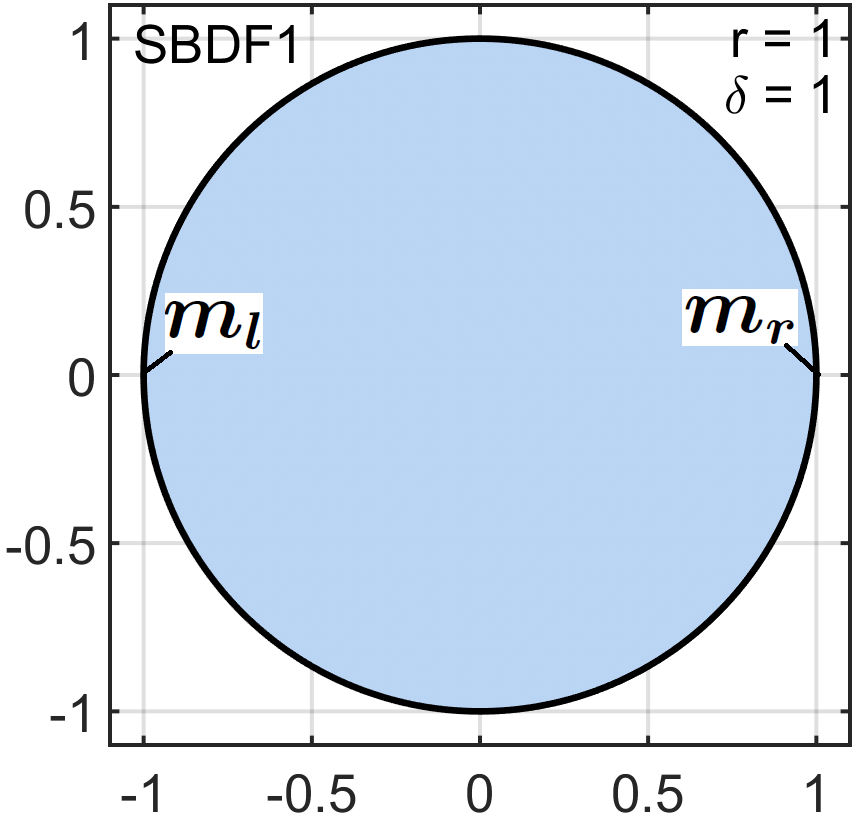} 
\includegraphics[width = 0.31\textwidth]{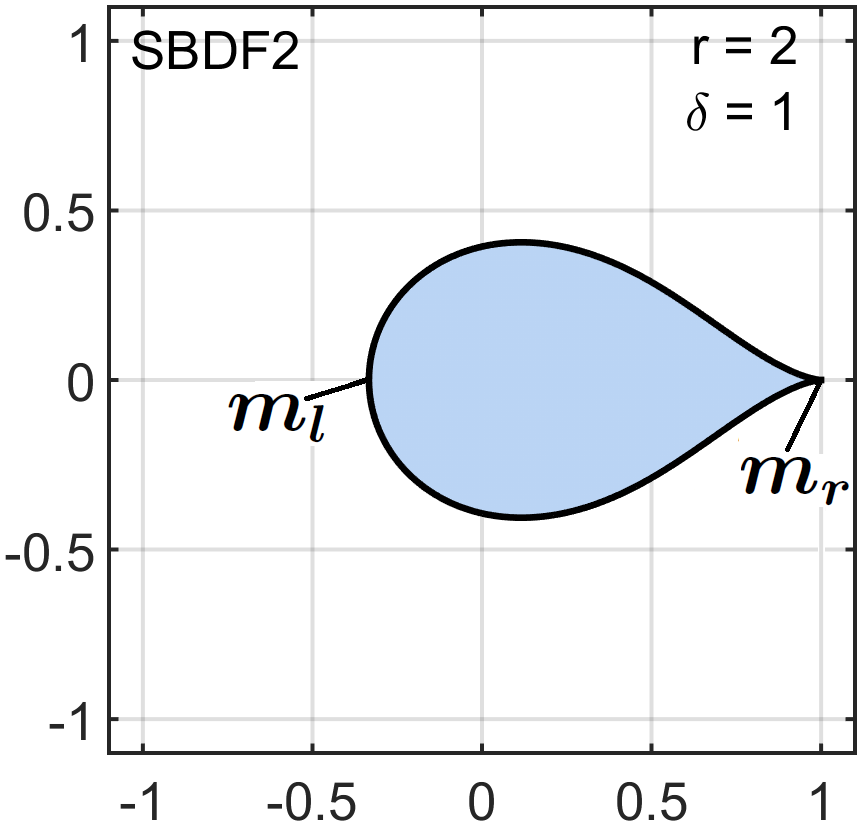} 
\includegraphics[width = 0.31\textwidth]{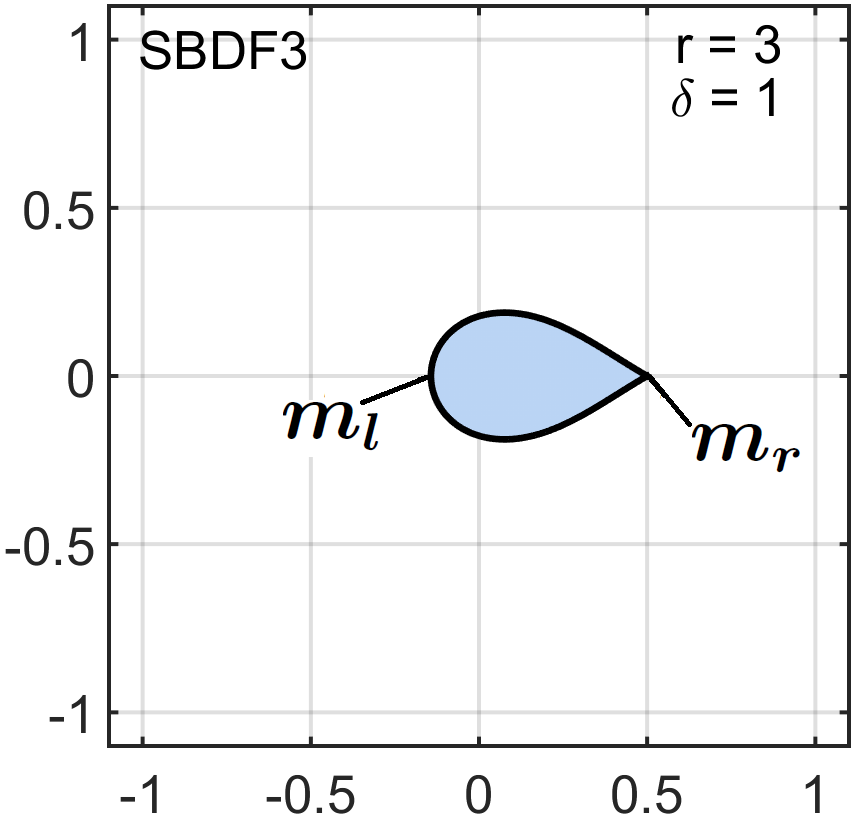} 
\caption{The sets $\mathcal{D}_{-\infty}$ (which by virtue of 
 Proposition~\ref{NestedSets} equal $\mathcal{D}$) are shown shaded. The
 parameters are: $\param = 1$ (SBDF schemes) and orders $r = 1,2,3$ (left
 to right). Formulas for the boundary are given by
 Theorem~\ref{CharacterizationUStab}.} 
 \vspace*{-1.25em}
 \label{BoundaryLocus}
\end{figure}
%
   Figure~\ref{BoundaryLocus} illustrates
   Theorem~\ref{CharacterizationUStab} by plotting the sets
   $\mathcal{D}_{-\infty}$ for the well-known SBDF schemes.
Using the characterization of $\mathcal{D}_{-\infty}$ in 
Theorem~\ref{CharacterizationUStab}, we are now in a position to show that not
only is $\mathcal{D} \subseteq \mathcal{D}_{-\infty}$, but that this inclusion
is also an equality: $\mathcal{D} = \mathcal{D}_{-\infty}$.

To first illustrate that $\mathcal{D} = \mathcal{D}_{-\infty}$, in
Figure~\ref{Nested_stability_regions} we plot $\mathcal{D}_y$ for different
values of $y$, using the \emph{boundary locus} (chapter 7.6,
\cite{LeVeque2007}) method. Specifically, $\mathcal{D}_y$ is a region whose
boundary is a subset of the locus
\begin{align}\label{Gamma_y}
 \vGamma_y := \Big\{ \frac{1}{b(z)}\big( c(z) - y^{-1} a(z) \big): 
   |z| = 1 \Big\}, \quad
 \vGamma_{-\infty} := \Big\{\frac{c(z)}{b(z)} : |z| = 1 \Big\}.
\end{align}
Equation (\ref{Gamma_y}) is obtained by isolating $\mu$ in equation
(\ref{Model_Eq}) and letting $z$ vary over the unit circle.
Figure~\ref{Nested_stability_regions} shows the nested stability regions 
$\mathcal{D}_y$ for orders $r= 3, 4, 5$ and fixed parameter value $\param=1$.
In the figure, the solid curve traces out $\vGamma_y$ corresponding to the
boundary locus for $\mathcal{D}_y$. The dashed curves show as a reference
$\vGamma_y$ for different $y$ values. Although the plots are only for one
value of $\param$, the limiting behavior $\mathcal{D} = \mathcal{D}_{-\infty}$
is observed for all $0 < \param \leq 1$.

We now show that the set equality $\mathcal{D} =\mathcal{D}_{-\infty}$  is a
direct consequence of the fact that the function $G(\param)$ (defined below
for the \imex schemes  in Definition~\ref{NewImExCoeff}) is positive. Note
that $G(\param)$, roughly speaking, is a measure of the distance of
$\vGamma_y$ to the set $\mathcal{D}_{-\infty}$ --- and it is the key to showing
that $\mathcal{D} = \mathcal{D}_{-\infty}$.
\begin{align}\label{GFunction}
 G(\param) := \inf_{y < 0} \min_{w \in \vGamma_{y}} \Big[ \Big(
   \mathrm{Re}\big( \varphi(w) \big) -
   (1 - \param/2) \Big) (1-y) \delta^{-2}  \Big].
\end{align}
This function may be numerically computed, which leads to:
%
\begin{numobservation}\label{NumObservation}
 Numerical computations (shown in Figure~\ref{calT_Gpos}, right) indicate
 that: for $0 < \param \leq 1$ and $2 \leq r \leq 5$, $G(\param) > 0$.
\end{numobservation}
%
This fact is introduced as an assumption below, in
Proposition~\ref{NestedSets}.

The positive factor $(1-y) \delta^{-2}$ in equation \eqref{GFunction}
is included to re-scale the difference between $\mathrm{Re}(\varphi(w))$ and
$(1-\param/2)$, which vanishes as $y\rightarrow -\infty$ or $\param \rightarrow 0$.
This re-scaling helps to visually verify that $G(\param)$ does not change
sign, even as $y\rightarrow -\infty$ or $\param \rightarrow 0$. To computationally
handle the infinite interval $-\infty < y < 0$, we introduce the change of
variables $\tilde{y} = (1 - y)^{-1}$, so that $0 < \tilde{y} < 1$.
For each fixed value of $\tilde{y}$, we parameterize $\vGamma_{y}$ as the
image of the unit circle, which then allows us to compute $G(\param)$ as
a double minimization over two real variables on bounded intervals.
%
\begin{proposition}\label{NestedSets} 
 (The set $\mathcal{D} = \mathcal{D}_{-\infty}$)
 (i) For $r = 1$ and $0 < \param \leq 1$, $\mathcal{D} = \mathcal{D}_{-\infty}$.
 (ii) For $2 \leq r \leq 5$, assume that: $G(\param) > 0$, $0 < \param \leq 1$.
 Then
 \begin{align}\label{NumericalTheorem}
  \mu \in \mathcal{D}_{-\infty} \quad\Longrightarrow\quad
  \mu \in \mathcal{D}_{y} \textrm{ for any } y \in \negNum. 
 \end{align}
 In other words, for every $y \in \negNum$ the set $\mathcal{D}_y$ contains 
 the limiting set $\mathcal{D}_{-\infty}$.  As a result, the unconditional
 stability region is $\mathcal{D} = \mathcal{D}_{-\infty}$.
\end{proposition}
%
\begin{proof} (Proposition~\ref{NestedSets})
 For (i), the proof is straightforward as $\mathcal{D}_y$ is a disk centered 
 at $1-(\param^{-1}-y^{-1})$ with radius $\param^{-1}-y^{-1}$. For (ii)
 the proof involves two steps. First, we use a standard continuity argument
 to show that if $\mu \in \mathcal{D}_{-\infty}$, but
 $\mu \notin \mathcal{D}_{y_0}$ for some $y_0 < 0$, then there is an
 intermediate $y$-value ($-\infty < y < y_0$) where $\mu$ must lie on the
 boundary locus	$\mu \in \vGamma_y$.  Next we show that $\vGamma_y$ is
 bounded away from $\mathcal{D}_{-\infty}$ when $y < 0$. It then follows that
 $\mu \in \mathcal{D}_{y}$ whenever $\mu \in \mathcal{D}_{-\infty}$.

 To proceed with the first step, we define the following polynomial function
 based on equation \eqref{Model_Eq}
 \begin{align}
  P(z; \tilde{y}) := c(z) - \mu b(z) + \frac{\tilde{y}}{1-\tilde{y}} a(z).
 \end{align}
 Here $y =1 - \tilde{y}^{-1}$, so that $0 < \tilde{y} < 1$
 (resp. $\tilde{y} = 0$) corresponds to $y < 0$ (resp. $y = -\infty$), which
 will be useful in the subsequent continuity argument.  
	To minimize additional notation,
	we will continue to use $\mathcal{D}_{y}$ and 
	$\vGamma_{y}$ as sets, and $\tilde{y}$ as the 
	parameter in the polynomials, with the understanding that 
	$y=1 - \tilde{y}^{-1}$.	
	Then $\mathcal{D}_{y}$ is defined as $\mu \in \mathbb{C}$
	such that $P(z;\tilde{y})$ has $r$ roots inside the unit circle,
	or alternatively: (i) $P(z; \tilde{y}) \neq 0$ on the unit circle
	$|z| = 1$, and (ii) the function $F(\tilde{y}) = r$, where
	$F(\tilde{y})$ counts the number of roots $|z| < 1$ via the Cauchy
	integral formula:
 \[
   F(\tilde{y}) := \frac{1}{2\pi \imath} 
   \oint_{|z| = 1} \frac{P_z(z; \tilde{y})}{P(z;\tilde{y})}\du z.
 \]
 Now, $F(\tilde{y})$ is continuous as a function of $\tilde{y}$, 
	and also a constant, as long as it is defined. 
	The only way $F(\tilde{y})$ may change values is if 
	$P(z;\tilde{y}) = 0$ vanishes for some $|z| = 1$ on the unit 
	circle, which implies $\mu \in \vGamma_{y}$. Hence, if for a 
	given $\mu$, $F(0) = r$ and $F(\tilde{y}_0) \neq r$,
	then there must exist a point $0 < \tilde{y} < \tilde{y}_0$
	such that $\mu \in \vGamma_{y}$.

 To show that $\vGamma_{ y}$ does not intersect $\mathcal{D}_{-\infty}$ for
 $0 < \tilde{y} < 1$, we exploit the fact that the mapping $\varphi(z)$,
 defined in Theorem~\ref{CharacterizationUStab}, simplifies the shape of 
 $\varphi(\mathcal{D}_{-\infty})=\mathcal{T}$. Specifically, $\varphi(z)$ is
 a one-to-one mapping of $\mathbb{C}$ to the wedge $\mathcal{W}$, so that it
 is sufficient to show that the mappings of $\mathcal{D}_{-\infty}$ and
 $\vGamma_{y}$ under $\varphi(z)$ do not intersect, i.e.
 $\varphi(\vGamma_{y})$ does not intersect $\mathcal{T}$, for
 $0 < \tilde{y} < 1$.  Since $\mathcal{T}$ is contained	within the half-plane
 $\mathrm{Re}(z) < 1 - \param/2$, we arrive at the following observation: if
 \begin{align}\label{Positive_Function}
  \mathrm{Re}\big( \varphi(w ) \big) - (1-\param/2) > 0,
  \quad \textrm{for all } 0 < \tilde{y} < 1, \; w \in \vGamma_{y},
 \end{align}
 then $\varphi(\vGamma_{y})$ and $\mathcal{T}$ do not intersect.
	Multiplying	the left-hand side of the 
	inequality \eqref{Positive_Function} by 
	the positive factor $\param^{-2} \tilde{y}^{-1}=\param^{-2}(1-y) > 0$, 
	and minimizing over $0 < \tilde{y} < 1, \; w \in \vGamma_{y}$, 
	yields the function $G(\param)$.  Hence, we arrive at the 
	conclusion that $\varphi(\vGamma_{y})$ and $\mathcal{T}$ 
	do not intersect whenever $G(\param) > 0$, which together
	with the first step of the proof, implies
	$D_{-\infty} \subseteq \mathcal{D}_y$ for all $y < 0$.
\end{proof}
%
%
\begin{figure}[htb!]
\includegraphics[width = 0.32\textwidth]{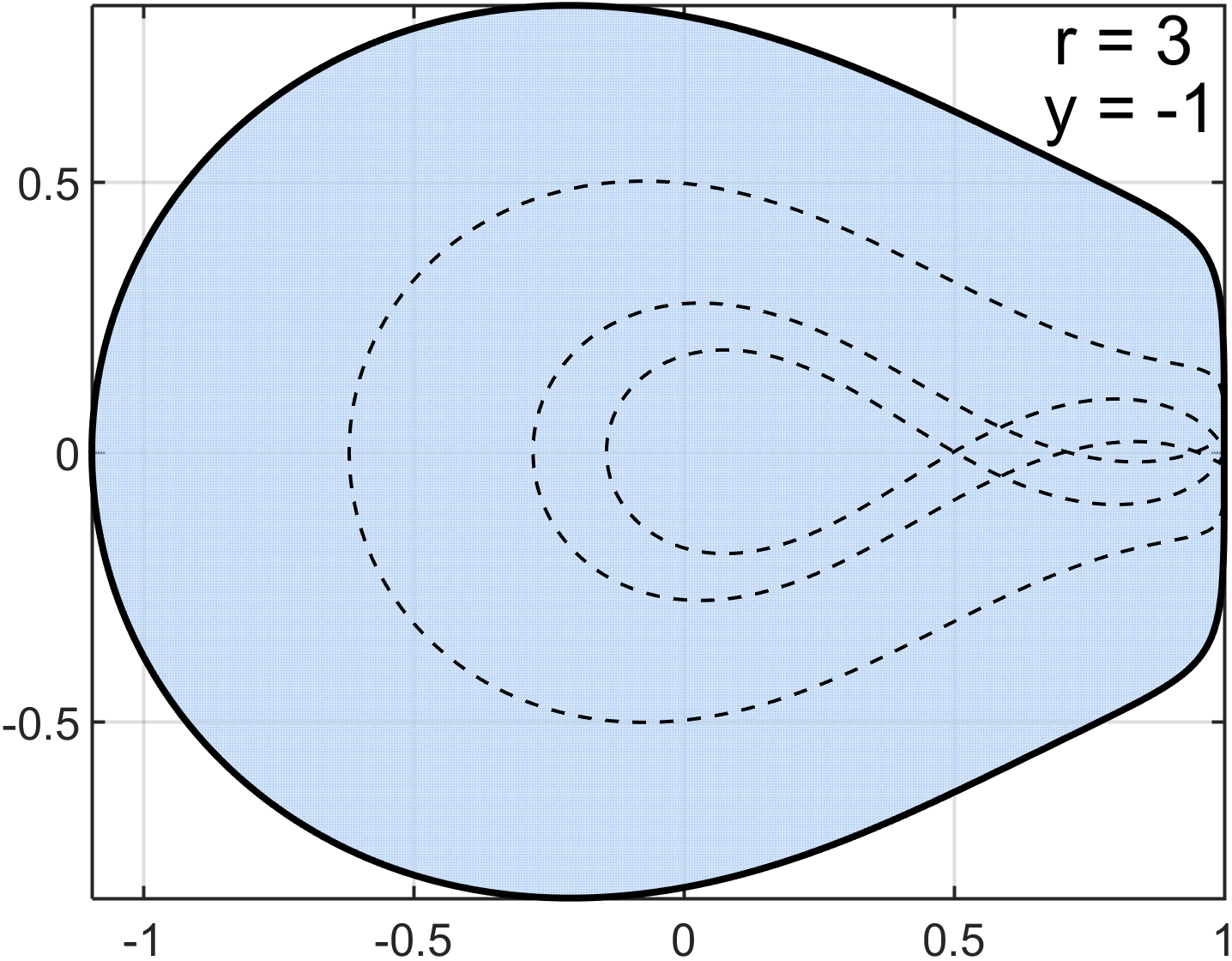} \hfill
\includegraphics[width = 0.32\textwidth]{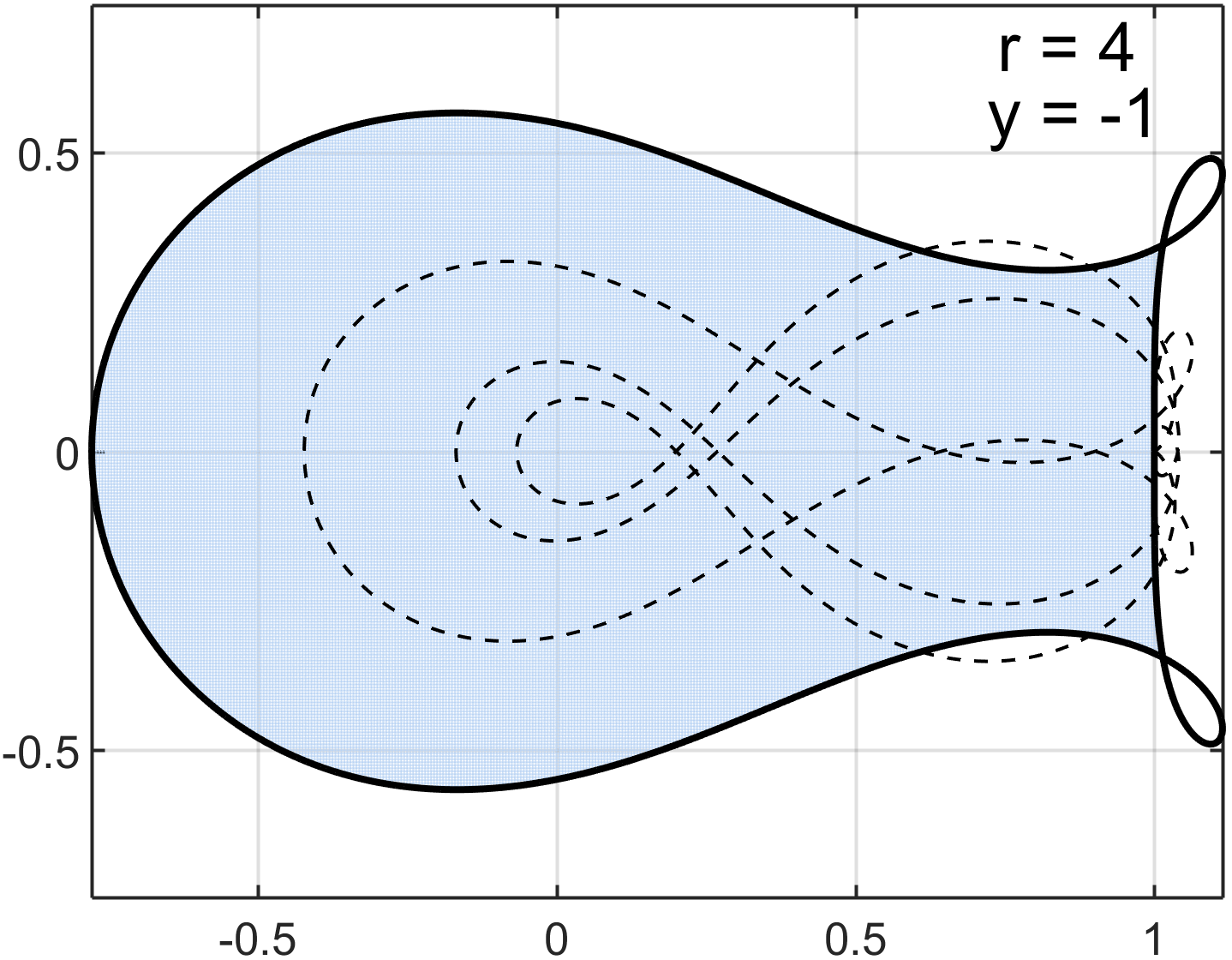} \hfill
\includegraphics[width = 0.32\textwidth]{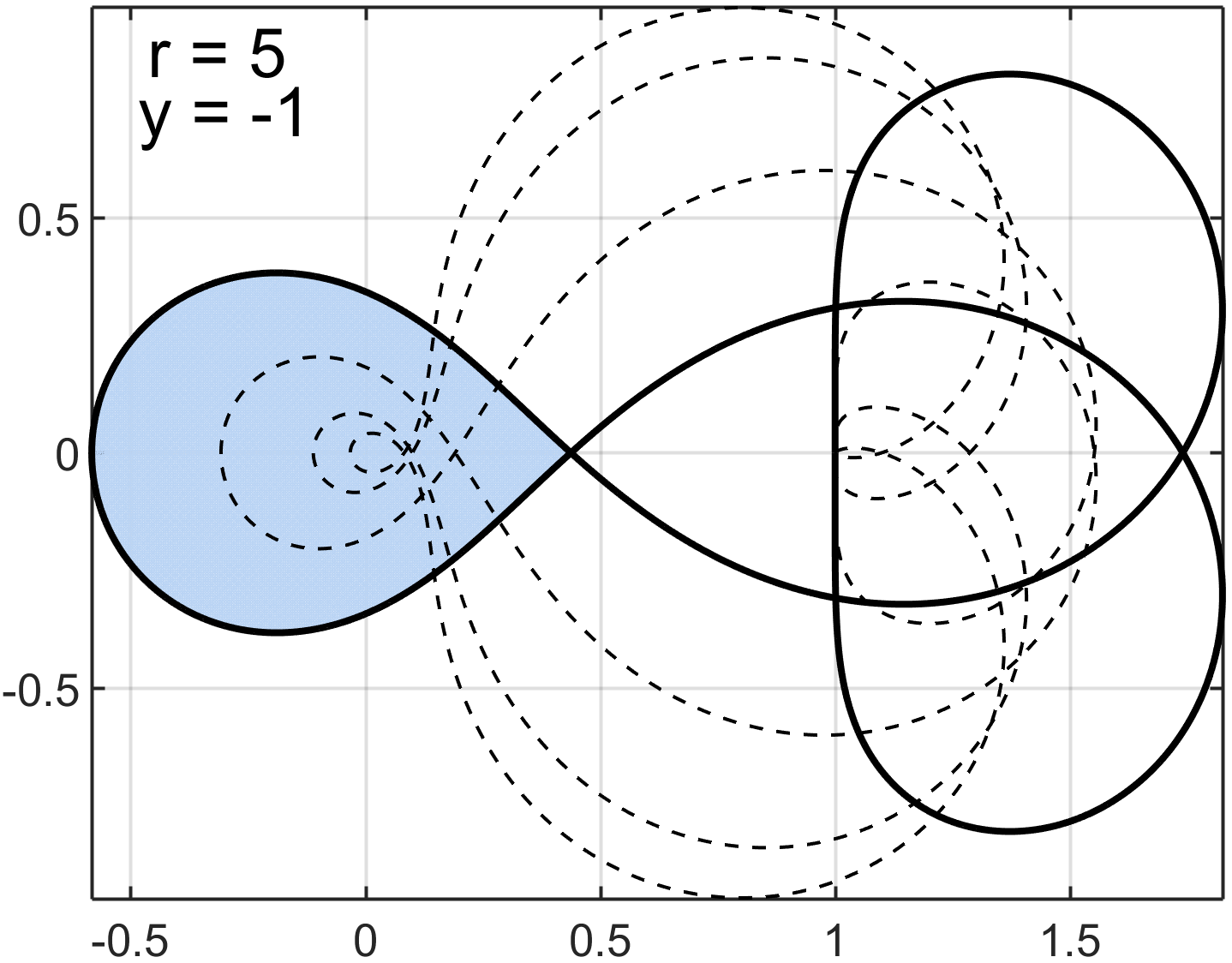}

\vspace*{1em}
\includegraphics[width = 0.32\textwidth]{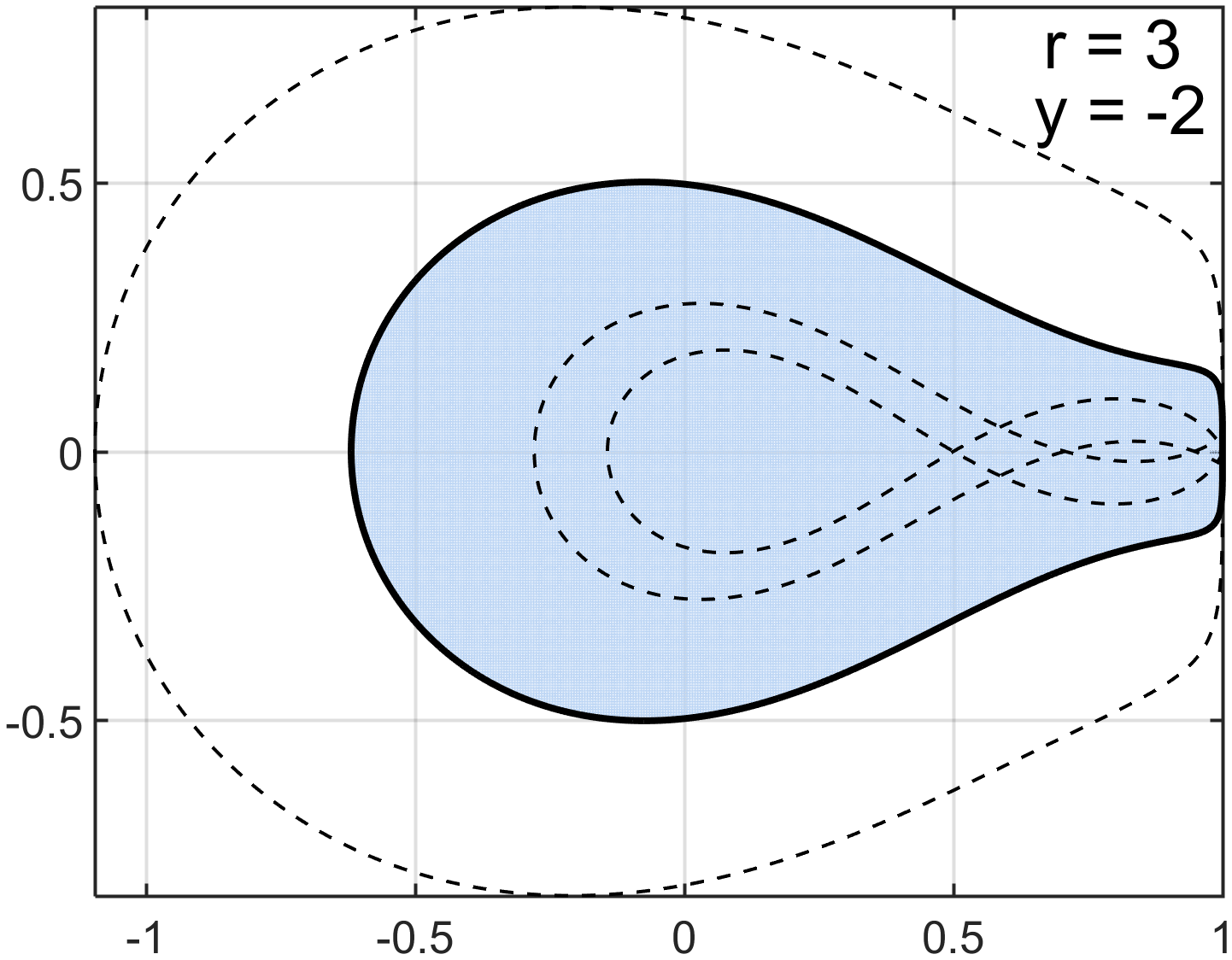} \hfill
\includegraphics[width = 0.32\textwidth]{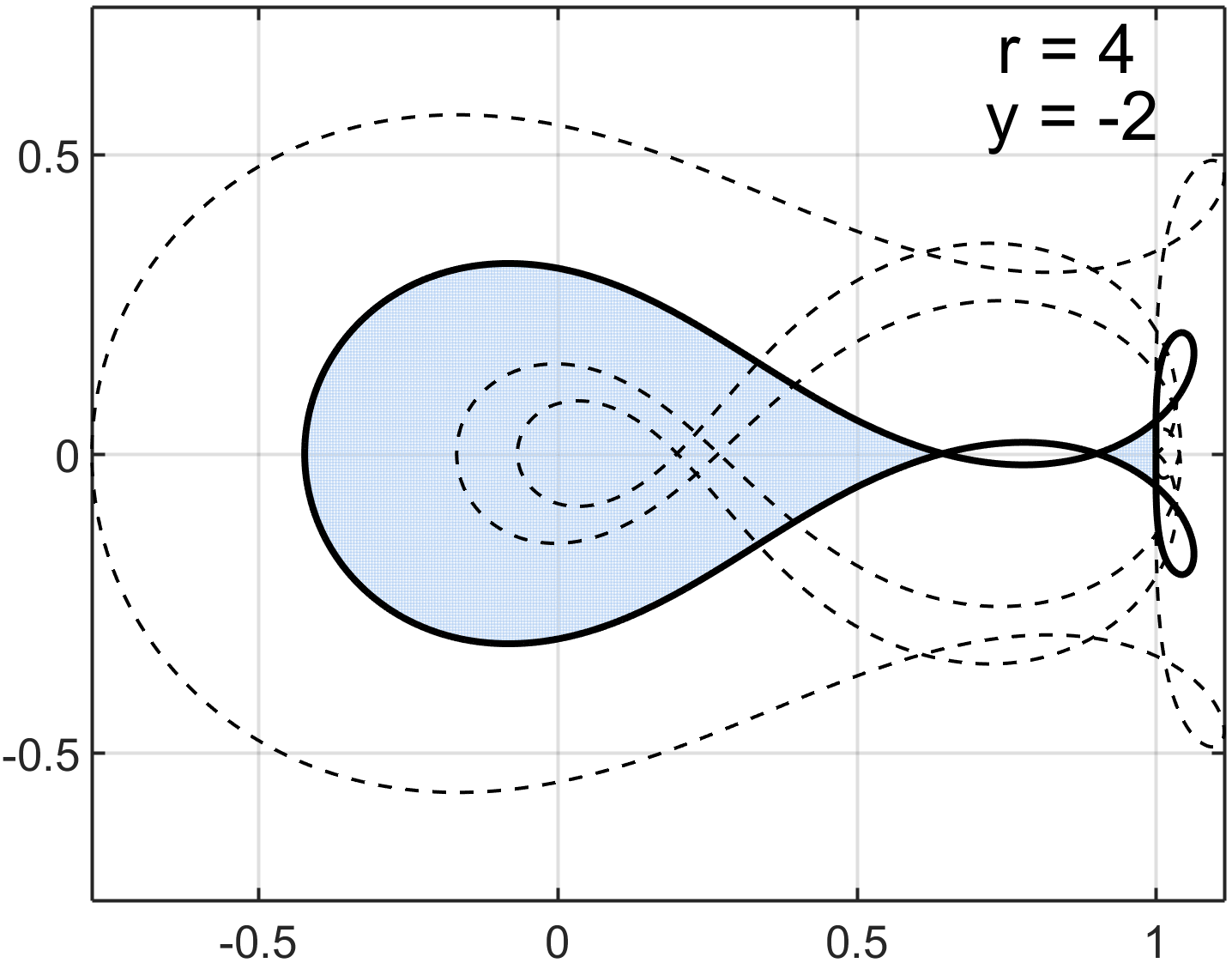} \hfill
\includegraphics[width = 0.32\textwidth]{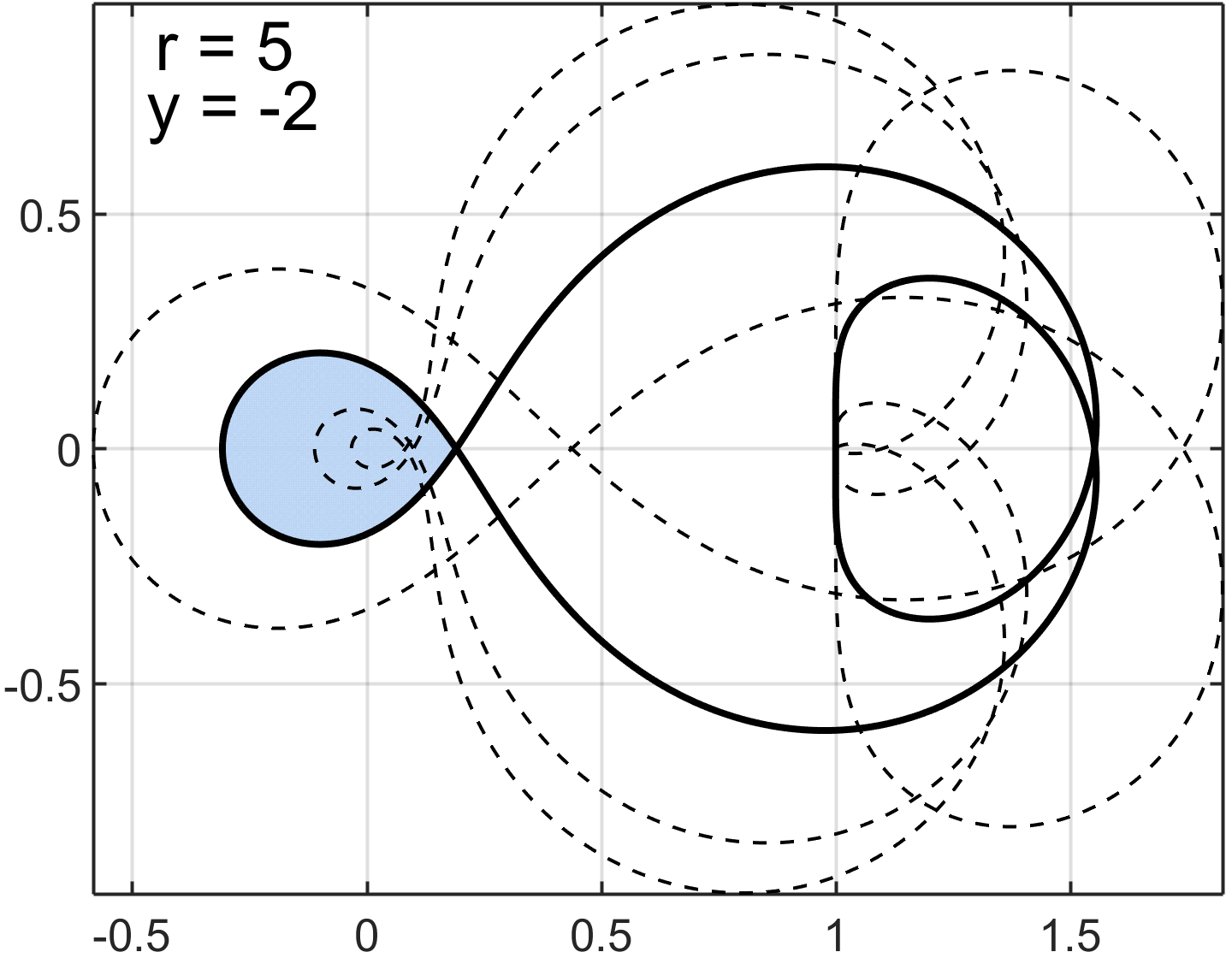}

\vspace*{1em}
\includegraphics[width = 0.32\textwidth]{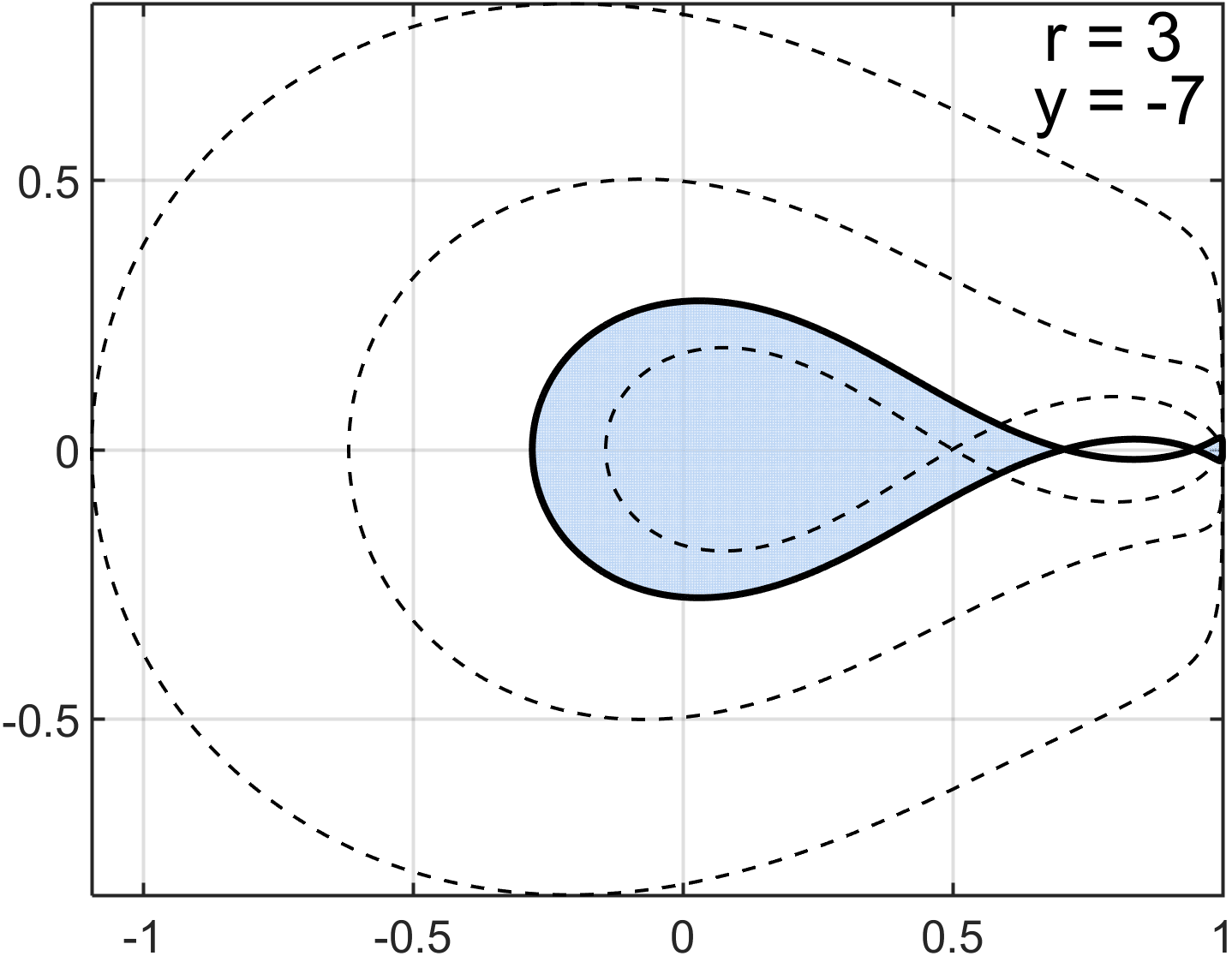} \hfill
\includegraphics[width = 0.32\textwidth]{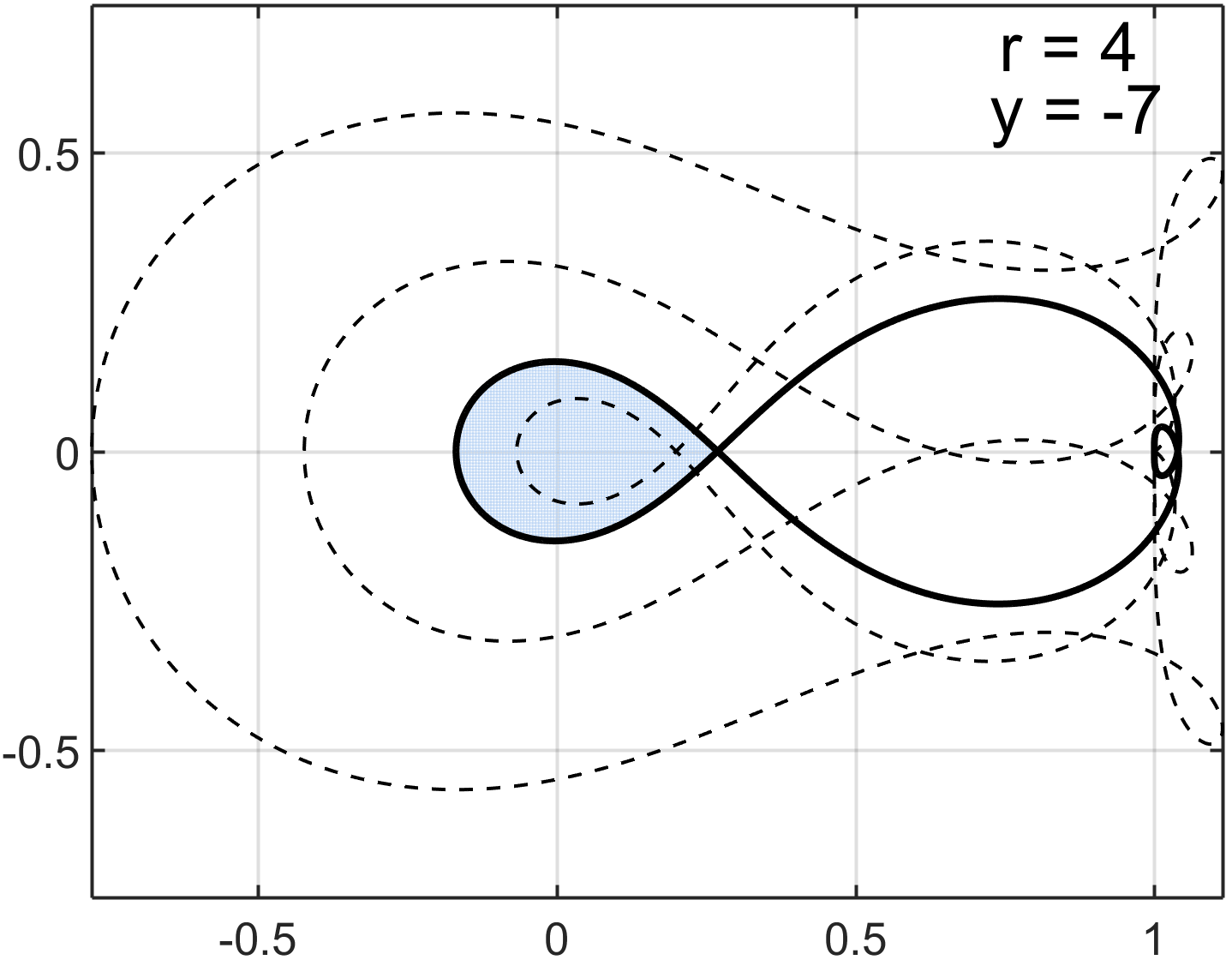} \hfill
\includegraphics[width = 0.32\textwidth]{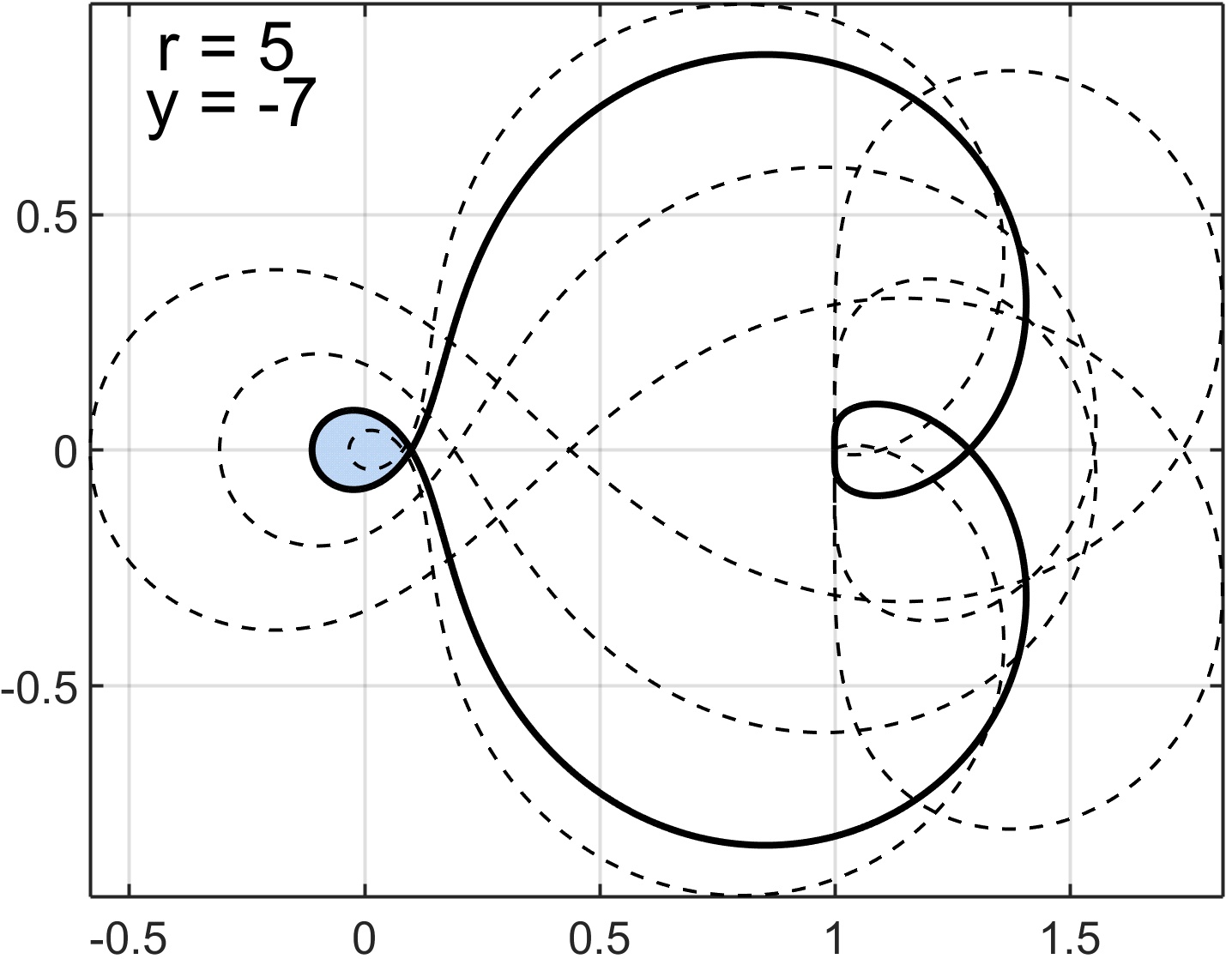}

\vspace*{1em}
\includegraphics[width = 0.32\textwidth]{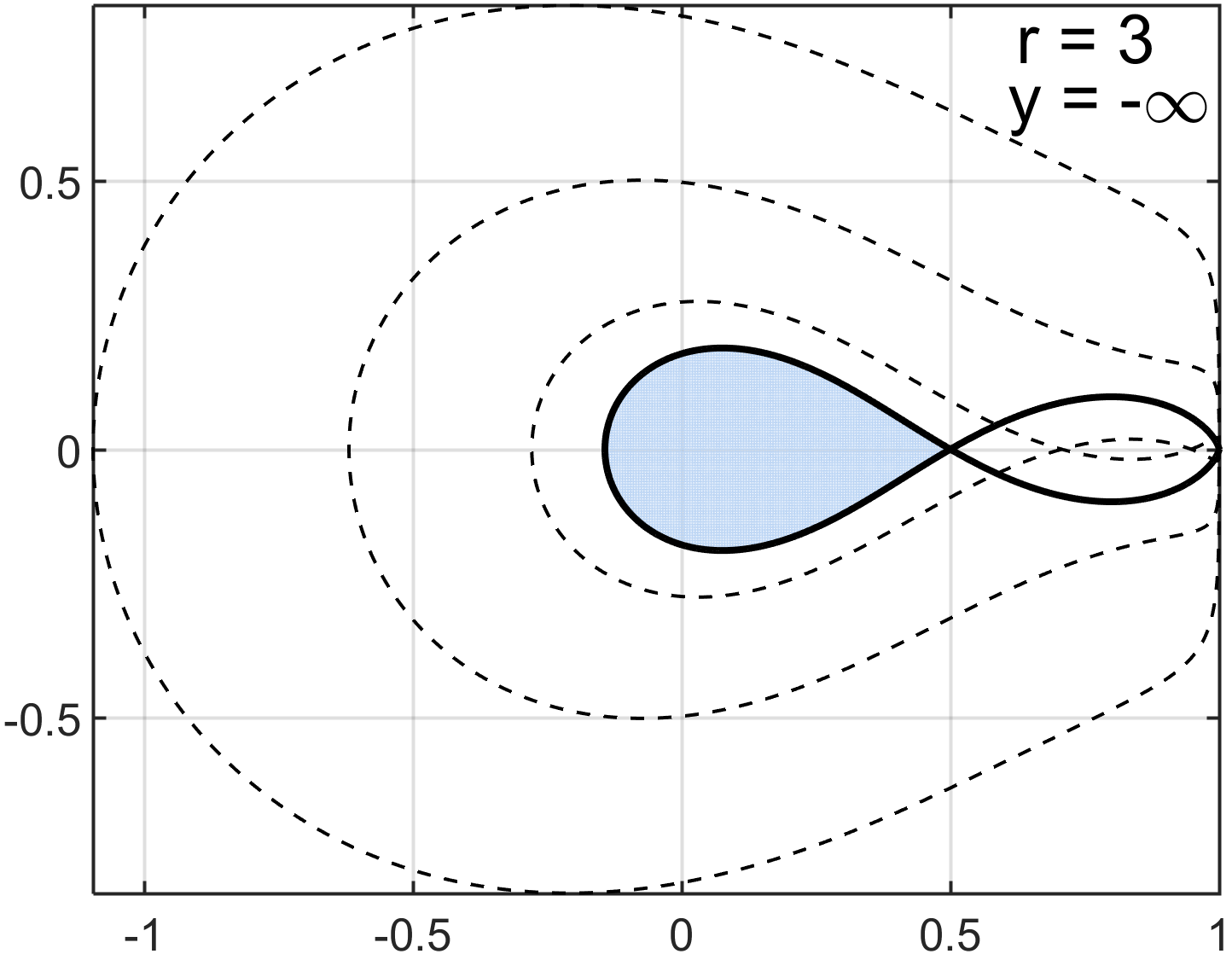} \hfill
\includegraphics[width = 0.32\textwidth]{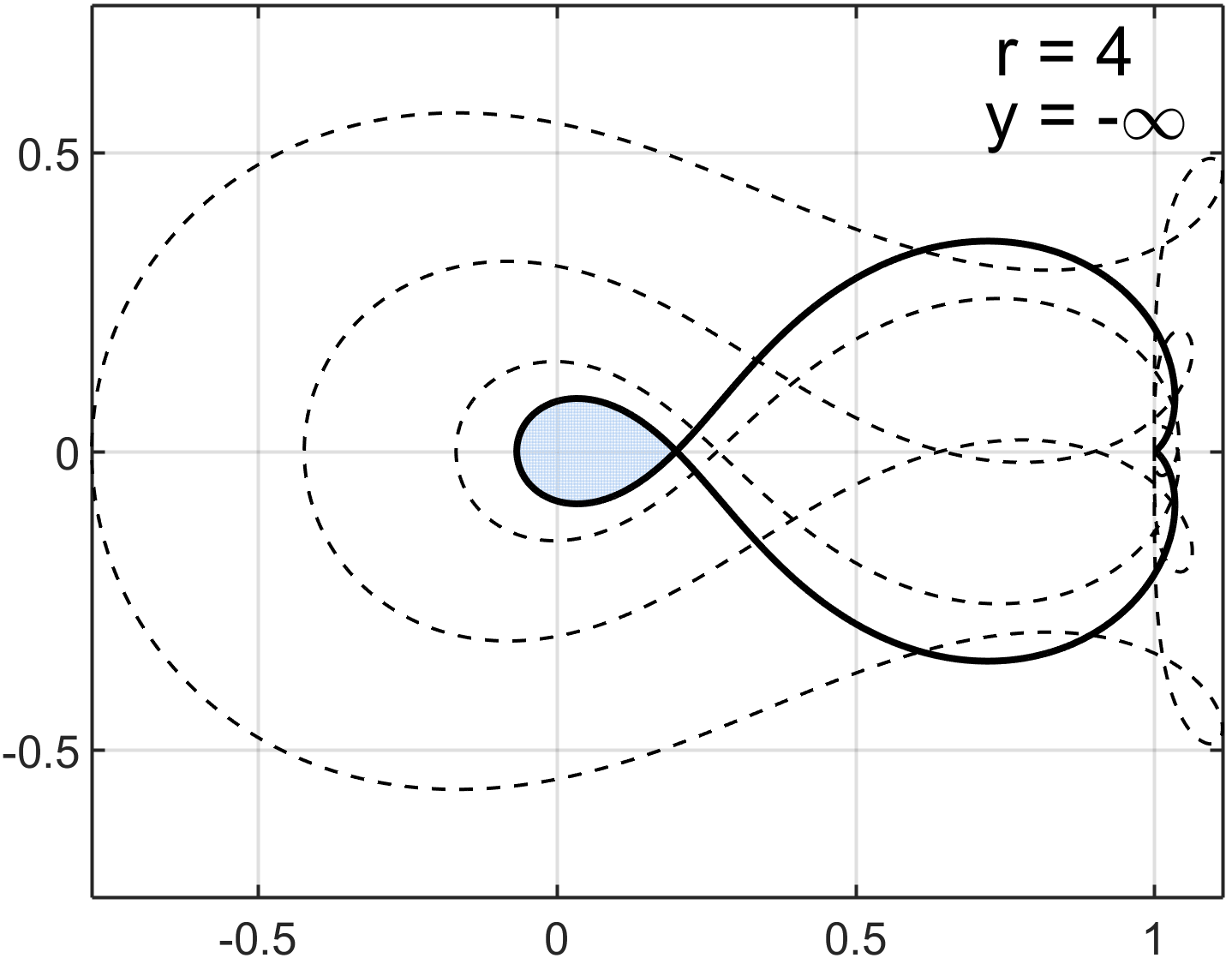} \hfill
\includegraphics[width = 0.32\textwidth]{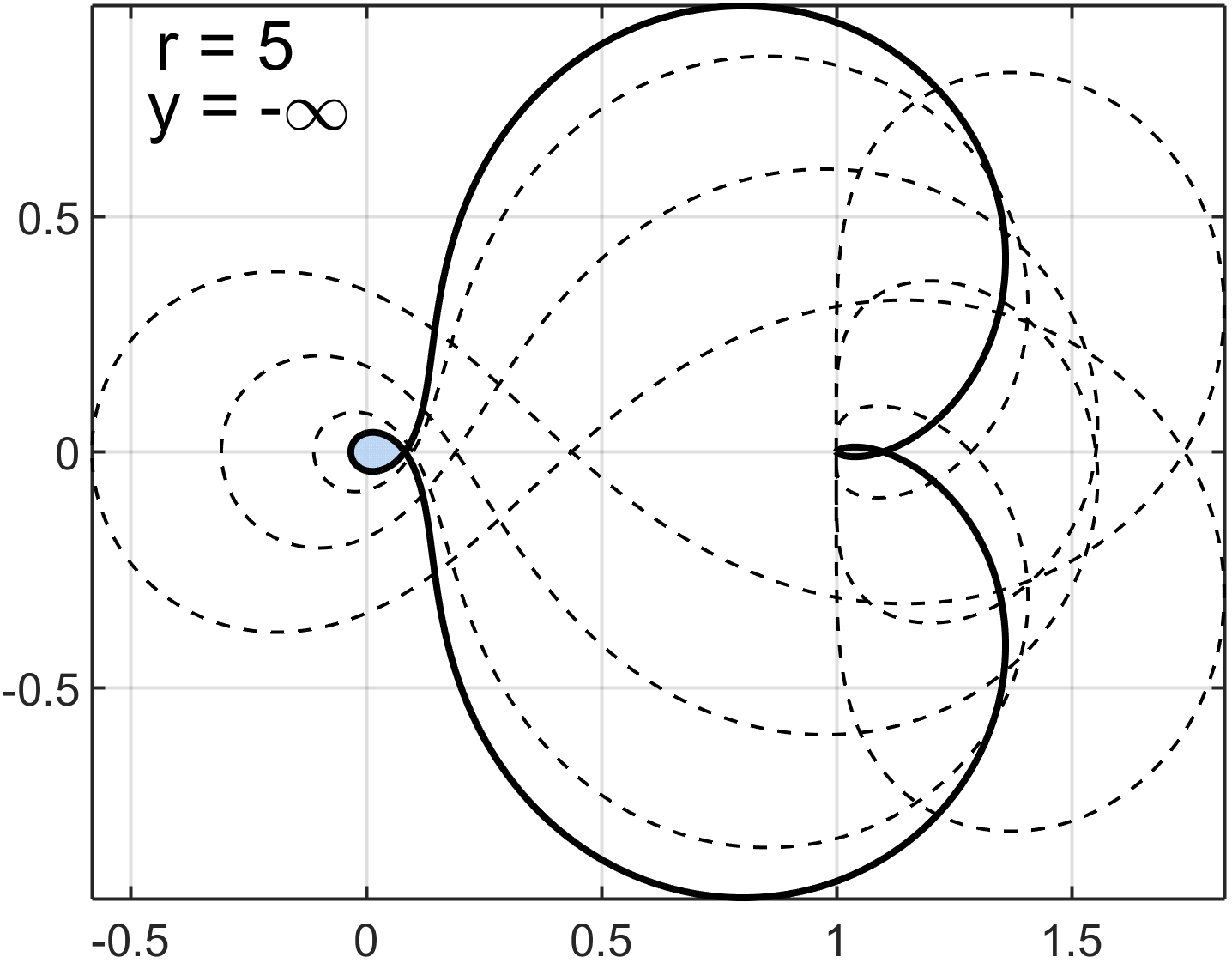}
\caption{Visualization of Proposition~\ref{NestedSets}: 
	$\mathcal{D}_{-\infty}$ is contained in $\mathcal{D}_y$ for all $y < 0$.
   Plot of the boundary locus $\vGamma_y$ (black curves) and the stability
   regions $\mathcal{D}_{y}$ (shaded regions) for: $y = -1, -2, -7, -\infty$
   (top to bottom), orders $r = 3, 4, 5$ (left to right), and fixed
   parameter value $\param = 1$. In each plot, the dashed lines
   $\vGamma_{-1}, \vGamma_{-2}, \vGamma_{-7}, \vGamma_{-\infty}$ are shown
   for reference. Note that the inclusion
   $\mathcal{D}_{-\infty} \subseteq \mathcal{D}_y$ is valid for all
   $y \in \negNum$.
} \vspace*{-1.25em}
\label{Nested_stability_regions}
\end{figure}
%

	With the exact boundary locus description in 
   Theorem~\ref{CharacterizationUStab}, and the subsequent result that
   $\mathcal{D} = \mathcal{D}_{-\infty}$, one may provide an asymptotic
   description of $\mathcal{D}$ in the limit $\param \ll 1$.
%
\begin{remark}\label{AsymptoticD} (Asymptotic $\mathcal{D}$) 
 Define the circle $C$ as
 \begin{align*}
  C &= \Big\{ z \in \mathbb{C} : \Big|z + \frac{1}{r\param} -  
       \frac{r+1}{2r}\Big| \leq \frac{1}{r \param} \Big\}.	
 \end{align*}
 Taking the asymptotic limit $\param \ll 1$ and values of
	$|z - 1| \gg \param$ in formula (\ref{ExactBoundary}) for 
	$\partial \mathcal{D}$ (which correspond to points in $\mathcal{D}$ 
	away from the right-most values 
	along the real axis), the exact boundary 
	$\partial \mathcal{D}$ approaches the circle $\partial C$:
	$\mathcal{D} \approx C + \mathcal{O}(\param).$ 
	For $r = 1$, the domain $\mathcal{D} = C$ is a 
	circle for all $0 < \param \leq 1$. 
\end{remark}
%
The circle $C$ in Remark~\ref{AsymptoticD} is obtained via an asymptotic 
computation, i.e. $\param \rightarrow 0$, of (\ref{ExactBoundary}).
Specifically, note that the starting value of the locus description for
$\mathcal{D}_{-\infty}$ in Theorem~\ref{CharacterizationUStab} satisfies 
$|z_0 - 1| = \mathcal{O}(\param)$, so that the locus parameter $z$ almost 
traces through an entire circle. Consider points $|z - 1| \gg \param$ 
and expand $c(z)/b(z)$ in a Laurent series in powers of $\param$ 
about $z = 1$:
\begin{align} \nonumber
 \frac{c(z)}{b(z)} &= \frac{(z-1)^r + \param r 
    (z-1)^{r-1} + \dotsb}{ \param r (z-1)^{r-1} + 
    \param^2 \frac{r(r-1)}{2}(z-1)^{r-2} + \dotsb} \\ \label{ApproxCircle}
 &= \frac{1}{\param r} 
   \Big( \frac{(z-1) + \param r + \mathcal{O}(\param^2)}{1 + 
   \param \frac{(r-1)}{2}(z-1)^{-1} + \mathcal{O}(\param^2)} 
   \Big) = \frac{1}{r \param}(z - 1) + \frac{r+1}{2r} + \mathcal{O}(\delta).
\end{align}
For values $|z| = 1$, equation (\ref{ApproxCircle}) describes the boundary 
of the circle $C$ defined in Remark~\ref{AsymptoticD} with radius 
$\frac{1}{r\param}$ and center $\frac{r+1}{2r} - \frac{1}{r\param}$. Hence
$\partial \mathcal{D} \approx \frac{1}{r \param}(z - 1) + \frac{r+1}{2r} 
   + \mathcal{O}(\delta)$, for $|z - 1| \gg \mathcal{O}(\param)$.  
Figure~\ref{NewImExStabilityRegions} shows the regions $\mathcal{D}$ for
different parameter values $\param$ and orders $2 \leq r \leq 5$. In
particular, the figure illustrates how the regions $\mathcal{D}$ grow
larger with decreasing $\param$ values, and also approach the asymptotic
circle $C$.

Having precise estimates for the geometric properties of $\mathcal{D}$, 
such as the formulas for $m_r$, $m_l$ and $C$, is very useful for the 
design of unconditionally stable schemes.  Specifically the design
of an unconditionally stable scheme require a simultaneous 
choice of matrix splitting $(\mat{A},\; \mat{B})$, and time stepping
coefficients $(a_j, b_j, c_j)$. If one knows, either through
direct numerical computation or analytic estimates, $W_p$ for a matrix
splitting $(\mat{A},\; \mat{B})$, then the estimates for $m_r, m_l$ and
$C$ can be used to choose a $\param$ value large enough to guarantee that
$W_p \subseteq \mathcal{D}$.  Such a choice of $\param$ will then provide
the suitable time stepping coefficients that guarantee unconditional stability.
We highlight such an approach in several numerical examples in
\Srm\ref{Sec_Examples}, as well as in greater detail in a companion 
paper on the practical aspects of unconditional stability for multistep 
\imex schemes.

\begin{SCfigure}
\begin{tabular}{l l}
\includegraphics[width = 0.31\textwidth]{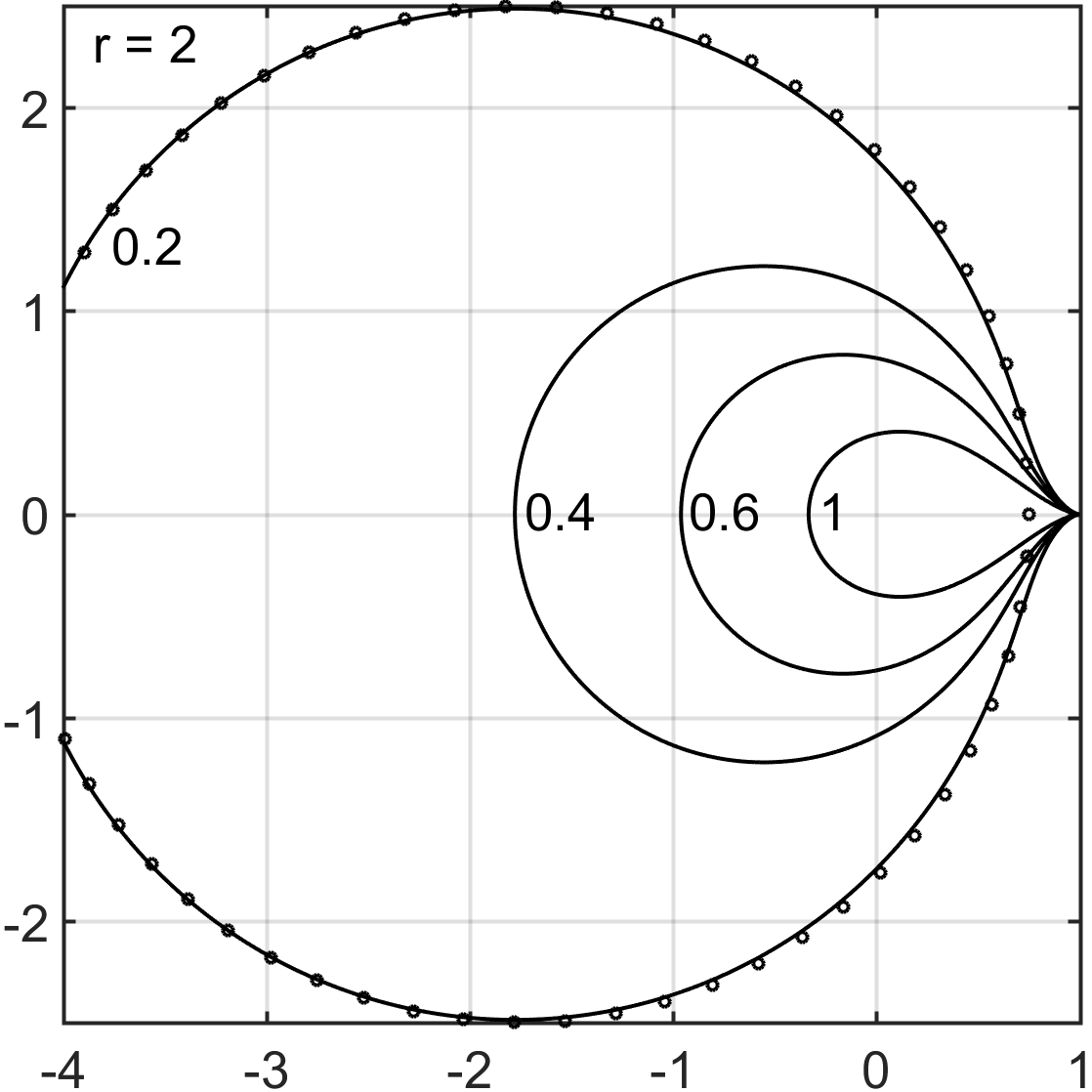}  & 
\includegraphics[width = 0.31\textwidth]{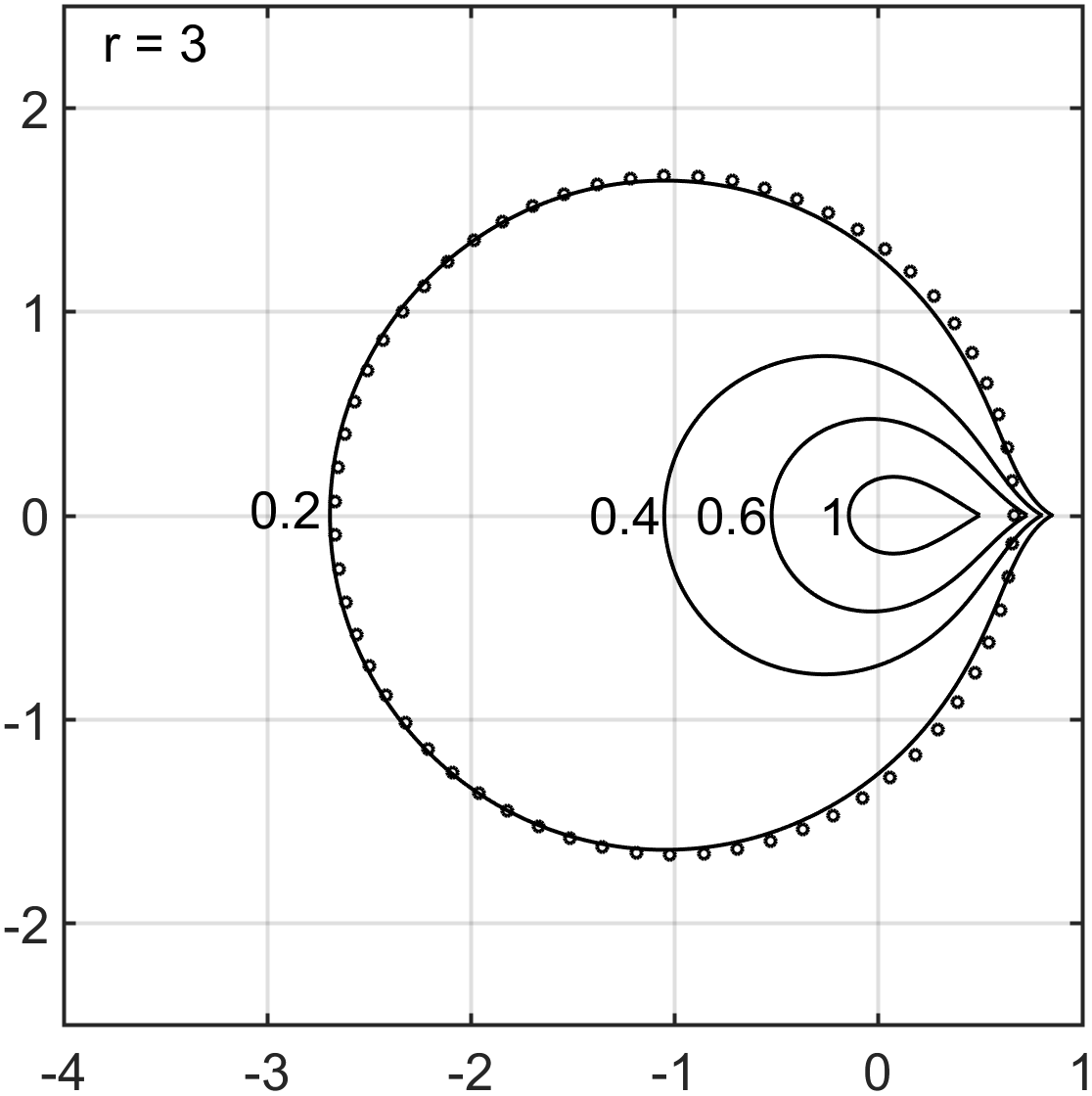}  \\
   \includegraphics[width = 0.31\textwidth]{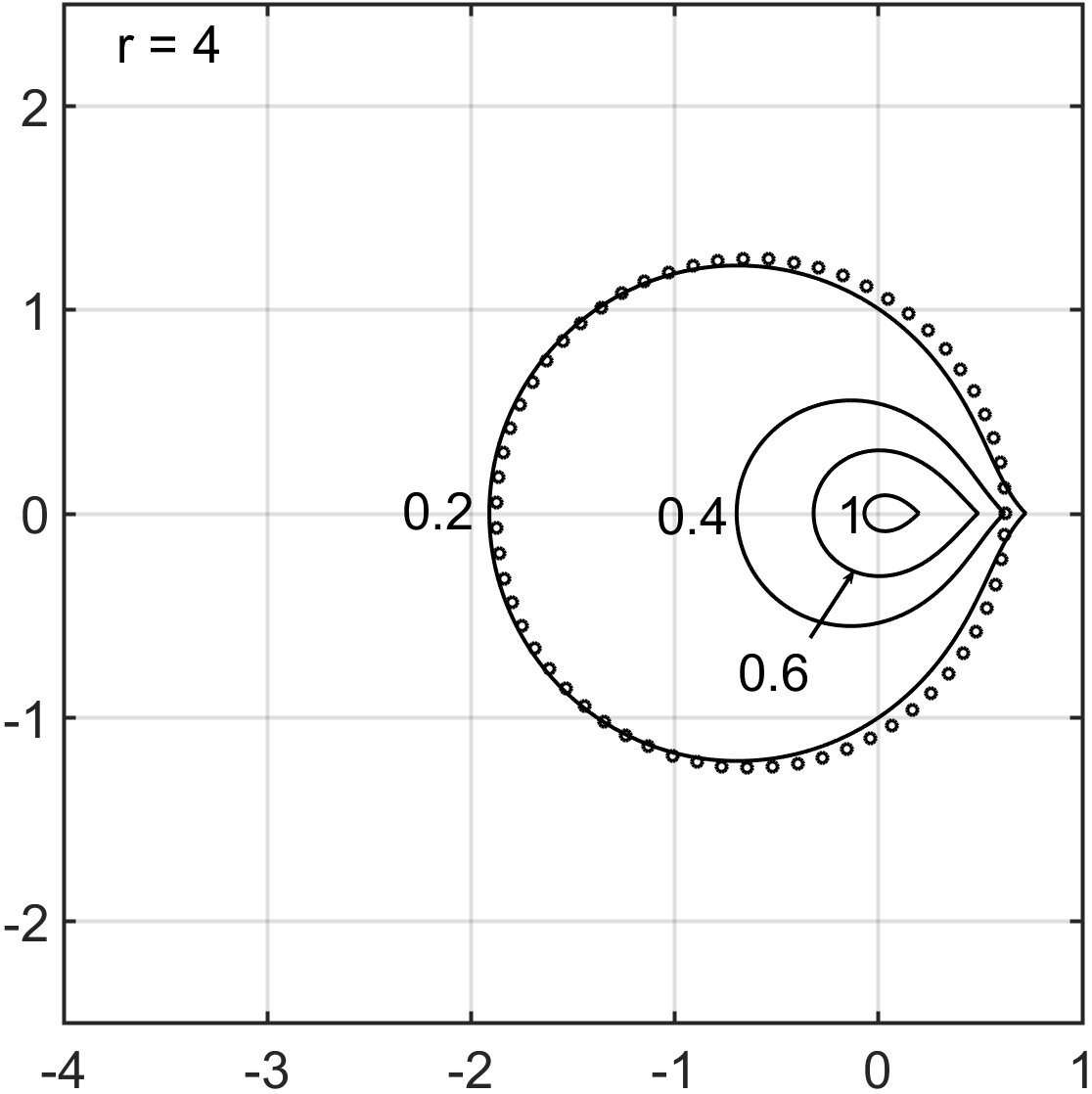} &
   \includegraphics[width = 0.31\textwidth]{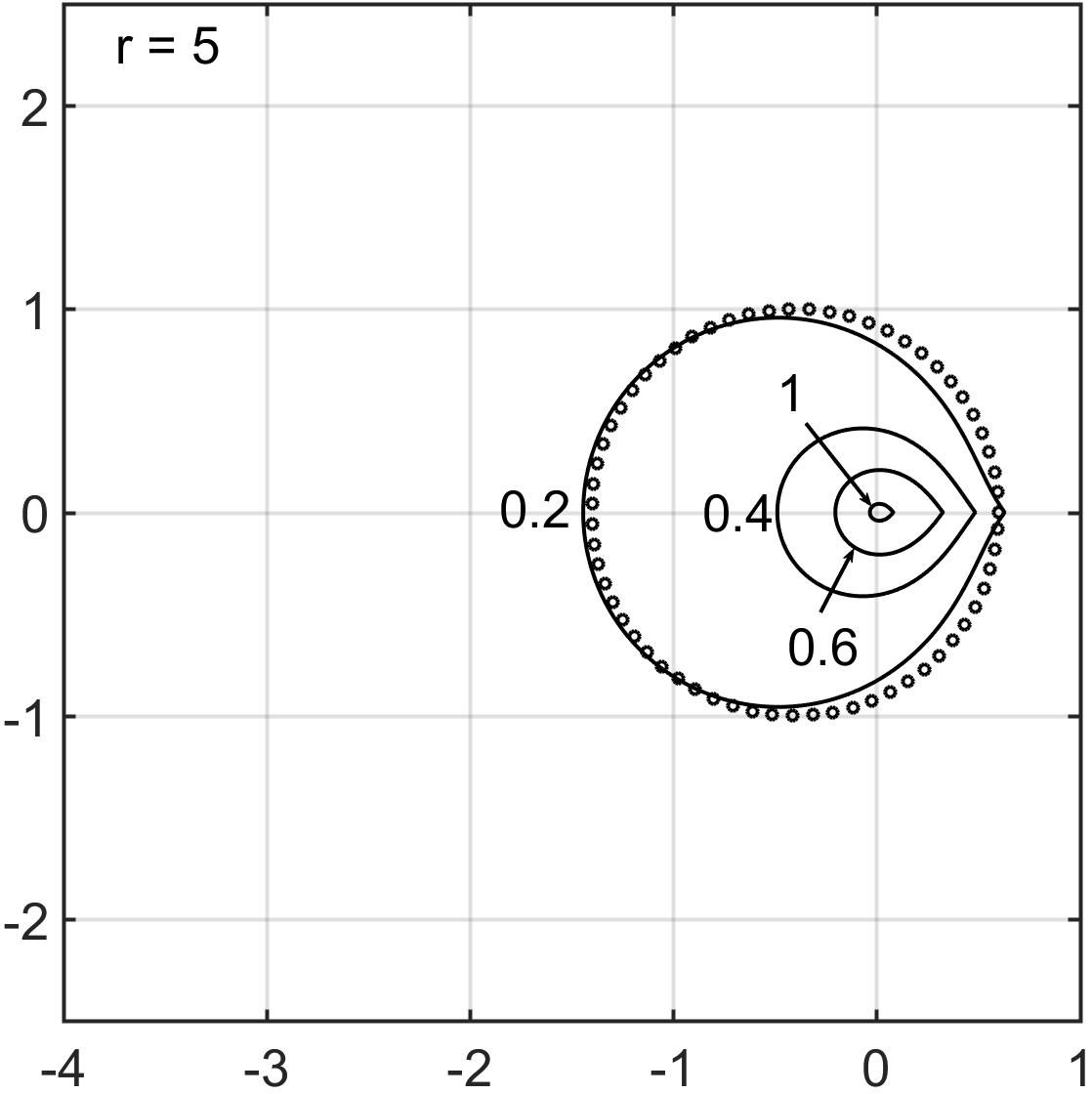}
\end{tabular}\vspace*{-0.75em}
\caption{Region of unconditional stability $\mathcal{D}$ for orders
 $r = 2, 3, 4, 5$.  In each sub-figure, the boundary $\partial\mathcal{D}$
 (solid line) is shown for parameter values $\param = 0.2, 0.4, 0.6, 1$
 ($\param = 1$ corresponds to SBDF).
 \\
 For small $\param \ll 1$, the stability region becomes arbitrarily large.
 With the exception of points near the positive real axis, it approaches
 the asymptotic circle $C$ defined in Theorem~\ref{CharacterizationUStab}. 
 The dots ($\circ$) show $C$ for $\param = 0.2$.}
\label{NewImExStabilityRegions}
\end{SCfigure}
%
\subsection{Necessary conditions for unconditional stability}
\label{SubSection_NecessaryConditions}
The sufficient conditions for unconditional stability 
$W_p \subseteq \mathcal{D}$, are not sharp, and we supplement 
them with additional necessary conditions. Let
\[ \sigma( (-\mat A)^{-1}\mat B) = \{ \mu \in \mathbb{C} : 
 \mu (-\mat{A}) \vec{u} = \mat{B} \vec{u}, 
 \vec{u} \neq \vec{0}\} \]
be the generalized eigenvalues of $(-\mat{A}), \mat{B}$. 
%
\begin{proposition}\label{NecCondition}(Necessary condition for unconditional
 stability)
	Given a set of \imex time stepping coefficients
    (Definition~\ref{NewImExCoeff}), and the corresponding stability diagram
    $\mathcal{D}$, then a necessary condition for unconditional stability of
    the scheme in (\ref{fulltimestepping}), is that the eigenvalues satisfy
    $\sigma( (-\mat A)^{-1}\mat B) \subseteq \mathcal{D} \cup \vGamma_{-\infty}$.
\end{proposition}
%
\begin{proof}
 The idea behind the necessary condition is that in the limit of
 large time steps $\delt \rightarrow \infty$, the nonlinear eigenvalue 
 problem (\ref{eqn_NLeigenvalue0}) governing stability, can be solved
 using the eigenvectors of the matrix $(-\mat A)^{-1}\mat B$. 
 As a result, a necessary condition for unconditional stability may 
 be placed on the eigenvalue spectrum
 $\mu \in \sigma((-\mat{A})^{-1}\mat{B})$.

 We first prove a slightly stronger statement. Let 
 \[
  \mathcal{A} := \{ \mu \in \mathbb{C} : (\ref{Simplified_Polynomial}) 
  \textrm{ has a solution } |z| > 1\}.
 \]
 Then $\sigma((-\mat{A})^{-1}\mat{B}) \subseteq \mathcal{A}^c$ is a necessary 
 condition for unconditionally stability.  This is because, in the limit 
 $\delt \rightarrow \infty$, the nonlinear eigenvalue problem
 (\ref{eqn_NLeigenvalue0}) becomes
 \begin{align}\label{ApproxEV}
  \mat{T}(z) \vec{u} = -c(z) \mat{A} \vec{u} - b(z) \mat{B} \vec{u} = 0.
 \end{align}
 Hence, an eigenvector $\vec{u}_{\mu}$ to $(-\mat A)^{-1}\mat B$ 
 with eigenvalue $\mu \in \sigma( (-\mat A)^{-1}\mat B)$ becomes 
 an eigenvector of (\ref{ApproxEV})
 \begin{align}\label{ApproxEVb}
  \mat{T}(z) \vec{u}_{\mu} = -(c(z) - \mu b(z) ) \mat{A} \vec{u}_{\mu} = 0.  
 \end{align}
 Thus the eigenvalues $z$ satisfy (\ref{Simplified_Polynomial}), 
 since $\mat{A}\vec{u}_{\mu} \neq 0$ because $\mat{A}$ is invertible.  
 If $\mu$ is also in $\mathcal{A}$, then at least one solution to 
 (\ref{ApproxEVb}) satisfies $|z| > 1$.  Finally, we note
 that any nonlinear eigenvalue $|z| > 1$, arising in
 the limit $\delt \rightarrow \infty$, will yield a slightly
 perturbed eigenvalue when $0 < \delt^{-1} \ll 1$.
 Thus, for any $\delt$ sufficiently large (but finite) an
 unstable eigenvalue satisfying $|z| > 1\/$ will exist.

 Finally we observe that $\mathcal{A}^c \subseteq 
 \mathcal{D} \cup \vGamma_{-\infty}$.
 The reason is that every $\mu \in \mathcal{A}^c$ has one of the following
 properties: 
 (i)  all solutions to (\ref{Simplified_Polynomial}) have 
      $|z| < 1$, implying $\mu \in \mathcal{D}$, or
 (ii) at least one solution to (\ref{Simplified_Polynomial}) has $|z| = 1$ 
      (with all the others $|z| < 1$), implying $\mu \in \vGamma_{-\infty}$.
\end{proof}
%
\begin{remark}\label{Nottight}
 Numerical experiments (such as the diagrams in
 Figure~\ref{Nested_stability_regions})	suggest that the set 
	$\mathcal{D} \cup \vGamma_{-\infty}$ in Proposition~\ref{NecCondition} can 
	be further reduced to include only the portion of $\vGamma_{-\infty}$ 
	that is the boundary $\partial \mathcal{D}$ and the single 
	point $\{1\}$. 	
\end{remark}
%
\begin{remark}\label{Limiting_D_Behavior}
 In the limit $\param \rightarrow 0$, $\mathcal{D}$ approaches the circle
 $C$ which encompasses an entire complex half-plane:
 \begin{align}
	\Big\{ \mu \in \mathbb{C} :  \mathrm{Re}(\mu) 
	< \frac{r+1}{2r}, 1 \leq r \leq 5 \Big\} 
	\subseteq \lim_{\param \rightarrow 0}\mathcal{D}.
 \end{align}
 The limiting $\mathcal{D}$ also contains the real half-line
 $(-\infty, (1 + \cos^r(\pi/r))^{-1})$, for $2 \leq r \leq 5$.
\end{remark}
%
\subsection{Numerical error dependence on $\param$ for the new \imex coefficients}\label{Sec_Error}
Up to now, the results appear to indicate that one should choose
$\param \ll 1$ (extremely small) to yield a large unconditional stability
region. In this section we describe why this is not a good strategy.  In
particular, we investigate the dependence of the 
\emph{global truncation error} (GTE) on $\param$ for the 
new \imex coefficients.  We do so by 
running numerical tests, and computing the \emph{error constants} 
which characterize the leading order asymptotic GTE behavior in $\delt$.

The GTE at time $t_n = n\delt$, is defined by 
$\| \vec{u}_n - \vec{u}^*(n\delt)\|_{\ell^\infty}$  and depends on $\mat{L}$, 
the time stepping coefficients, and the forcing $\vec{f}(t)$. Here
$\vec{u}^*(t)$ is the exact ODE solution to (\ref{Eqn_ODE}) at time $t$.
Formally, the new \imex schemes given in Definition~\ref{NewImExCoeff}
achieve $r$-th order 
accuracy, so that the $\textrm{GTE} = \mathcal{O}(\delt^{r})$.  The leading
order constant in the GTE depends on $\mat{A}, \mat{B}, \vec{f}$ and the 
time stepping coefficients (for error constants in an LMM see equation
(2.3), p.~373, in \cite{HairerNorsettWanner1987}).
In \imex schemes one may examine two 
separate error constants, an implicit $C_{I,r}$ (resp.\ explicit $C_{E,r}$) 
constant characterizing the error of a purely implicit (resp.\ explicit)
scheme where $\mat{B} = 0$ (resp.\ $\mat{A} = 0$):
\begin{align}\label{ErrConstants}
	C_{I,r} := \frac{R_{I,r}}{c(1)} = \param^{-r} R_{I,r}, \quad \quad
	C_{E,r} := \frac{R_{E,r}}{b(1)} = \param^{-r} R_{E,r}. 
\end{align}
Here we have used the fact that $c(1) = b(1) = \param^r$ for the new 
\imex schemes, while the constants $R_{I,r}, R_{E,r}$ quantify how 
much the $r$-th order coefficients (when $r = s$) fail to satisfy 
the $(r+1)th$ order conditions (\ref{OrderConditions})
\[
  R_{I, r} = \frac{1}{(r+1)!}
  \sum_{j = 0}^r\Big( a_j j^{r+1} - (r+1)c_j j^r\Big), \hspace{2mm}
  R_{E,r} = \frac{1}{(r+1)!} 
  \sum_{j = 0}^r\Big( a_j j^{r+1} - (r+1)b_j j^r\Big). 
\]
Even though $R_{I,r}, R_{E,r}$ depend on $\param$, both constants satisfy
$R_{I,r} = \mathcal{O}(1), R_{E,r} = \mathcal{O}(1)$, for all values of 
$0 < \param \leq 1$.  As a result, the asymptotic $\param \ll 1$ behavior 
on the GTE for the new \imex coefficients is 
$\textrm{GTE} = \mathcal{O}( \param^{-r} \delt^r)$.
	The numerical tests in \Srm\ref{Sec_Examples}, as well as those in
  \Srm\ref{Supp_gte_test}, confirm the estimate
  GTE $\sim \param^{-r} \delt^r$.
As a result of this scaling, we adopt the general philosophy: given a
splitting $(\mat{A}\/,\,\mat{B})$, choose $\param$ as large as possible
while maintaining unconditional stability.
%

\section{Two illustrative examples}\label{Sec_Examples}
In this section we highlight the potential of the new \imex coefficients to 
obtain unconditionally stable schemes. The first 
example (\Srm\ref{Subsec_A_smaller_B}) illustrates that a small implicit 
term can stabilize a larger explicit term. The second example 
(\Srm\ref{Subsec_var_coeff}) represents the numerical discretization of a 
variable coefficient diffusion equation. For this stiff problem, unconditional 
stability for orders $r > 2$ is beyond the capabilities of classical SBDF 
schemes; however, the new coefficients achieve the goal.

\subsection{A single variable ODE}\label{Subsec_A_smaller_B}
Consider the ODE
\begin{align}\label{Ex1_simpleODE}
	u_t = -10u = -u - 9u,
\end{align}
with splitting $\mat{A}u := -u$ and $\mat{B}u := -9u$.  For this simple
case, $(\mat{A},\mat{B})$ are numbers, or $1\times 1$ matrices.
An important observation is that $|\mat{A}| = 1$, while $|\mat{B}| = 9$,
i.e., the implicit term is 9 times smaller than the explicit term.

The set $W_1 = \{ -9 \}$ consists of one element, and it is also equal to the
generalized eigenvalue $\sigma( (-\mat{A})^{-1} \mat{B}) = \{-9\}$.
Therefore unconditional stability requires that $\{-9\} \subseteq \mathcal{D}$.
Using the fact that the left-most endpoint of $\mathcal{D}$ is given by
$m_l$ in formula in Remark~\ref{CharacterizationUStab}, one obtains
unconditional stability for an $r$-th order scheme, provided that
\[
	\frac{-(2-\param)^r}{2^r - (2-\param)^r} < -9, \quad \Longleftrightarrow \quad
	\param < 2\Big[ 1 - \big(\frac{9}{10} \big)^{1/r} \Big].
\]
Note that for a fixed $\param$ value, the unconditional stability regions
$\mathcal{D}$ become smaller with increasing $r$.  Setting $r = 5$ inside 
the inequality yields $\param < 0.0417$.  Therefore, a choice of the parameter
value $\param = 0.04$ inside the new \imex coefficients
guarantees that $W_1 \subseteq \mathcal{D}$ for $r = 5$, and hence
subsequently for all $1 \leq r \leq 5$.
Hence, the smaller implicit term stabilizes the instabilities generated by the explicit 
term, thus achieving unconditional stability.

\subsection{A PDE example: variable coefficient diffusion}\label{Subsec_var_coeff}

\begin{figure}[htb!]
\centering
  \begin{minipage}[b]{0.48\textwidth}
\includegraphics[width = 0.65\textwidth]{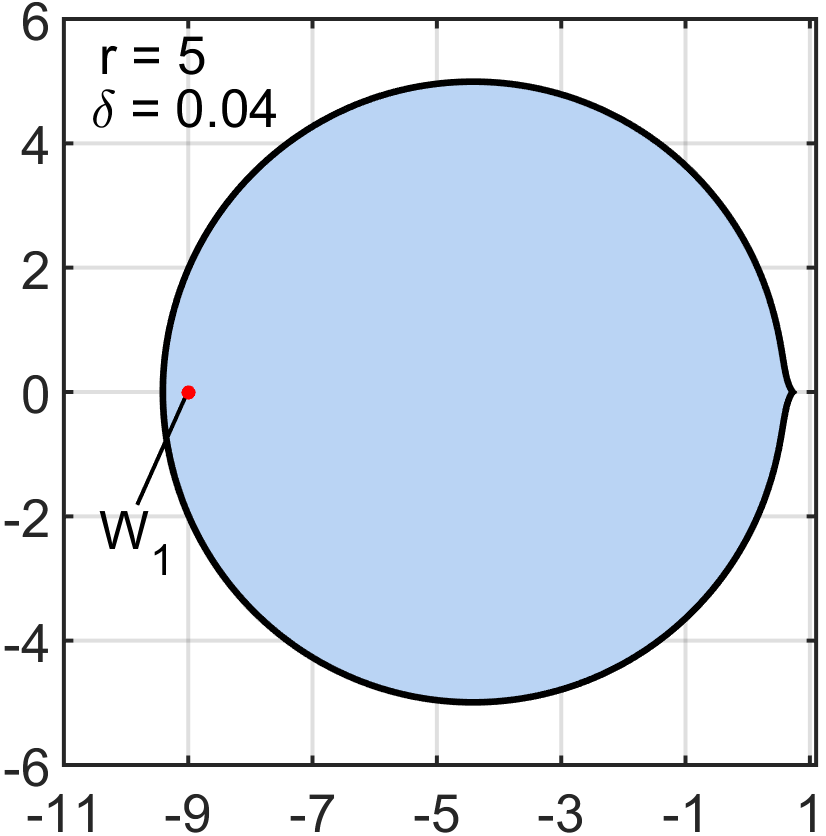}
\caption{ODE example (\ref{Ex1_simpleODE}). Note that $W_1$ (red $\circ$) is
contained in $\mathcal{D}$ (shaded region) for the parameter value
$\param = 0.04$ and order $r = 5$.  Using the new \imex 
coefficients with $\param = 0.04$ yields an unconditionally stable scheme.\\}
\label{Example_single_var_ODE}
  \end{minipage}
  \hfill
\begin{minipage}[b]{0.48\textwidth}   
		\includegraphics[width = 0.65\textwidth]{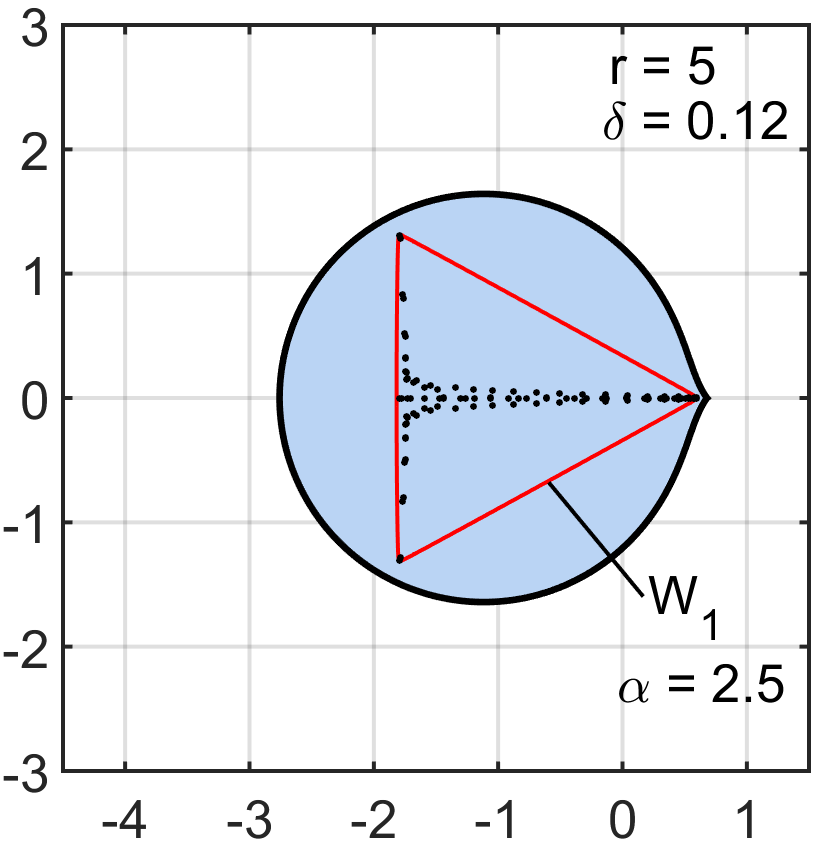}
\caption{PDE example for variable diffusion coefficient problem.
The stability diagram $\mathcal{D}$ (shaded region) is
for the parameter $\param = 0.12$ and order $r = 5$, and contains $W_1$ 
	(red curve shows the boundary). The black dots show 
	the generalized eigenvalues $\sigma( (-\mat{A})^{-1}\mat{B} )$.}
		\label{Example_var_coeff}
  \end{minipage}
\end{figure}

This example demonstrates how one might use the new \imex
coefficients, in conjunction with the sufficient conditions for unconditional
stability, to avoid a stiff time step restriction in
the spatial discretization of a PDE.   Specifically,
we numerically solve the variable coefficient diffusion equation
on the domain $\vOmega = (-1, 1)$:
\[
	u_t = \big( d(x) u_x \big)_x + f(x, t),
	\quad \quad \textrm{on } \vOmega \times (0, T],
\]
with Dirichlet boundary conditions,
$u = 0$, on $x \in \{-1, 1\}$.  Here $d(x) > 0$ is a spatially
dependent diffusion coefficient.

For the spatial discretization, we adopt a Chebyshev spectral method
(chapter 5--7, \cite{Trefethen2000}) using the $\sizeN + 2$
Chebyshev collocation points\footnote{Note that here the points 
$1 = x_0 > x_1 > \dotsb > x_{\sizeN + 1} = -1$ are in \emph{reverse} 
order, following the usage in \cite{Trefethen2000}.}
\[ x_j = \cos\Big( \frac{j \pi}{\sizeN + 1} \Big), \quad 0 \leq j \leq \sizeN + 1,
	\quad \textrm{with } \quad
	\vec u = \begin{pmatrix}
		u(x_1), \ldots, u(x_{\sizeN})
	\end{pmatrix}^T \in \mathbb{R}^{\sizeN}.\]
We also use the boundary conditions to set $u(x_0) = u(x_{\sizeN + 1}) = 0$ so
that there are only $\sizeN$ independent variables.  Let $\mat{D}_{\sizeN}$ 
be the spectral differentiation matrix, so that
$\mat{D}_{\sizeN} \vec{u} \approx u_x(x)$.
The matrix $\mat{L}$ is then built using the Dirichlet boundary
conditions by constructing
\[
	\mat{L} = \mat{D}_{\sizeN} \;
	\textrm{diag}\big( d(x_j) \big) \; \mat{D}_{\sizeN}.
\]
Here $\mat{L}$ acts on $\vec{u}$ at the $\sizeN$ grid
points $x_1, x_2, \ldots, x_{\sizeN}$ (see page 62, chapter 7 in 
\cite{Trefethen2000} for details).  
Note that due to collocation of the boundary conditions,
the matrix $\mat{L}$, as well as the Laplacian $(\mat{D}_{\sizeN})^2$,
are not symmetric.  However, the spectrum of $\mat{L}$ and $\mat{D}_{\sizeN}^2$
are still purely real, in contrast with the situation in truly asymmetric 
problems, such as advection-diffusion.

In practice, the semi-implicit time stepping of (\ref{fulltimestepping}),
using the schemes defined by Definition~\ref{NewImExCoeff},
requires both a choice of splitting $(\mat{A}, \mat{B})$ and a set of new
\imex coefficients fixed by a choice of $\param$.  For this example we 
consider a splitting where $\mat{A}$ is a scalar multiple of the symmetrized 
part of the discrete, spectral Laplacian:
\[
	\mat A = \frac{ \alpha }{2} \Big( (\mat{D}_{\sizeN})^2 + (\mat{D}_{\sizeN}^2)^T\Big),
	\quad
	\mat B = \mat{L} - \mat{A}
\]
with an $\alpha > 0$. Here the choice of $\mat{A}$
is negative definite and symmetric.

It is worth noting that in general,
$\mat{A}$ and $\mat{B}$ do not commute, therefore motivating the use of the 
new unconditional stability criteria. 
For this class of splittings, we focus on using the generalized numerical
range $W_1$.  The reason is that the size and shape of $W_1$ depends only
very weakly on $\sizeN$ for large $\sizeN$.  

There are now two free variables to choose: (i) $\alpha$, which fixes the relative
splitting of the (symmetric) implicit Laplacian to the explicit variable diffusion,
and (ii) $\param$, which fixes the \imex coefficients.  Ideally, one would like to
simultaneously choose $\alpha$ and $\param$ to obtain unconditional stability and
also minimize the overall error in the scheme.  We defer a detailed discussion
on how one may minimize the error for a companion paper on practical aspects
of unconditional stability. Here we state briefly how one may first choose
$\alpha$, followed by $\param$ to obtain unconditional stability.

Decreasing $\alpha$ moves the set
$W_1$ left in the complex plane --- into a region that may be stabilized by the
new \imex coefficients.  Specifically, we choose $\alpha$ small enough so that
the right-most point of $W_1$ is pushed to the left of the right-most
point of the limiting set $\mathcal{D}$ (see Remark~\ref{Limiting_D_Behavior} 
for the right-most point of $\mathcal{D}$).
Once $W_1$ is sufficiently far left, we choose a sufficiently small $\param$
value to ensure that $W_1 \subseteq \mathcal{D}$.
To compute $W_1$, we first build the matrix
$\mat X = (-\mat A)^{-\frac{1}{2} } \mat \;
		\mat B \; (-\mat A)^{-\frac{1}{2}}$,
followed by using the MATLAB \emph{Chebfun} routine	
\cite{DriscollHaleTrefethen2014} to	compute $W_1 = W(\mat X)$, based on a 
classical algorithm due to Johnson \cite{Johnson1978}.

Finally, we perform a convergence test using the variable diffusion coefficient
\begin{align*}
	d(x) &= 4 + 3 \cos(2\pi x).
\end{align*}

Figure~\ref{Example_var_coeff} shows the set $W_1$ for a variable
coefficient $d(x)$ and a value of $\alpha = 2.5$.  In addition, the
figure also shows a plot of the enclosing stability region
$\mathcal{D}$ for order $r = 5$ and the parameter value $\param = 0.12$.
Note that the unconditional stability region $\mathcal{D}$
becomes smaller as the order $r$ increases, so that $\param = 0.12$
 automatically guarantees unconditional stability for all orders
$1 \leq r \leq 5$.
For a convergence test, we use a manufactured solution approach and prescribe
a forcing function $f(x,t)$ to yield an exact solution:
\[	u^*(x, t) = \sin(20 t) \sin(2\pi x) e^{\sin(2\pi x)}.\]
The numerical test case is also chosen to satisfy the \emph{exact} initial data:
$\vec{u}_j = \vec{u}^*(x, j \delt)$ evaluated at the grid points, for 
$j = 0, -1, \ldots, -r+1$.  Table~\ref{VariableDiff_CvgTest3} shows the absolute $L^{\infty}(\vOmega)$ errors
for an integration time $t_f = 1$ and grid $\sizeN = 100$.  Convergence rates
for $1 \leq r \leq 5$ are observed as expected. Computations are done using
MATLAB with double precision floating point
arithmetic. Errors are limited to $10^{-9}$ for $r = 5$ due to machine
precision and round off errors.

\begin{remark}
	An important observation is that  
	the set $W_1$ remains bounded as $\sizeN \rightarrow \infty$. 
	This result is of great practical relevance: one 
	fixed value of $\param$ can yield a stability region that contains 
	$W_1$ for arbitrary $\sizeN$.  For instance, the 
	convergence results in Table~\ref{VariableDiff_CvgTest3} are all 
	computed using the same value of $\param$. 
	Therefore, the new time stepping schemes can be advantageous in PDE
	applications where the parameter $\param$ can be chosen for a particular
	splitting of the differential operators; and hold uniformly for
	any level of discretization of those operators (i.e., for a whole
	family of matrix splittings).	
\end{remark}
This example can be seen as a blueprint for many practical applications: the 
implicit part is simple and efficient to solve for (symmetric, constant 
coefficient), and the new \imex coefficients enable one to obtain a numerical 
approximation that is unconditionally stable, thus avoiding diffusive-type 
time step restriction associated with explicit methods.

\begin{table}[htb!]
\centering
\begin{small}
\begin{tabular}{ |@{~} c@{~} ||@{~} c@{~} |@{~} c@{~} ||@{~} c@{~} | @{~}c @{~}||@{~} c@{~} |@{~} c@{~} ||@{~} c@{~} |@{~} c@{~} ||@{~} c@{~} |@{~} c@{~} |}
\hline
 $\delt$ & Error & Rate & Error & Rate & Error & Rate & Error & Rate & Error & Rate  \\ \hline
         & $r = 1$ &  & $r = 2$ &  & $r = 3$ &  & $r = 4$ &  & $r = 5$ &  \\ \hline
$2^{-6}$ 	 & 	  2.1e+00 	 & 	 0.4 &	  1.4e+00 	 & 	 0.5 &	  1.0e+00 	 & 	 1.8 &	  1.9e+00 	 & 	 4.4 &	  4.0e+00 	 & 	 5.9 \\
$2^{-7}$ 	 & 	  1.3e+00 	 & 	 0.7 &	  7.6e-01 	 & 	 0.9 &	  4.4e-01 	 & 	 1.2 &	  4.2e-01 	 & 	 2.2 &	  6.8e-01 	 & 	 2.6 \\
$2^{-8}$ 	 & 	  7.0e-01 	 & 	 0.9 &	  1.8e-01 	 & 	 2.1 &	  2.4e-01 	 & 	 0.9 &	  1.5e-01 	 & 	 1.5 &	  1.9e-02 	 & 	 5.2 \\
$2^{-9}$ 	 & 	  3.6e-01 	 & 	 1.0 &	  7.3e-02 	 & 	 1.3 &	  5.1e-02 	 & 	 2.2 &	  3.8e-03 	 & 	 5.3 &	  4.8e-03 	 & 	 2.0 \\
$2^{-10}$ 	 & 	  1.8e-01 	 & 	 1.0 &	  3.0e-02 	 & 	 1.3 &	  5.8e-03 	 & 	 3.1 &	  5.5e-04 	 & 	 2.8 &	  1.8e-04 	 & 	 4.7 \\
$2^{-11}$ 	 & 	  8.2e-02 	 & 	 1.1 &	  8.8e-03 	 & 	 1.8 &	  6.0e-04 	 & 	 3.3 &	  5.4e-05 	 & 	 3.4 &	  4.7e-06 	 & 	 5.3 \\
$2^{-12}$ 	 & 	  3.9e-02 	 & 	 1.1 &	  2.3e-03 	 & 	 1.9 &	  6.7e-05 	 & 	 3.2 &	  3.9e-06 	 & 	 3.8 &	  1.2e-07 	 & 	 5.3 \\
$2^{-13}$ 	 & 	  1.9e-02 	 & 	 1.0 &	  6.0e-04 	 & 	 2.0 &	  7.9e-06 	 & 	 3.1 &	  2.6e-07 	 & 	 3.9 &	  3.7e-09 	 & 	 5.0 \\
 \hline
\end{tabular}
\end{small}
	\caption{Errors for variable coefficient diffusion test case
	$\alpha = 2.5$, $\param = 0.12$,
	$t_f = 1$, $\sizeN = 100$. Exact solution
	$u^* = \sin(20 t) \sin(2\pi x) e^{\sin(2\pi x) } $. Note that an explicit
	scheme, such as explicit Euler, would require a time step restriction 
	$\mathcal{O}(\sizeN^{-2}) \sim 10^{-4} \sim 2^{-13}$.  Here unconditional
	stability allows one to choose a time step based solely on accuracy 
	considerations.}
	\label{VariableDiff_CvgTest3}
\end{table}


\section{Discussion and conclusions}
We have introduced a stability region $\mathcal{D}$, along with
a generalized numerical range, as a way to guarantee unconditional stability
for \imex LMMs with a negative definite implicit term. It should be
stressed that this type of study of unconditional stability is, structurally,
not limited to \imex LMMs and can also be examined in the context of
any other time stepping scheme, such as RK methods, exponential integrators,
deferred correction, or Richardson extrapolation. Moreover, unconditional
stability (and further generalizations of $\mathcal{D}$) can in principle be
examined also when the implicit term is not symmetric negative
definite, such as for stiff wave problems.

In addition to sufficient criteria for unconditional stability we have also
introduced a family of \imex LMM coefficients, parameterized
by $0 < \param \leq 1$ (which reduce to classical SBDF when $\param = 1$).
This parameter $\param$ incurs crucial implications for stability, 
and the examples
in \Srm\ref{Sec_Examples} highlight how the new \imex coefficients can yield
highly efficient time-stepping schemes.


In light of these substantial advantages, three points of caution have to be stressed:
\begin{enumerate}[(a)]
	\item The error constant for an $r$-th order method scales as
	$\param^{-r}$.
	\item Computations with $\param \ll 1$ may substantially
    amplify round-off errors.
	\item L-stability, or small growth factors, are desirable properties
    for stiff equations, and lost for $\param<1$.
	If one uses the new \imex coefficients as a fully implicit
	scheme (i.e., choosing $\mat{A} := \mat{L}$, $\mat{B} = 0$), then
	stability of the test equation $u_t = \lambda u$ is characterized by
	roots of the polynomial $a(z) - \delt \lambda c(z) = 0$.
	In the limit $\delt \rightarrow \infty$, the roots approach
	$\zeta := 1-\param$ (repeated $r$ times).  L-stability is only attained
	when the roots $\zeta$ have $\param = 1$, corresponding to SBDF.
	Moreover, if $\param \ll 1$, then the growth factor $1-\param$ is close to $1$,
    implying that stiff modes may require many time steps to decay.
\end{enumerate}
To conclude, major drawbacks of the new \imex
schemes are incurred only if $\param \ll 1$. In practice, a moderate $\param$
value (for instance
$\param \sim 0.1$) is frequently sufficient to stabilize a matrix
splitting. In such a case the debilitating drawbacks of the new coefficients
pale in comparison to the alternative of having to use a stiff time
step restriction.


\section{Tables of new \imex coefficients}\label{Supp_coeff_tables}

This section presents the new \imex coefficients $(a_j, b_j, c_j)$ 
for $0 \leq j \leq r$, as a function of $0 < \param \leq 1$.  
To use the coefficients 
in practice, first (i) choose a small enough value of $\param$ 
that guarantees unconditional stability, (ii) substitute the chosen
value of $\param$ into the tables in this section to obtain the 
time stepping coefficients at the required order.  

\setlength{\mylength}{0.36\textwidth + \tabcolsep + \arrayrulewidth}%

\medskip
\noindent
\begin{small}
\begin{tabular}{ |@{~}>{\centering\arraybackslash}p{0.08\textwidth} @{~} |@{~}>{\centering\arraybackslash} p{0.03\textwidth} @{~} || @{~}>{\centering\arraybackslash}p{0.18\textwidth} @{~}| @{~}>{\centering\arraybackslash}p{0.18\textwidth} @{~}|@{~} >{\centering\arraybackslash}p{0.18\textwidth} @{~} |@{~} >{\centering\arraybackslash}p{0.18\textwidth} @{~} |}
\hline
	Order &  & $j = 3$ & $j = 2$ & $j = 1$ & $j = 0$  \\ \hline
	1 & $a_j$  & . &  . & $\param$ & $-\param$ \\ 
	  & $c_j$  & . &  . & 1 & ($\param$-1) \\ 
	  & $b_j$  & . &  . & 0 & $\param$ \\ \hline \hline
    2 & $a_j$  & . & $2\param - \frac{1}{2}\param^2$ & $-4\param + 2\param^2$ & $2\param- \frac{3}{2}\param^2$ \\ 
      & $c_j$  & . & 1 & $2(\param - 1)$ & $(\param - 1)^2$ \\ 
      & $b_j$  & . & 0 & $2\param$ & $(\param - 1)^2 - 1$ \\ \hline \hline
    3 & $a_j$  & $3\param - \frac{3}{2}\param^2 + \frac{1}{3}\param^3$ & $-9\param + \frac{15}{2}\param^2 -\frac{3}{2}\param^3$ & $9\param - \frac{21}{2}\param^2 + 3\param^3$ & $-3\param + \frac{9}{2}\param^2 - \frac{11}{6}\param^3$ \\ 
      & $c_j$  & 1 & $3(\param-1)$ & $3(\param - 1)^2$ & $(\param - 1)^3$ \\ 
      & $b_j$  & 0 & $3\param$ & $-6\param + 3\param^2$ & $(\param - 1)^3 + 1$ \\ \hline 
\end{tabular}
\end{small}

\medskip
\medskip

\noindent
\begin{small}
\begin{tabular}{ |@{~}>{\centering\arraybackslash}p{0.08\textwidth} @{~} |@{~} >{\centering\arraybackslash}p{0.03\textwidth} @{~} || @{~}>{\centering\arraybackslash} p{\mylength} @{~}| @{~}>{\centering\arraybackslash}p{\mylength} @{~}|}
\hline
	Order &        &  		 & $j = 4$   \\ \hline
	4     & $a_j$  &  	.	 & $4\param - 3\param^2 + \frac{4}{3}\param^3 - \frac{1}{4}\param^4$ \\
	      & $c_j$  &  	.	 &  1 		 \\
	      & $b_j$  &  	.    &  0 		 \\ \hline \hline
	      &        & $j = 3$ & $j = 2$   \\ \hline 
	      & $a_j$  & $-16\param + 18\param^2 - \frac{22}{3}\param^3 + \frac{4}{3}\param^4$ & $24\param - 36\param^2 + 18\param^3 -3\param^4$  \\ 
	      & $c_j$  & $4(\param-1)$ 	& $6(\param - 1)^2$ 		\\ 
      	  & $b_j$  & $4\param$ 		& $-12\param + 6\param^2$  	\\ \hline \hline
	      &        & $j = 1$ & $j = 0$   \\ \hline 
	      & $a_j$  & $-16\param + 30\param^2 - \frac{58}{3}\param^3 + 4\param^4$ & $4\param - 9\param^2 + \frac{22}{3}\param^3 - \frac{25}{12}\param^4$ \\     
      	  & $c_j$  & $4(\param - 1)^3$ & $(\param - 1)^4$ \\ 
      	  & $b_j$  & $12\param - 12\param^2 + 4\param^3$  & $(\param - 1)^4  - 1$ \\ \hline	      
\end{tabular}
\end{small}

\medskip
\medskip
\noindent
\begin{small}
\begin{tabular}{ |@{~}>{\centering\arraybackslash} p{0.08\textwidth} @{~} |@{~} >{\centering\arraybackslash}p{0.03\textwidth} @{~} || @{~} >{\centering\arraybackslash}p{\mylength} @{~}| @{~}>{\centering\arraybackslash}p{\mylength} @{~}|}
\hline
	Order &        & $j = 5$ & $j = 4$   \\ \hline
    5     & $a_j$  & $5\param - 5\param^2 + \frac{10}{3}\param^3 - \frac{5}{4}\param^4 + \frac{1}{5}\param^5$ & $-25\param + 35\param^2 - \frac{65}{3}\param^3 + \frac{95}{12}\param^4 - \frac{5}{4}\param^5$\\ 
          & $c_j$  & 1 & $5(\param - 1)$\\ 
          & $b_j$  & 0 & $5\param$ \\ \hline \hline          
	 	  &        & $j = 3$ & $j = 2$   \\ \hline
          & $a_j$  &  $50\param - 90\param^2 + \frac{190}{3}\param^3 - \frac{65}{3}\param^4 + \frac{10}{3}\param^5$ & $- 50\param + 110\param^2 - \frac{280}{3}\param^3+ 35\param^4 -5\param^5  $\\ 
          & $c_j$  & $10(\param - 1)^2$ & $10(\param - 1)^3$ \\ 
          & $b_j$  & $-20\param + 10\param^2$  & $30\param + 10\param^3 - 30\param^2$  \\ \hline \hline          
	      &        & $j = 1$ & $j = 0$   \\ \hline
          & $a_j$  & $25\param - 65\param^2 + \frac{200}{3}\param^3 - \frac{365}{12}\param^4 + 5\param^5$ & $- 5\param + 15\param^2 - \frac{55}{3}\param^3 + \frac{125}{12}\param^4 - \frac{137}{60}\param^5 $ \\ 
          & $c_j$  & $5(\param - 1)^4$ & $(\param - 1)^5$ \\ 
          & $b_j$  & $-20\param  + 30\param^2 - 20\param^3+ 5\param^4$ & $(\param - 1)^5 + 1$\\ \hline    
    \end{tabular} 
\end{small}

\bigskip

{\bf Acknowledgments:}
	The authors wish to acknowledge support by the
National Science Foundation through grants DMS--1318709 (Seibold and
Zhou) and DMS--1318942 (Rosales); as well as partial support through grants
DMS--1719637 (Rosales), DMS--1719693 (Shirokoff) and DMS--1719640 (Seibold
and Zhou).  D. Shirokoff was supported by a grant from the Simons
Foundation ($\#359610$).

\appendix

\section{Properties of $W(\mat X)$}\label{Supp_properties_W}
For completeness, we 
list, without proof (see chapter 1, \cite{HornJohnson1991} for a general treatment), 
several well-known properties of the numerical 
range.  Denote the spectrum (set of all eigenvalues) of $\mat{X} \in \mathbb{C}^{\sizeN \times \sizeN}$ as
\begin{align}\label{Spectrum}
	\sigma(\mat{X}) := \{ \lambda \in \mathbb{C} : 
\mat{X}\vec{v} = \lambda \vec{v}, \vec{v} \neq \vec{0}\},
\end{align}
and numerical range as
\begin{align} 
 W(\mat X) := \{ \langle \vec x, \mat X \vec x \rangle :
            \| \vec x \| = 1, \vec x \in \mathbb{C}^{\sizeN}\}\/. 
\end{align}
Then the following hold:
\begin{enumerate}
	\item $W(\mat X) \subset \mathbb{C}$ is a closed and bounded subset of 
	the complex plane.
	\item $\sigma(\mat{X}) \subseteq W(\mat X)$. 
	\item $W(\mat X + \mat Y) \subseteq W(\mat X) + W(\mat Y)$, where 
	$W(\mat X) + W(\mat Y) = \{ x + y : x \in W(\mat X), y \in W(\mat Y)\}$.
	\item $W(\alpha \mat I + \beta \mat X) = \alpha + \beta \; W(\mat X)$, 
	where $\mat I$ is the $\sizeN \times \sizeN$ identity matrix and 
	$\alpha, \beta \in \mathbb{C}$.
	\item If $\mat X$ is normal, then $W(\mat X)$ is the convex hull of the 
	eigenvalues of $\mat X$.
	\item $W( \mat{X} \oplus \mat{Y}) = \textrm{conv} \{W(\mat{X}), W(\mat{Y}) \}$ 
	is the convex hull of $W(\mat{X})$ and $W(\mat{Y})$.
	\item (Hausdorff-Toeplitz theorem) $W(\mat X)$ is convex (even when 
	$\mat X$  is not normal).
	\item $\max |W(\mat X)|$ defines a matrix norm.
  (the numerical radius).
\end{enumerate}
Property (2) implies that for any $p \in \mathbb{R}$, one has 
$\sigma( (-\mat{A})^{-1}\mat{B} )\subseteq W_p$.

\section{Verification of Proposition~\ref{Order_conditions}}\label{Supp_Prop_Order_conditions}

This section discusses the verification of Proposition~\ref{Order_conditions} 
regarding the order conditions and zero-stability
for the new \imex coefficients.  For completeness we include the formulas for the 
order conditions and also the definition of zero-stability. 

 For an $r$-th order method ($r \leq s\/$), the $3s+2$ \imex coefficients cannot
   be independent and must satisfy the order conditions:
\begin{align} \label{OrderConditions} 
 &\sum_{j = 0}^s a_j = 0\/, \hspace{10mm}
  \sum_{j = 0}^s j a_j = \sum_{j = 0}^s c_j = \sum_{j = 0}^s b_j\/, \hspace{10mm}
  \sum_{j = 0}^s \frac{j^2}{2} a_j = \sum_{j = 0}^s j c_j =
         \sum_{j = 0}^s j b_j\/,  \\ \nonumber
 &\ldots \quad \sum_{j = 0}^s \frac{j^r}{r!} a_j =
  \sum_{j = 0}^s \frac{j^{r-1}}{(r-1)!} c_j =
  \sum_{j = 0}^s \frac{j^{r-1}}{(r-1)!} b_j\/.
\end{align}
The formulas (\ref{OrderConditions}) then impose $2r + 1$ linear constraints 
on the coefficients and agree with the ones in 
\cite{AscherRuuthWetton1995}.

For completeness we recall here the
definition for zero stable schemes
\begin{definition}\label{Def_ZeroStability}
 (Zero stability) The scheme (\ref{eqn_NLeigenvalue0}) is zero stable if
 every simple solution to $a(\zeta) = 0\/$ satisfies
    $|\zeta| \leq 1\/$, and every repeated solution satisfies
    $|\zeta| < 1$.
\end{definition}

\begin{proof}
We lack an analytic proof of zero-stability. However,
   in Figure 1 we plot the complex roots $z$ to $a(z)=0$
   for orders $2 \leq r \leq 5$.  The plot shows that $a(z)$ has 
   $r-1$ distinct roots strictly within the unit circle (one root at 
   $z = 1$) for all $0 < \delta \leq 1$, indicating that the schemes are 
   zero-stable.

To show that $(a_j, b_j, c_j)$ satisfy the order conditions, consider first 
fixing a set of coefficients $c_j$ via (\ref{NewImex_c}).  Consistency 
requires that the local truncation error for the \imex scheme after 
one time step $\delt$ be $\mathcal{O}( \delt^{r+1})$. 
In other words  (Theorem 2.4, pg. 370, \cite{HairerNorsettWanner1987})
there is a root $z_{\delt}$ to
\begin{align} \label{Implicit_Root}
	a( z_{ \delt}) = \delt c(z_{ \delt}), \textrm{ satisfying }  
	z_{ \delt} = e^{ \delt}+\mathcal{O}( \delt^{r+1}), \textrm{ when }  
	\delt \rightarrow 0.
\end{align}
Or equivalently, letting $z = e^{ \delt}$:
\[ a(z) = (\ln z) c(z) + \mathcal{O}\big( (\ln z)^{r+1}\big), 
\textrm{ when } z\rightarrow 1. \]
In the above equation, the polynomial $a(z)$ is of degree $r$ and must agree 
with $f(z) = (\ln z) c(z)$ to order $r$ near $z = 1$. Therefore $a(z)$ is 
the $r$-th order Taylor polynomial of $f(z)$ about $z = 1$. 

Regarding $b(z)$, the order conditions inductively imply that $c(1)=  b(1)$, 
$c'(1) = b'(1)$, $\ldots$, $c^{(r)}(1) = b^{(r)}(1)$. Hence, $c(z) - b(z)$ is 
a polynomial with $z = 1$ as an $r$-th repeated root so that 
$c(z) - b(z) \propto (z-1)^r$.  For $b(z)$ to define an explicit scheme, the 
degree $b(z) < $ degree $c(z) = r$. Therefore the proportionality constant 
must be $1$ so that  $b(z) = c(z) - (z-1)^r$.
\end{proof}

\begin{figure}[htb!]
\centering
\includegraphics[width = 0.4\textwidth]{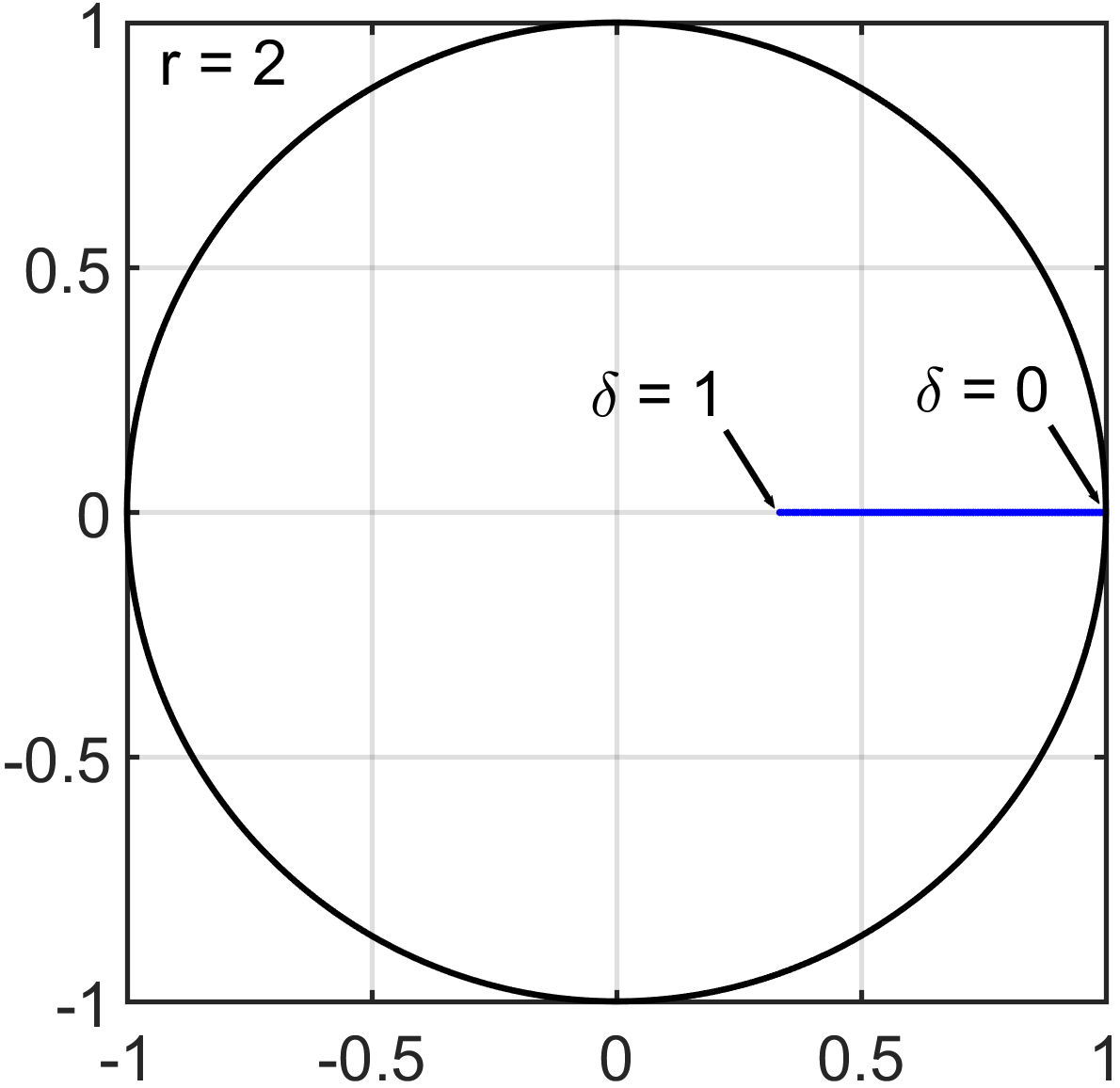} 
\includegraphics[width = 0.4\textwidth]{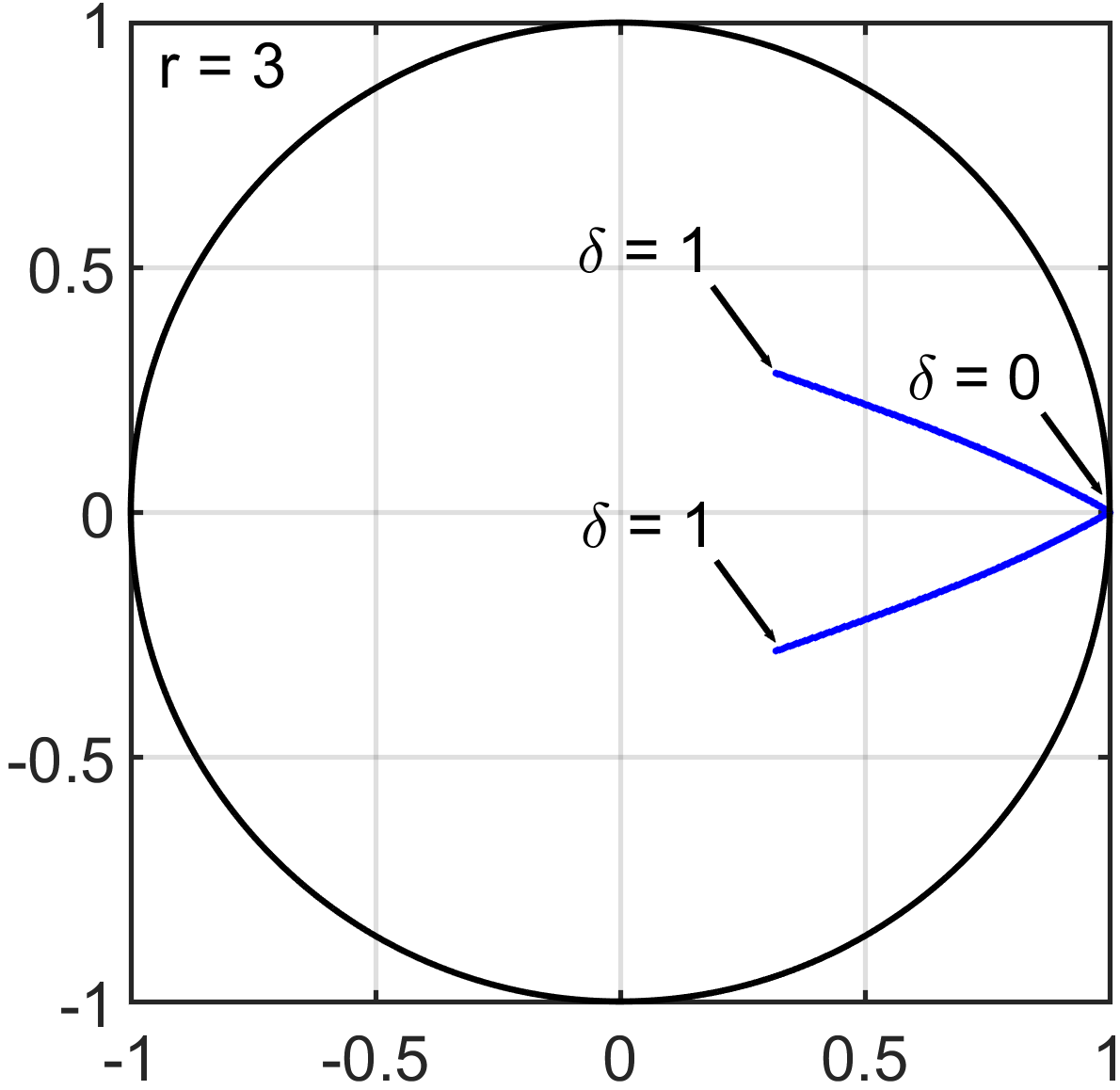} \\
\includegraphics[width = 0.4\textwidth]{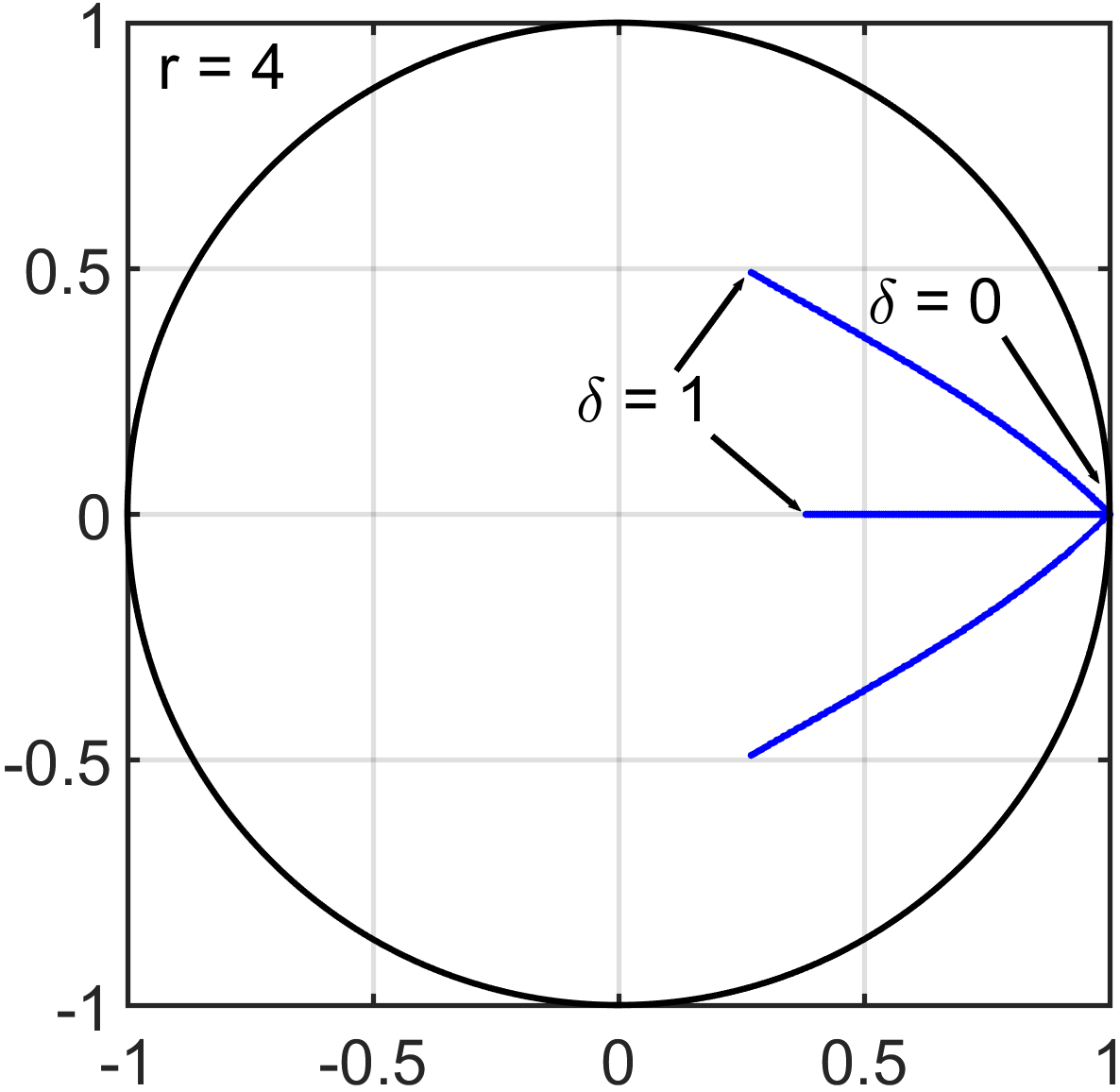} 
\includegraphics[width = 0.4\textwidth]{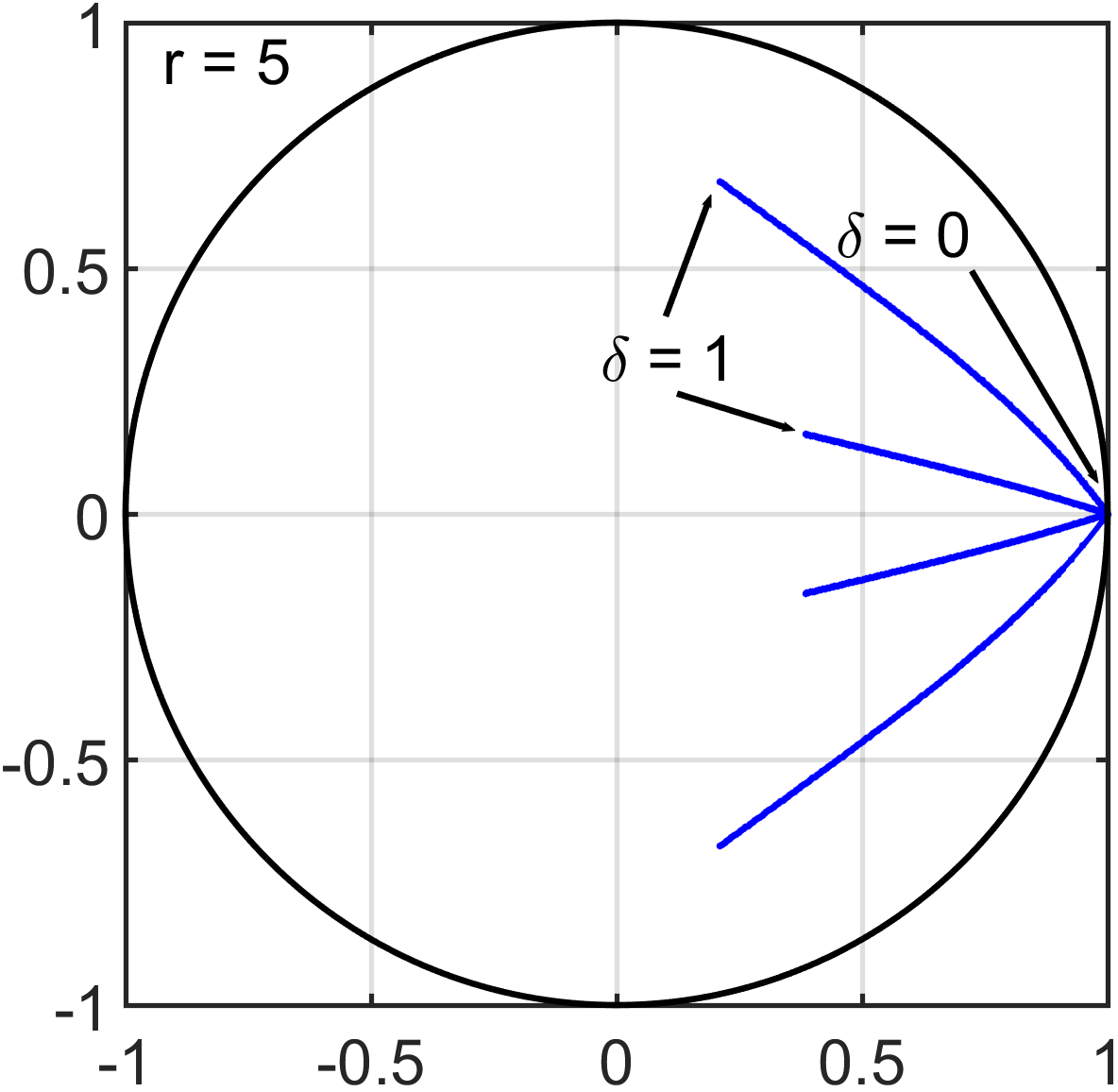} 
\caption{Zero stability for the new \imex coefficients defined by
 (\ref{NewImex_c})--(\ref{NewImex_a}), and orders $r = 2$--$5\/$. 
 The blue curve shows the roots of $a(z)$ in the complex plane for 
 different $\param$.
 Note $z = 1$ is a root for all $\param$ since $a(1) = 0$. The remaining
 $r-1$ roots remain strictly inside the unit circle for all
 $0 < \param \leq 1$, and approach $z = 1$ as $\param \rightarrow 0$.}
 \label{Zero_Stability}
\end{figure}

\section{Numerical test: Global Truncation Error constant}\label{Supp_gte_test}
To numerically examine the 
$\param$-dependence on the error, we compute the GTE of the test ODE:
\begin{align}\label{GTE_Test}
	u_t = -u, 
\end{align}	
where $\mat Au = -u$ (and $\mat Bu = 0$). We take a final integration time 
$t = 1$, $u^*(1) = e^{-1}$.  We perform tests with two different sets of 
time steps $\delt$: 
\begin{itemize}
	\item Table~\ref{GTE_Table1} shows the GTE (scales as $\delta^{-r}$), and 
	error rates for a fixed $\delt = 10^{-3}$ and variable $\param$ values.  
	\item Table~\ref{GTE_Table2} shows the GTE for different $\param$ values, 
	with a time step $\delt = \frac{1}{5} \param$ taken as a fixed
	fraction of $\param$.  The table 
	shows that the GTE is approximately constant over decreasing $\param$,
	and decreases by 
	$\big(\frac{\delt}{\param}\big)^r = \big(\frac{1}{5}\big)^r$ with the order. 
\end{itemize}


\begin{table}[htb!]
\centering
\begin{footnotesize}
\begin{tabular}{ |@{~} c @{~}||@{~} c @{~}|@{~} c@{~} ||@{~} c@{~} |@{~} c@{~} ||@{~} c@{~} |@{~} c@{~} ||@{~} c@{~} |@{~} c@{~} || @{~}c @{~}|@{~} c@{~} |}
\hline
 $\param$ & $r= 1$ & Rate & $r= 2$ & Rate & $r= 3$ & Rate & $r= 4$ & Rate & $r = 5$ & Rate \\ \hline 
$2^{0}$ 	 & 1.839e-04 	 & - 	& 1.227e-07 	 & - 	& 9.203e-11 	 & -	& 7.370e-14 	 & -	& 1.030e-17 	 & - \\ 
$2^{-1}$ 	 & 5.514e-04 	 & -1.58 	& 8.587e-07 	 & -2.81 	& 1.381e-09 	 & -3.91 	& 2.284e-12 	 & -4.95 	& 2.304e-15 	 & -7.81  \\ 
$2^{-2}$ 	 & 1.285e-03 	 & -1.22 	& 4.539e-06 	 & -2.40 	& 1.611e-08 	 & -3.54 	& 5.754e-11 	 & -4.65 	& 1.663e-13 	 & -6.17  \\ 	
$2^{-3}$ 	 & 2.749e-03 	 & -1.10 	& 2.073e-05 	 & -2.19 	& 1.560e-07 	 & -3.28 	& 1.175e-09 	 & -4.35 	& 8.027e-12 	 & -5.59  \\ 
$2^{-4}$ 	 & 5.658e-03 	 & -1.04 	& 8.838e-05 	 & -2.09 	& 1.371e-06 	 & -3.14 	& 2.126e-08 	 & -4.18  	& 3.138e-10 	 & -5.29  \\ 
$2^{-5}$ 	 & 1.141e-02 	 & -1.01 	& 3.637e-04 	 & -2.04 	& 1.144e-05 	 & -3.06 	& 3.589e-07 	 & -4.08 	& 1.095e-08 	 & -5.13  \\ 
$2^{-6}$ 	 & 2.263e-02 	 & -0.99 	& 1.454e-03 	 & -2.00 	& 9.160e-05 	 & -3.00 	& 5.681e-06 	 & -3.98 	& 3.438e-07 	 & -4.97  \\  
$2^{-3}$ 	 & 2.7e-03 	 & -1.10 	& 2.1e-05 	 & -2.19 	& 1.6e-07 	 & -3.28 	& 1.2e-09 	 & -4.35 	& 8.0e-12 	 & -5.59  \\ 
$2^{-4}$ 	 & 5.7e-03 	 & -1.04 	& 8.8e-05 	 & -2.09 	& 1.4e-06 	 & -3.14 	& 2.1e-08 	 & -4.18  	& 3.1e-10 	 & -5.29  \\ 
$2^{-5}$ 	 & 1.1e-02 	 & -1.01 	& 3.6e-04 	 & -2.04 	& 1.1e-05 	 & -3.06 	& 3.6e-07 	 & -4.08 	& 1.1e-08 	 & -5.13  \\ 
$2^{-6}$ 	 & 2.3e-02 	 & -0.99 	& 1.5e-03 	 & -2.00 	& 9.2e-05 	 & -3.00 	& 5.7e-06 	 & -3.98 	& 3.4e-07 	 & -4.97  \\  
   \hline		
    \end{tabular}                     
	\end{footnotesize}
	\caption{Global truncation error. Error rates varying $\param$ with fixed $\delt = 10^{-3}$ for the equation $u_t = -u$. Calculations for orders $r = 4, 5$ required 64 digits of accuracy. Errors for $\param = 2^{-2}$ are not shown, but used to compute rates for the values $\param = 2^{-3}$.}
	\label{GTE_Table1}
\end{table} 
 
  
\begin{table}[htb!]
\centering
\begin{small}
\begin{tabular}{ | @{~}c@{~} ||@{~} c@{~} |@{~} c@{~} |@{~} c@{~} |@{~} c@{~} |@{~} c@{~} |}
\hline
 $\param$ & $r = 1$ & $r = 2$ & $r = 3$ & $r = 4$ & $r = 5$ \\
\hline  
$2^{0}$ 	 & 3.400e-02 	 & 5.047e-03 	 & 8.545e-04 	 & 1.509e-04 	 & 2.704e-05 	 \\ 
$2^{-1}$ 	 & 5.102e-02 	 & 7.967e-03 	 & 1.278e-03 	 & 2.043e-04 	 & 3.239e-05 	\\ 
$2^{-2}$ 	 & 5.903e-02 	 & 9.766e-03 	 & 1.573e-03 	 & 2.404e-04 	 & 3.480e-05 	\\   
$2^{-3}$ 	 & 6.291e-02 	 & 1.069e-02 	 & 1.728e-03 	 & 2.587e-04 	 & 3.584e-05 	\\ 
$2^{-4}$ 	 & 6.482e-02 	 & 1.116e-02 	 & 1.804e-03 	 & 2.673e-04 	 & 3.618e-05 	\\ 
$2^{-5}$ 	 & 6.577e-02 	 & 1.139e-02 	 & 1.842e-03 	 & 2.713e-04 	 & 3.629e-05 	\\ 
$2^{-6}$ 	 & 6.625e-02 	 & 1.150e-02 	 & 1.860e-03 	 & 2.732e-04 	 & 3.753e-05 	\\ 
$2^{-7}$ 	 & 6.648e-02 	 & 1.156e-02 	 & 1.870e-03 	 & 2.742e-04 	 & 3.592e-05 \\ 
$2^{-8}$ 	 & 6.660e-02 	 & 1.159e-02 	 & 1.874e-03 	 & 2.718e-04 	 & {\bf 3.634e-05} 	  \\ 
$2^{-9}$ 	 & 6.666e-02 	 & 1.160e-02 	 & 1.877e-03 	 & 2.818e-04 	 & {\bf 3.634e-05}  \\ 
$2^{-10}$ 	 & 6.669e-02 	 & 1.161e-02 	 & 1.878e-03 	 & {\bf 2.750e-04 } 	 & {\bf 3.635e-05} \\ 
	\hline		
    \end{tabular}                     
    \end{small}
	\caption{Global truncation error fixing $\delt = \frac{1}{5}\param$ for 
	the equation $u_t = -u$. Note that the overall error decreased by 
	$\sim \frac{1}{5}$ with each order, and remains constant with a decrease 
	in $\param$. Numbers in bold required 64 digits of accuracy in the 
	computation.}
	\label{GTE_Table2}
\end{table}


\section{Systematic study of second order \imex coefficients}\label{Supp_roots}

The purpose of this section is to show systematically the following necessary 
condition for large regions of unconditional stability in second order \imex 
LMMs: {\bf The roots of $c(z)$ must become close to $1$}.
An \imex scheme of order $r$ with 
$s = r$-steps is characterized by $r$ parameters.  The proposed coefficients
in Definition~\ref{NewImExCoeff} only exploit a one-parameter family of 
\imex coefficients.  
Instead, here we examine \imex coefficients that arise from 
a polynomial $c(z)$ with complex roots. 

Schemes with $r = s =  2$ are characterized by 
\begin{align*}
	c(z) &= (z - \sigma)(z - \bar{\sigma}) = z^2 - 
	(\sigma + \bar{\sigma} ) z + |\sigma|^2. \\
	b(z) &= c(z) - (z-1)^2 = [2- (\sigma + \bar{\sigma} )] 
	z + |\sigma|^2 - 1 = Az - B,
\end{align*}
where $A = 2- (\sigma + \bar{\sigma} )$, $B = 1- |\sigma|^2$.  Note that 
the roots of $c(z)$ must satisfy $|\sigma| \leq 1$ for unconditional 
stability and hence $A, B$ are real and $0 \leq A \leq 4$, $0 \leq B \leq 1$.

The boundary of the limiting stability region $\mathcal{D}_{-\infty}$ 
takes the form:
\[ 
	\frac{c(z)}{b(z)} =  \frac{z^2 - A z + B -1 }{Az - B}, 
	\quad \textrm{such that } |z| = 1. 
\]

Note that the numerator $|c(z)| \leq 1 + 4 = 5$ is bounded.  Therefore for 
$\mathcal{D}_{-\infty}$ to be large, the denominator $b(z)$ must become small.  
Examining $b(z)$ we have
\begin{align*}
	|b(z)|^2 &= (Az - B)(A\bar{z} - B) 
	= A^2 + B^2 - 2 AB \; \textrm{Re } z, \quad \quad -1 \leq \textrm{Re } 
	z \leq 1 \\
	&\geq A^2 + B^2 - 2 AB = (A-B)^2 \geq 0.
\end{align*}

Therefore the denominator $b(z)$ is bounded away from $0$, unless 
$A \approx B$. Here
\[ (A - B) = (\sigma - 1)(\bar{\sigma} - 1).\]
Therefore, a necessary condition for a large stability region is to have the 
roots $\sigma = 1 - \delta$ for some arbitrary complex 
$\delta \in \mathbb{C}$ with $|\delta| \ll 1$.  The polynomials $b(z), c(z)$ 
are then
\[ 
b(z) = 2(\textrm{Re}\;\delta)\; z - 2\textrm{Re}\;\delta - |\delta|^2, 
\quad
c(z) = (z - 1 + \delta) (z - 1 + \bar{\delta}).
\]

Computing the implicit GTE error constant in equation (\ref{ErrConstants})
yields
\[
C_{I,r} = |\delta|^{-2}\Big( \frac{-74}{6} - 
\frac{35}{6}|\delta|^2 + \frac{563}{36} \textrm{Re }\delta\Big) 
= \mathcal{O}(|\delta|^{-2}), \quad |\delta| \rightarrow 0.
\]
Therefore, large regions
of unconditional stability are accompanied by a decrease in the error 
constant (when applied to arbitrary general initial data).

One can also examine the case where the polynomial $c(z)$ 
has real, but unequal, roots.  


\clearpage

{\small
\bibliographystyle{siam}
\bibliography{../references_complete}

\begin{thebibliography}{10}

\bibitem{AbdulleMedovikov2001}
{\sc A.~Abdulle and A.~A. Medovikov}, {\em Second order {C}hebyshev methods
  based on orthogonal polynomials}, Numer. Math., 90 (2001), pp.~1--18.

\bibitem{Ahlfors1979}
{\sc L.~Ahlfors}, {\em Complex analysis}, McGraw-Hill, Inc., third~ed., 1979.

\bibitem{Akrivis2013}
{\sc G.~Akrivis}, {\em Implicit-explicit multistep methods for nonlinear
  parabolic equations}, Mathematics of Computation, 82 (2012), pp.~45--68.

\bibitem{AkrivisCrouzeixMakridakis1998}
{\sc G.~Akrivis, M.~Crouzeix, and C.~Makridakis}, {\em Implicit-explicit
  multistep finite element methods for nonlinear parabolic problems},
  Mathematics of Computation, 67 (1998), pp.~457--477.

\bibitem{AkrivisCrouzeixMakridakis1999}
\leavevmode\vrule height 2pt depth -1.6pt width 23pt, {\em Implicit-explicit
  multistep methods for quasilinear parabolic equations}, Numer. Math, 82
  (1999), pp.~521--541.

\bibitem{AkrivisKarakatsani2003}
{\sc G.~Akrivis and F.~Karakatsani}, {\em Modified implicit-explicit {BDF}
  methods for nonlinear parabolic equations}, BIT Numerical Mathematics, 43
  (2003), pp.~467--483.

\bibitem{AnitescuLaytonPahlevani2004}
{\sc M.~Anitescu, W.~Layton, and F.~Pahlevani}, {\em Implicit for local
  effects, explicit for nonlocal is unconditionally stable}, ETNA, 18 (2004),
  pp.~174--187.

\bibitem{AscherRuuthWetton1995}
{\sc U.~Ascher, S.~J. Ruuth, and B.~Wetton}, {\em Implicit-explicit methods for
  time dependent partial differential equations}, SIAM J. Numer. Anal., 32
  (1995), pp.~797--823.

\bibitem{BadalassiCenicerosBanerjee2003}
{\sc V.~Badalassi, H.~Ceniceros, and S.~Banerjee}, {\em Computation of
  multiphase systems with phase field models}, J. Comput. Phys., 190 (2003),
  pp.~371--397.

\bibitem{BertozziJuLu2011}
{\sc A.~Bertozzi, N.~Ju, and J.-W. Lu}, {\em A biharmonic-modified forward time
  stepping method for fourth order nonlinear diffusion equations}, Discrete and
  continuous dynamical systems, 29 (2011), pp.~1367--1391.

\bibitem{CahnHilliard1958}
{\sc J.~Cahn and J.~Hilliard}, {\em Free energy of a nonuniform system {I}.
  interfacial free energy}, J. Chem. Phys., 28 (1958), pp.~258--267.

\bibitem{Ceniceros2002}
{\sc H.~Ceniceros}, {\em A semi-implicit moving mesh method for the focusing
  nonlinear {S}chroedinger equation}, Comm. on Pure and Appl. Anal., 1 (2002),
  pp.~1--14.

\bibitem{ChristliebJonesPromislowWettonWilloughby2014}
{\sc A.~Christlieb, J.~Jones, K.~Promislow, B.~Wetton, and M.~Willoughby}, {\em
  High accuracy solutions to energy gradient flows from material science
  models}, J. Comput. Phys., 257 (2014), pp.~193--215.

\bibitem{ConcusGolub1973}
{\sc P.~Concus and G.~H. Golub}, {\em Use of fast direct methods for the
  efficient numerical solution of nonseparable elliptic equations}, SIAM J.
  Numer. Anal., 10 (1973), pp.~1103--1120.

\bibitem{Crouzeix1980}
{\sc M.~Crouzeix}, {\em Une m\'{e}thode multipas implicite-explicite pour
  l'approximation des \'{e}quations d'\'{e}volution paraboliques}, Numer. Math,
  35 (1980), pp.~257--276.

\bibitem{DouglasDupont1970b}
{\sc J.~Douglas and T.~Dupont}, {\em Alternating-direction {G}alerkin methods
  on rectangles}, in Numerical Solution of Partial Differential Equations,
  B.~Hubbard, ed., vol.~II, College Park, Md., 1971, SYNSPADE-1970, Univ. of
  Maryland, Academic Press, New York, pp.~133--213.

\bibitem{DriscollHaleTrefethen2014}
{\sc T.~A. Driscoll, N.~Hale, and L.~N. Trefethen}, {\em Chebfun guide}, 2014.
\newblock Code title: Field of values and numerical abscissa.

\bibitem{ElseyWirth2013}
{\sc M.~Elsey and B.~Wirth}, {\em A simple and efficient scheme for phase field
  crystal simulation}, M2AN, 47 (2013), pp.~1413--1432.

\bibitem{Eyre1998}
{\sc D.~Eyre}, {\em Unconditionally gradient stable time marching the
  {C}ahn-{H}illiard equation}, in Computational and Mathematical Models of
  Microstructural Evolution, J.~W. Bullard, R.~Kalia, M.~Stoneham, and L.~Chen,
  eds., vol.~53, Warrendale, PA, USA, 1998, Materials Research Society,
  pp.~1686--1712.

\bibitem{FrankHundsdorferVerwer1997}
{\sc J.~Frank, W.~Hundsdorfer, and J.~Verwer}, {\em On the stability of {IMEX}
  {LM} methods}, Appl. Numer. Math., 25 (1997), pp.~193--–205.

\bibitem{Glasner2003}
{\sc K.~Glasner}, {\em A diffuse interface approach to {H}ele-{S}haw flow},
  Nonlinearity, 16 (2003), pp.~49--66.

\bibitem{GlasnerOrizaga2016}
{\sc K.~Glasner and S.~Orizaga}, {\em Improving the accuracy of convexity
  splitting methods for gradient flow equations}, J. Comput. Phys., 315 (2016),
  pp.~52--64.

\bibitem{GottliebOrszag1977}
{\sc D.~Gottlieb and B.~Orszag}, {\em Numerical Analysis of Spectral Methods},
  CBMS-NSF Regional Conference Series in Applied Mathematics, SIAM Press, 1977.

\bibitem{GreengardRokhlin1987}
{\sc L.~Greengard and V.~Rokhlin}, {\em A fast algorithm for particle
  simulations}, J. Comput. Phys., 73 (1987), pp.~325--348.

\bibitem{GuanLowengrubWangWise2014}
{\sc Z.~Guan, J.~Lowengrub, C.~Wang, and S.~Wise}, {\em Second-order convex
  splitting schemes for periodic nonlocal {C}ahn-{H}illiard and {A}llen-{C}ahn
  equations}, J. Comput. Phys., 277 (2014), pp.~48--71.

\bibitem{HairerNorsettWanner1987}
{\sc E.~Hairer, S.~P. N{\o}rsett, and G.~Wanner}, {\em Solving ordinary
  differential equations {I}: Nonstiff problems}, Springer-Verlag, Berlin,
  second revised edition~ed., 1987.

\bibitem{WannerHairer1991}
{\sc E.~Hairer and G.~Wanner}, {\em Solving ordinary differential equations
  {II}: Stiff and Differential-Algebraic Problems}, vol.~1, Springer-Verlag,
  Berlin, 1991.

\bibitem{HornJohnson1991}
{\sc A.~Horn and C.~Johnson}, {\em Topics in Matrix analysis}, Cambridge
  University Press, 1991.

\bibitem{HundsdorferVerwer2003}
{\sc W.~Hundsdorfer and J.~Verwer}, {\em Numerical Solution of Time-Dependent
  Advection-Diffusion-Reaction Equations}, Springer Series in Comput. Math. 33,
  Springer, 2003.

\bibitem{JeltschNevanlinna1981}
{\sc R.~Jeltsch and O.~Nevanlinna}, {\em Stabiliity of explicit time
  discretizations for solving initial value problems}, Numer. Math., 37 (1981),
  pp.~61--91.

\bibitem{JeltschNevanlinna1982}
\leavevmode\vrule height 2pt depth -1.6pt width 23pt, {\em Stability and
  accuracy of time discretizations for initial value problems}, Numer. Math.,
  40 (1982), pp.~245--296.

\bibitem{Johnson1978}
{\sc C.~R. Johnson}, {\em Numerical determination of the field of values of a
  general complex matrix}, SIAM J. Numer. Anal., 15 (1978), pp.~595--602.

\bibitem{JohnstonLiu2004}
{\sc H.~Johnston and J.-G. Liu}, {\em Accurate, stable and efficient
  {N}avier-{S}tokes solvers based on explicit treatment of the pressure term},
  J. Comput. Phys., 199 (2004), pp.~221--259.

\bibitem{JuZhangZhuDu2014}
{\sc L.~Ju, J.~Zhang, L.~Zhu, and Q.~Du}, {\em Fast explicit integration factor
  methods for semilinear parabolic equations}, J. Sci. Comput., 62 (2015),
  pp.~431--455.

\bibitem{KarniadakisIsraeliOrszag1991}
{\sc G.~Karniadakis, M.~Israeli, and S.~A. Orszag}, {\em High-order splitting
  methods for the incompressible {N}avier-{S}tokes equations}, J. Comput.
  Phys., 97 (1991), pp.~414--443.

\bibitem{KimMoin1985}
{\sc J.~Kim and P.~Moin}, {\em Application of a fractional step method to
  incompressible {N}avier-{S}tokes equations}, J. Comput. Phys., 59 (1985),
  pp.~308--323.

\bibitem{Koto2009}
{\sc T.~Koto}, {\em Stability of implicit-explicit linear multistep methods for
  ordinary and delay differential equations}, Front. Math. China, 4 (2009),
  pp.~113--129.

\bibitem{LaytonTrenchea2012}
{\sc W.~Layton and C.~Trenchea}, {\em Stability of two {IMEX} methods, {CNLF}
  and {BDF2}-{AB2}, for uncoupling systems of evolution equations}, Appl.
  Numer. Math., 62 (2012), pp.~112–--120.

\bibitem{LeVeque2007}
{\sc R.~J. LeVeque}, {\em Finite difference methods for ordinary and partial
  differential equations: {S}teady-state and time-dependent problems}, Society
  for Industrial and Applied Mathematics, first~ed., 2007.

\bibitem{LiuLiuPego2007}
{\sc J.-G. Liu, J.~Liu, and R.~L. Pego}, {\em Stability and convergence of
  efficient {N}avier-{S}tokes solvers via a commutator estimate}, Comm. Pure
  Appl. Math., 60 (2007), pp.~1443--1487.

\bibitem{LiuLiuPego2010}
\leavevmode\vrule height 2pt depth -1.6pt width 23pt, {\em {Stable and accurate
  pressure approximation for unsteady incompressible viscous flow}}, J. Comput.
  Phys., 229 (2010), pp.~3428--3453.

\bibitem{MilewskiTabak1999}
{\sc P.~A. Milewski and E.~G. Tabak}, {\em A pseudo-spectral algorithm for the
  solution of nonlinear wave equations}, SIAM J. Sci. Comput., 21 (1999),
  pp.~1102--1114.

\bibitem{Minion2003}
{\sc M.~L. Minion}, {\em Semi-implicit spectral deferred correction methods for
  ordinary differential equations}, Commun. Math Sci., 1 (2003), pp.~471--500.

\bibitem{ShengWangDuWangLiuChen2010}
{\sc G.~Sheng, T.~Wang, Q.~Du, K.~Wang, Z.~Liu, and L.~Q. Chen}, {\em
  Coarsening kinetics of a two phase mixture with highly disparate diffusion
  mobility}, Commun. Comput. Phys., 8 (2010), pp.~249--264.

\bibitem{ShirokoffRosales2010}
{\sc D.~Shirokoff and R.~R. Rosales}, {\em An efficient method for the
  incompressible {N}avier-{S}tokes equations on irregular domains with no-slip
  boundary conditions, high order up to the boundary}, J. Comput. Phys., 230
  (2011), pp.~8619--8646.

\bibitem{Smereka2003}
{\sc P.~Smereka}, {\em Semi-implicit level set methods for curvature and
  surface diffusion motion}, J. Sci. Comput., 19 (2003), pp.~439--456.

\bibitem{Trefethen2000}
{\sc L.~N. Trefethen}, {\em Spectral Methods in {MATLAB}}, SIAM, Philadelphia,
  2000.

\bibitem{TrefethenBau1997}
{\sc L.~N. Trefethen and D.~Bau}, {\em Numerical Linear Algebra}, SIAM,
  Philadelphia, 2000.

\bibitem{Trenchea2016}
{\sc C.~Trenchea}, {\em Second order implicit for local effects and explicit
  for nonlocal effects is unconditionally stable}, Romai J., 12 (2016),
  pp.~163--178.

\bibitem{Varah1980}
{\sc J.~M. Varah}, {\em Stability restrictions on second order, three level
  finite difference schemes for parabolic equations}, SIAM J. Numer. Anal., 17
  (1980), pp.~300--309.

\bibitem{XuLiWu2016}
{\sc J.~Xu, Y.~Li, and S.~Wu}, {\em Convex splitting schemes interpreted as
  fully implicit schemes in disguise for phase field modeling}, 2016.
\newblock arXiv:1604.05402.

\bibitem{YanChenWangWise2015}
{\sc Y.~Yan, W.~Chen, C.~Wang, and S.~Wise}, {\em A second-order energy stable
  {BDF} numerical scheme for the {C}ahn-{H}illiard equation}, 2015.

\end{thebibliography}
}

\end{document}